 \documentclass[a4paper]{amsart}
\usepackage[latin1]{inputenc}
\usepackage{amsmath}
\usepackage{amsfonts}
\usepackage{amssymb}
\usepackage{amsthm}
\usepackage{rotating}
\usepackage[usenames, dvipsnames]{color}
\usepackage{tikz}
\usepackage[all]{xy}
\usepackage{enumitem}
\usepackage[normalem]{ulem}
\usepackage{soul}
\numberwithin{table}{section}

\definecolor{liens}{rgb}{1,0,0}
\definecolor{test}{rgb}{1,0,1}
\definecolor{armygreen}{rgb}{0.2, .8, 0.6}
\definecolor{burntumber}{rgb}{0.54, 0.2, 0.14}
\definecolor{dartmouthgreen}{rgb}{0, 0.7, 0.8}
\definecolor{deeppink}{rgb}{1.0, 0.08, 0.58}
\usepackage[colorlinks=true, linkcolor=blue, 
hyperfootnotes=true,citecolor=blue,urlcolor=black]{hyperref}
\newcommand{\mfs}{\color{black}}
\newcommand{\mfss}{\color{black}}

\newtheorem{thm}{Theorem}[section]

\newtheorem{lemma}[thm]{Lemma}
\newtheorem{lem}[thm]{Lemma}
\newtheorem{prop}[thm]{Proposition}

\newtheorem{defn}[thm]{Definition}
\newtheorem{define}[thm]{Definition} 

\newenvironment{prf}[1]{\trivlist
\item[\hskip \labelsep{\bf #1.\hspace*{.3em}}]}{~\hspace{\fill}~$\square$\endtrivlist}

\newtheorem{rmk}[thm]{Remark}
\newtheorem{rem}[thm]{Remark} 

\theoremstyle{remark}
\newtheorem{exa}[thm]{Example}

\numberwithin{equation}{section}

\def\Res{\operatorname{Res}}

\def\Z{\mathbb{Z}}
\def\C{\mathbb{C}}
\def\R{\mathbb{R}}
\def\Q{\mathbb{Q}}

\def\P1{\mathbb{P}^{1}}

\def\beq{\begin{equation}}
\def\eeq{\end{equation}}

\def\Etproj{\overline{E_t}}

\def\Scal{\mathcal{S}}
\def\Pcal{\mathcal{P}}
\def\Qcal{\mathcal{Q}}
\def\Ocal{\mathcal{O}}

\def\CX{{\mathbb C}}

\def\QX{{\mathbb Q}}

\def\NX{{\mathbb N}}
\def\ZX{{\mathbb Z}}

\def\PX{{\mathbb P}}

\def\ores{{\rm ores}}

\def\calD{{\mathcal{D}}}

\def\calL{{\mathcal{L}}}

\def\ty{{\tilde{y}}}

\def\hh{{\widehat{h}}}
\def\ita{{\iota_1}}
\def\itb{{\iota_2}}
\def\htilde{{\tilde{h}}}
\def\gtilde{{\tilde{g}}}
\def\hhat{{\hat{h}}}

\newcommand{\walk}[1]{w_{\textrm{#1}}}

\def\tEtproj{{\widetilde{E_t}}}

\newcommand{\charl}[1]{\textcolor{blue}{ #1}}
\def\tx{\widetilde{x}}
\def\ty{\widetilde{y}}
\def\a{\alpha}
\def\b{\beta}
\def\c{\gamma}

\usepackage{tikz}
\usetikzlibrary{cd,arrows,positioning,calc}
\usetikzlibrary{decorations.markings}
\tikzset{>=stealth}
\tikzcdset{arrow style=tikz}
\tikzset{map/.style={row sep=0em, column sep=0em}}
\tikzset{dot/.style={circle,fill=black,minimum size=5pt,inner sep=0pt}}
\tikzstyle directed=[postaction={decorate,decoration={markings, mark=at position .5 with {\arrow{stealth}}}}] 
\usepackage{tkz-euclide,amsmath}
\usetikzlibrary{calc,intersections,through,backgrounds}
\usepackage{subfig}

\begin{document}

\title{On differentially algebraic generating series for walks in the quarter plane}

\author{Charlotte Hardouin}
\address{Universit\'e Paul Sabatier - Institut de Math\'ematiques de Toulouse, 118 route de Narbonne, 31062 Toulouse.}
\email{hardouin@math.univ-toulouse.fr}
\author{Michael F. Singer }
\address{Department of Mathematics, North Carolina State University,
Box 8205, Raleigh, NC 27695-8205, USA}
\email{singer@math.ncsu.edu}


\thanks{This project has received funding from the European Research Council (ERC) under the European Union's Horizon 2020 research and innovation programme under the Grant Agreement No 648132. The {first} author would like to thank the ANR-11-LABX-0040-CIMI within
the program ANR-11-IDEX-0002-0 for its partial support. The {first} author's work is also supported by ANR De rerum natura (ANR-19-CE40-0018) . The work of the second author was partially supported by a grant from the Simons Foundation (\#349357, Michael Singer). {\mfss Both authors would like to thank the Mathematical Sciences Research Institute for sponsoring a visit during which significant progress on this paer was made.}
 The first author would like to thank Marcello Bernardara, Thomas Dedieu and  Stephane Lamy  for many discussions and references on elliptic surfaces.
 }

\date{October 2, 2020}

\bibliographystyle{amsalpha}

\begin{abstract} We refine necessary and sufficient conditions for the generating series of a weighted model of a quarter plane walk to be differentially algebraic. In addition, we give algorithms  based on the theory of Mordell-Weil lattices,  that, for each weighted model, yield polynomial conditions on the weights  determining this property of  the associated generating series.
\end{abstract}
\maketitle
\setcounter{tocdepth}{1}
\tableofcontents
\pagestyle{myheadings}
\markboth{C.~HARDOUIN, M.F.~SINGER}{DIFFERENTIAL ALGEBRAIC GENERATING SERIES}


\section{Introduction}
The enumeration of planar lattice walks confined to the {\mfs first} quadrant has {\mfs attracted a  considerable amount}  of interest over the past fifteen years.  For the lattice $\Z^2$, a  lattice path model is comprised of a finite set $\mathcal{D}$ of lattice vectors called the \emph{step set} together with a starting point $P \in \Z^2$. The combinatorial question boils down to the count $q_{i,j}(n)$ of $n$-step walks, i.e., of  polygonal chains, that remain in the first  quadrant,  starting from $P$, 
ending at  $(i,j)$ and  consisting of $n$ oriented  line segments whose associated translation vectors belong to $\mathcal{D}$. This question is ubiquitous since lattice walks encode several classes of mathematical objects, in discrete mathematics (permutations, trees, planar maps), in probability theory (lucky games, sums of discrete random variables), statistics (non-parametric tests). We refer to the introduction of \cite{BanderierFlajolet} for more details on these applications {\mfs as well as \cite{KH2010} for applications in other scientific areas.}

{\mfs Many algebraic  and analytic} properties of the  combinatorial sequence of a  lattice  walk  are  embodied in the algebraic nature of the associated generating function. For instance, 
for the lattice $\Z^2$, the linear recurrences satisfied by the sequence $(q_{i,j}(n))_{i,j,n}$ corresponds to the fact that the generating function \begin{equation} \label{genfnc} Q(x,y,t)= \sum_{i,j,n \geq 0} q_{i,j}(n) x^iy^jt^n \end{equation} is $D$-finite, that is, satisfies linear differential equations in the derivation with respect to $x,y$ and $t$. This correspondence yields a classification of the generating series {\mfs as to:} algebraic functions over $\Q(x,y,t)$, $D$-finite functions, differentially algebraic  functions ({\mfs those satisfying} a polynomial relations with their derivatives) and differentially transcendental functions. Recently, the works of many authors  led to a complete classification {\mfs of generating series associated to} lattice walks  with small steps,  that is, with step set $\calD \subset \{-1,0,1\}^2 \backslash \{(0,0)\}$. These works  combine a wide 
variety of technics: singularity analysis via the Kernel Method, probabilistic method, guess and proof strategies and Galois theory of functional equations. Many researchers have contributed answers to these questions and our brief exposition below does not do justice to these contributions.  Nonetheless, since detailed descriptions of these various contributions exist elsewhere (see for example \cite{BBMR17, DHRS, DHRS19}) we will limit ourselves to a brief summary. 

Of the $2^8-1$ possible choices of step sets it is shown in \cite{BMM} that taking symetries into account and eliminating trivial sets, one need only consider  $79$ of these \charl{models}.  Of these, $23$   models have $D$-finite (in all variables)  generating series (\cite{BMM, BostanKauersTheCompleteGenerating}) of which $4$ are algebraic. The remaining $56$  models were shown to have non-$D$-finite generating series with respect to various variables in  \cite{KurkRasch,MR09,MelcMish, BRS}. In \cite{BBMR17, DHRS, DHRS19, DreyRasc}, the more general question of differential transcendence is addressed.  In \cite{BBMR17}, the authors give new uniform proofs of the $4$  algebraic cases  and also show that $9$ (see Figure~\ref{figure:ex})  of the $56$  non-$D$-finite models in fact have differentially algebraic generating functions. Using criteria from the Galois theory of difference equations, the authors  of \cite{DHRS,DHRS19} show that $47$ of the $56$ non-$D$-finite  models  have differentially transcendental generating functions  and reproved the fact from \cite{BBMR17} that the remaining  $9$  are differentially algebraic. (Figure~\ref{figure:ex} below reproduces Figure 2 of \cite{DHRS} with a table comparing the notations of \cite[Table 2]{BBMR17} and \cite[Figure 2]{DHRS}).
\begin{figure}[h!]
$$
\underset{\walk{IIB.1}}{\begin{tikzpicture}[scale=.4, baseline=(current bounding box.center)]
\foreach \x in {-1,0,1} \foreach \y in {-1,0,1} \fill(\x,\y) circle[radius=2pt];
\draw[thick,->](0,0)--(0,1);
\draw[thick,->](0,0)--(1,0);
\draw[thick,->](0,0)--(-1,-1);
\draw[thick,->](0,0)--(0,-1);
\end{tikzpicture}}
\quad 
\underset{\walk{IIB.2}}{\begin{tikzpicture}[scale=.4, baseline=(current bounding box.center)]
\foreach \x in {-1,0,1} \foreach \y in {-1,0,1} \fill(\x,\y) circle[radius=2pt];
\draw[thick,->](0,0)--(0,1);
\draw[thick,->](0,0)--(1,0);
\draw[thick,->](0,0)--(-1,-1);
\draw[thick,->](0,0)--(1,-1);
\end{tikzpicture}
}
\quad
\underset{\walk{IIC.1}}{\begin{tikzpicture}[scale=.4, baseline=(current bounding box.center)]
\foreach \x in {-1,0,1} \foreach \y in {-1,0,1} \fill(\x,\y) circle[radius=2pt];
\draw[thick,->](0,0)--(0,1);
\draw[thick,->](0,0)--(1,1);
\draw[thick,->](0,0)--(-1,0);
\draw[thick,->](0,0)--(0,-1);
\end{tikzpicture}
}
\quad
\underset{\walk{IIB.3}}{\begin{tikzpicture}[scale=.4, baseline=(current bounding box.center)]
\foreach \x in {-1,0,1} \foreach \y in {-1,0,1} \fill(\x,\y) circle[radius=2pt];
\draw[thick,->](0,0)--(0,1);
\draw[thick,->](0,0)--(-1,0);
\draw[thick,->](0,0)--(1,0);
\draw[thick,->](0,0)--(1,-1);
\end{tikzpicture}}
\quad
\underset{\walk{IIC.4}}{\begin{tikzpicture}[scale=.4, baseline=(current bounding box.center)]
\foreach \x in {-1,0,1} \foreach \y in {-1,0,1} \fill(\x,\y) circle[radius=2pt];
\draw[thick,->](0,0)--(0,1);
\draw[thick,->](0,0)--(1,1);
\draw[thick,->](0,0)--(-1,0);
\draw[thick,->](0,0)--(1,0);
\draw[thick,->](0,0)--(-1,-1);
\end{tikzpicture}}
\quad
\underset{\walk{IIC.2}}{\begin{tikzpicture}[scale=.4, baseline=(current bounding box.center)]
\foreach \x in {-1,0,1} \foreach \y in {-1,0,1} \fill(\x,\y) circle[radius=2pt];
\draw[thick,->](0,0)--(0,1);
\draw[thick,->](0,0)--(1,1);
\draw[thick,->](0,0)--(-1,0);
\draw[thick,->](0,0)--(-1,-1);
\draw[thick,->](0,0)--(0,-1);
\end{tikzpicture}}
\quad
\underset{\walk{IIB.6}}{\begin{tikzpicture}[scale=.4, baseline=(current bounding box.center)]
\foreach \x in {-1,0,1} \foreach \y in {-1,0,1} \fill(\x,\y) circle[radius=2pt];
\draw[thick,->](0,0)--(0,1);
\draw[thick,->](0,0)--(1,0);
\draw[thick,->](0,0)--(-1,-1);
\draw[thick,->](0,0)--(0,-1);
\draw[thick,->](0,0)--(1,-1);
\end{tikzpicture}}
\quad 
\underset{\walk{IIC.5}}{\begin{tikzpicture}[scale=.4, baseline=(current bounding box.center)]
\foreach \x in {-1,0,1} \foreach \y in {-1,0,1} \fill(\x,\y) circle[radius=2pt];
\draw[thick,->](0,0)--(0,1);
\draw[thick,->](0,0)--(1,1);
\draw[thick,->](0,0)--(-1,0);
\draw[thick,->](0,0)--(1,0);
\draw[thick,->](0,0)--(0,-1);
\end{tikzpicture}}
\quad
\underset{\walk{IIB.7}}{\begin{tikzpicture}[scale=.4, baseline=(current bounding box.center)]
\foreach \x in {-1,0,1} \foreach \y in {-1,0,1} \fill(\x,\y) circle[radius=2pt];
\draw[thick,->](0,0)--(-1,1);
\draw[thick,->](0,0)--(0,1);
\draw[thick,->](0,0)--(1,0);
\draw[thick,->](0,0)--(0,-1);
\draw[thick,->](0,0)--(1,-1);
\end{tikzpicture}}
$$
\vspace{.2in}
$$
\begin{tabular}{|c|c|}
  \hline
\cite[Tab 2]{BBMR17} & \cite[Fig. 2]{DHRS} \\
  \hline
  1 & $\walk{IIB.1}$ (after $x\leftrightarrow y$)  \\
  2 & $\walk{IIB.2}$ (after $x\leftrightarrow y$)  \\
  3 & $\walk{IIC.1}$  \\
  4 & $\walk{IIB.3}$  \\
  5 & $\walk{IIC.4}$ \\
  6 & $\walk{IIC.2}$  \\
  7 & $\walk{IIB.6}$ (after $x\leftrightarrow y$)  \\
  8 & $\walk{IIC.5}$  \\
  9 & $\walk{IIB.7}$  \\
  \hline
\end{tabular}
$$
\caption{The 9 non-$D$-finite  models that have $D$-algebraic generating series together with a table comparing notations of \cite{BBMR17} and \cite{DHRS}}
\label{figure:ex}
\end{figure}

At the core of all these works, one finds two geometric objects: an algebraic curve defined over $\Q(t)$ called the \emph{kernel curve} of genus $0$ or $1$ and  a group of automorphisms of the curve called \emph{the group of the walk}. Though the finiteness of the group had been clearly 
related to the $D$-finiteness of the generating function, no combinatorial as well as geometric criteria had been proposed
to characterize the differential  algebraicity of the generating function.  \cite{DHRS} proposed  a criteria  based on the computation of residues of elliptic functions  and \cite{BBMR17} discovered the more algebraic notion of \emph{decoupled model} by a case-by-case analysis of the nine models of Figure~\ref{figure:ex}.The  notion of decoupled model allowed the authors of  \cite{BBMR17} to give an explicit expression of the generating function, which led to an explicit differential algebraic equation.

 {\mfs The study of walks with multiple steps or weighted walks (that is, lattice walks whose steps have been endowed with weights) yielded a more fecund understanding of these criteria.}  When all these weights are equal, a rescaling allows one to consider them all equal to $1$ whence the terminology \emph{unweighted model} to denote now  the $2^8-1$ models  introduced in the above paragraphs. The need for a classification of weighted walks confined in the quadrant arose in the classification of three dimensional walks confined in the octant. As shown in \cite{BoBMKaMe-16}, some of these three dimensional models can be reduced by projection to two-dimensional models with weights. Similarly to unweighted models,  one attaches to a weighted model a kernel  curve of genus zero or one and a group of automorphisms of this curve. When the group is finite, \cite[Cor.43]{DreyRasc} proves that the generating function is $D$-finite. When  the kernel curve is  of genus zero, the generating function is differentially transcendental by \cite{DHRS19}. The case of a kernel curve of genus one remained  open untill now and only some partial cases were treated. In \cite{DreyRasc}, the authors  proved the differential transcendence of the generating function for some classes of walks. In  \cite{KauersYatchak} and \cite{CourtielMelczerMishnaRaschel}, the authors study families of weighted models with finite group and the algebraicity of their generating functions.

 In this paper, we focus on weighted models with small steps. For these models, we  unify the approaches of \cite{DHRS} and \cite{BBMR17} and to show that a weighted model is decoupled if and only if its generating function is differentially algebraic. Moreover we  translate  the combinatorial question of the differential algebraicity   of the generating function in the purely arithmetic question of the linear dependence of two given points of the Mordell-Weil group of the kernel curve. Previous works had considered the kernel curve as a fixed elliptic curve by choosing a value of $t$, even transcendental over $\Q$. The novelty of our strategy is that we allow $t$ to vary so that we work with a pencil of elliptic  curves or equivalently with a rational surface whose general fiber is the kernel curve.  Relying on the theory of Mordell-Weil lattices and their classification for rational elliptic surfaces (see for instance \cite{SchuttShiodaBook}), we  construct an algorithm which given a weighted model determines the polynomial relations between the weights that correspond to a differentially algebraic generating function. For instance, for  the weighted  model 
$$
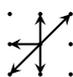
{\begin{tikzpicture}[scale=.4, baseline=(current bounding box.center)]
\foreach \x in {-1,0,1} \foreach \y in {-1,0,1} \fill(\x,\y) circle[radius=2pt];
\draw[thick,->](0,0)--(0,1);
\draw[thick,->](0,0)--(1,1);
\draw[thick,->](0,0)--(-1,0);
\draw[thick,->](0,0)--(-1,-1);
\draw[thick,->](0,0)--(0,-1);
\end{tikzpicture}}
$$
with nonzero weights $d_{1,1}, d_{0.-1}, d_{-1,-1}, d_{-1,0}, d_{0,1},d_{0,0}$,  the associated generating series is differentially algebraic if and only if 
$$ d_{0,1}d_{0,-1} - d_{1,1} d_{-1,-1} = 0.$$
 This relation is automatically satisfied when all the weights are equal to one so that the corresponding model $w_{IIC2}$ is one of the nine differentially algebraic models of Figure~\ref{figure:ex}. This shows that these nine cases are  {\mfs coincidences; they are just the only weighted models for which the weights equal to one satisfy the polynomial equations guaranteeing the differential algebraicity.}
 
 This geometric strategy  has therefore a combinatorial interest since it builds a bridge between the combinatorics of the walks and the combinatorics of the Mordell-Weil lattices. For walks in the first quadrant, the nature of the Mordell-Weil lattice is controlled by the relative position of the base points of the pencil of kernel curves.   This arithmetic point of view  might be also  well suited to attack the question of the specialization of the variable $t$ since it might be translated in terms of specialization of  independent points of the general fiber of an elliptic surface to linearly dependent points in a specialized fiber.
 
 The rest of the paper is organized as follows. In Section ~\ref{sec:ker} we review the notions of the kernel of a walk and the group of a walk. In Section ~\ref{certificate} we show that the criteria of  \cite{BBMR17} and \cite{DHRS}  are equivalent.  In Section~\ref{sec:rescrit} we simplify the latter criteria of \cite{DHRS}  by showing they are equivalent to showing that two poles of a certain function lie in the same orbit under an action already considered in \cite{DHRS}. Combining this with ideas from the theory of elliptic surfaces, we give, in Section~\ref{sec:Dalg} an algorithm  and some refinements which allow one to characterize in terms of polynomial relations those weights for which the generating series are  differentially  algebraic. {\mfs In Section~\ref{sec:algo}, we present some basic facts concerning the Kodaira-N\'eron Model of our family of elliptic curves, its Mordell-Weill lattice and N\'eron-Tate heights and present an algorithm which, once these are facts are accepted, reduces the computation of these polynomial conditions to  the calculation of an associated Weierstass equation and simple arithmetic. In Section~\ref{sec:refine}, we give a more detailed description of these objects and concepts, yielding a significant refinement of the algorithm.}  In Appendix~\ref{appendix:PR} we recall some facts concerning local parameters, poles and the notion of orbit residue introduced in \cite{DHRS}.

\section{Kernel curve and group of the walk}\label{sec:ker} From now on, we will fix a set of steps $\calD$ and weights $\{d_{i,j}\}$.We also  fix once and for all a value of $t$, transcendental over $\QX$ and occasionally suppress the symbol $t$ in  our notation. 
All studies concerning the behavior of the generating series (\ref{genfnc}) begin with the functional equation it satisfies (c.f.,~\cite{BMM}).  One first defines a Laurent polynomial called the {\it inventory} of the step set $\calD $
\begin{align}\label{inventory}
S(x,y) &:= \sum_{(i,j) \in \calD} d_{i,j} x^iy^j
\end{align} 
and a polynomial called the {\it kernel} of the walk
\begin{align}\label{kernel}
K(x,y,t) &:= xy(1-tS(x,y)).
\end{align} 
One then has that $Q(x,y,t)$ satisfies
\begin{align}\label{fnceqn}
K(x,y,t) Q(x,y,t) &= xy - F^1(x,t) - F^2(y,t) +td_{-1,-1}Q(0,0,t)
\end{align}
where 
\begin{align} \label{F1}
F^1(x,t)  &:= -K(x,0,t)Q(x,0,t) \ \ \ \ \   \mbox{and}  &F^2(y,t):= -K(0,y,t)Q(0,y,t).
\end{align}

\subsection{The Curve} The equation $K(x,y) = 0$ defines an affine curve $E_t$ in $\CX \times \CX$. As  in \cite{DHRS, DHRS19}, it is useful to consider a compactification $\Etproj$ of this curve in $\PX^1(\CX)\times\PX^1(\CX)$. This curve is defined by homogenizing each variable separately in $K(x,y)$, that is, 
\begin{defn} The {\rm kernel curve} associated to a quadrant model  is the curve 
\[ \Etproj = \{ ([x_0:x_1], [y_0:y_1] ) \in \PX^1(\CX)\times\PX^1(\CX) \ | \ \overline{K}(x_0,x_1,y_0,y_1,t) = 0\}\]

where $ \overline{K}(x_0,x_1,y_0,y_1,t)$ is the following bihomogeneous polynomial
\begin{equation}\label{eq:kernelwalk}
\overline{K}(x_0,x_1,y_0,y_1,t)={x_1^2y_1^2K(\frac{x_0}{x_1},\frac{y_0}{y_1},t)}= x_0x_1y_0y_1 -t \sum_{i,j=0}^2 d_{i-1,j-1} x_0^{i} x_1^{2-i}y_0^j y_1^{2-j}. 
 \end{equation}
 \end{defn}

  {The reducibility of $K(x,y,t)$  as an element of $\CX[x,y]$ can be expressed as a condition on  the set of steps of the model (see \cite[Lemma 2.3.2]{FIM} for $t=1$ and \cite[Proposition 1.2]{DHRS20}). The walks having reducible
kernel polynomials are called degenerate and their generating series is algebraic.  Thus, we will discard these cases  and  will assume that $K(x,y,t)$ is irreducible.}  In this case, the polynomial $K(x,y,t)$ has degree $2$ in each of its variables $x$ and $y$ and if $\Etproj$ is nonsingular it is of genus $1$, otherwise it has genus $0$.  The genus zero curves correspond to $28$ sets of steps  \cite[Cor. 2.6]{DHRS20}. Up to symmetry and discarding the sets of steps which do not enter the first quadrant, one can only focus on the five following set of steps 
$$\begin{tikzpicture}[scale=.2, baseline=(current bounding box.center)]
\foreach \x in {-1,0,1} \foreach \y in {-1,0,1} \fill(\x,\y) circle[radius=2pt];
\draw[thick,->](0,0)--(-1,1);
\draw[thick,->](0,0)--(0,1);
\draw[thick,->](0,0)--(1,-1);
\end{tikzpicture}
\quad 
 \begin{tikzpicture}[scale=.2, baseline=(current bounding box.center)]
\foreach \x in {-1,0,1} \foreach \y in {-1,0,1} \fill(\x,\y) circle[radius=2pt];
\draw[thick,->](0,0)--(-1,1);
\draw[thick,->](0,0)--(1,0);
\draw[thick,->](0,0)--(1,-1);
\end{tikzpicture}
\quad 
\begin{tikzpicture}[scale=.2, baseline=(current bounding box.center)]
\foreach \x in {-1,0,1} \foreach \y in {-1,0,1} \fill(\x,\y) circle[radius=2pt];
\draw[thick,->](0,0)--(-1,1);
\draw[thick,->](0,0)--(1,1);
\draw[thick,->](0,0)--(1,-1);
\end{tikzpicture}
\quad 
 \begin{tikzpicture}[scale=.2, baseline=(current bounding box.center)]
\foreach \x in {-1,0,1} \foreach \y in {-1,0,1} \fill(\x,\y) circle[radius=2pt];
\draw[thick,->](0,0)--(-1,1);
\draw[thick,->](0,0)--(0,1);
\draw[thick,->](0,0)--(1,1);
\draw[thick,->](0,0)--(1,-1);
\end{tikzpicture}
\quad 
\begin{tikzpicture}[scale=.2, baseline=(current bounding box.center)]
\foreach \x in {-1,0,1} \foreach \y in {-1,0,1} \fill(\x,\y) circle[radius=2pt];
\draw[thick,->](0,0)--(-1,1);
\draw[thick,->](0,0)--(1,0);
\draw[thick,->](0,0)--(1,1);
\draw[thick,->](0,0)--(1,-1);
\end{tikzpicture}
\quad 
\begin{tikzpicture}[scale=.2, baseline=(current bounding box.center)]
\foreach \x in {-1,0,1} \foreach \y in {-1,0,1} \fill(\x,\y) circle[radius=2pt];
\draw[thick,->](0,0)--(-1,1);
\draw[thick,->](0,0)--(0,1);
\draw[thick,->](0,0)--(1,1);
\draw[thick,->](0,0)--(1,0);
\draw[thick,->](0,0)--(1,-1);
\end{tikzpicture}
\quad 
\begin{tikzpicture}[scale=.2, baseline=(current bounding box.center)]
\foreach \x in {-1,0,1} \foreach \y in {-1,0,1} \fill(\x,\y) circle[radius=2pt];
\draw[thick,->](0,0)--(-1,1);
\draw[thick,->](0,0)--(0,1);
\draw[thick,->](0,0)--(1,0);
\draw[thick,->](0,0)--(1,-1);
\end{tikzpicture}$$
{The main result of \cite{DHRS19} is to show that the generating series of any  weighted model attached to one of the above set of steps is differentially transcendental. Thus in the whole paper, we will always assume that the model of our walk corresponds to a genus one curve, that is according to \cite[Cor.2.6]{DHRS20}, we will focus on the weighted models of the following set of steps.}

\begin{equation}\label{stepgenus1}\tag{G1}
\begin{array}{llllllllllllllll}
\begin{tikzpicture}[scale=.2, baseline=(current bounding box.center)]
\foreach \x in {-1,0,1} \foreach \y in {-1,0,1} \fill(\x,\y) circle[radius=2pt];
\draw[thick,->](0,0)--(1,1);
\draw[thick,->](0,0)--(-1,0);
\draw[thick,->](0,0)--(0,-1);
\end{tikzpicture}
& 
\begin{tikzpicture}[scale=.2, baseline=(current bounding box.center)]
\foreach \x in {-1,0,1} \foreach \y in {-1,0,1} \fill(\x,\y) circle[radius=2pt];
\draw[thick,->](0,0)--(0,1);
\draw[thick,->](0,0)--(1,0);
\draw[thick,->](0,0)--(-1,-1);
\end{tikzpicture} 
&
\begin{tikzpicture}[scale=.2, baseline=(current bounding box.center)]
\foreach \x in {-1,0,1} \foreach \y in {-1,0,1} \fill(\x,\y) circle[radius=2pt];
\draw[thick,->](0,0)--(0,1);
\draw[thick,->](0,0)--(1,1);
\draw[thick,->](0,0)--(-1,0);
\draw[thick,->](0,0)--(1,0);
\draw[thick,->](0,0)--(-1,-1);
\draw[thick,->](0,0)--(0,-1);
\end{tikzpicture}
&
\begin{tikzpicture}[scale=.2, baseline=(current bounding box.center)]
\foreach \x in {-1,0,1} \foreach \y in {-1,0,1} \fill(\x,\y) circle[radius=2pt];
\draw[thick,->](0,0)--(1,1);
\draw[thick,->](0,0)--(-1,-1);
\draw[thick,->](0,0)--(1,0);
\draw[thick,->](0,0)--(-1,0);
\end{tikzpicture}
&
\begin{tikzpicture}[scale=.2, baseline=(current bounding box.center)]
\foreach \x in {-1,0,1} \foreach \y in {-1,0,1} \fill(\x,\y) circle[radius=2pt];
\draw[thick,->](0,0)--(1,1);
\draw[thick,->](0,0)--(-1,-1);
\draw[thick,->](0,0)--(0,1);
\draw[thick,->](0,0)--(0,-1);
\end{tikzpicture}
&
\begin{tikzpicture}[scale=.2, baseline=(current bounding box.center)]
\foreach \x in {-1,0,1} \foreach \y in {-1,0,1} \fill(\x,\y) circle[radius=2pt];
\draw[thick,->](0,0)--(0,1);
\draw[thick,->](0,0)--(-1,0);
\draw[thick,->](0,0)--(1,0);
\draw[thick,->](0,0)--(0,-1);
\end{tikzpicture}
&
\begin{tikzpicture}[scale=.2, baseline=(current bounding box.center)]
\foreach \x in {-1,0,1} \foreach \y in {-1,0,1} \fill(\x,\y) circle[radius=2pt];
\draw[thick,->](0,0)--(-1,1);
\draw[thick,->](0,0)--(1,1);
\draw[thick,->](0,0)--(-1,-1);
\draw[thick,->](0,0)--(1,-1);
\end{tikzpicture}
&
\begin{tikzpicture}[scale=.2, baseline=(current bounding box.center)]
\foreach \x in {-1,0,1} \foreach \y in {-1,0,1} \fill(\x,\y) circle[radius=2pt];
\draw[thick,->](0,0)--(-1,1);
\draw[thick,->](0,0)--(1,0);
\draw[thick,->](0,0)--(1,1);
\draw[thick,->](0,0)--(-1,-1);
\draw[thick,->](0,0)--(-1,0);
\draw[thick,->](0,0)--(1,-1);
\end{tikzpicture}
&
\begin{tikzpicture}[scale=.2, baseline=(current bounding box.center)]
\foreach \x in {-1,0,1} \foreach \y in {-1,0,1} \fill(\x,\y) circle[radius=2pt];
\draw[thick,->](0,0)--(-1,1);
\draw[thick,->](0,0)--(0,1);
\draw[thick,->](0,0)--(1,1);
\draw[thick,->](0,0)--(-1,-1);
\draw[thick,->](0,0)--(0,-1);
\draw[thick,->](0,0)--(1,-1);
\end{tikzpicture}
&
\begin{tikzpicture}[scale=.2, baseline=(current bounding box.center)]
\foreach \x in {-1,0,1} \foreach \y in {-1,0,1} \fill(\x,\y) circle[radius=2pt];
\draw[thick,->](0,0)--(-1,1);
\draw[thick,->](0,0)--(0,1);
\draw[thick,->](0,0)--(1,1);
\draw[thick,->](0,0)--(-1,0);
\draw[thick,->](0,0)--(1,0);
\draw[thick,->](0,0)--(-1,-1);
\draw[thick,->](0,0)--(0,-1);
\draw[thick,->](0,0)--(1,-1);
\end{tikzpicture}
&
\begin{tikzpicture}[scale=.2, baseline=(current bounding box.center)]
\foreach \x in {-1,0,1} \foreach \y in {-1,0,1} \fill(\x,\y) circle[radius=2pt];
\draw[thick,->](0,0)--(-1,1);
\draw[thick,->](0,0)--(1,1);
\draw[thick,->](0,0)--(0,-1);
\end{tikzpicture}
&
\begin{tikzpicture}[scale=.2, baseline=(current bounding box.center)]
\foreach \x in {-1,0,1} \foreach \y in {-1,0,1} \fill(\x,\y) circle[radius=2pt];
\draw[thick,->](0,0)--(1,-1);
\draw[thick,->](0,0)--(1,1);
\draw[thick,->](0,0)--(-1,0);
\end{tikzpicture}
&
\begin{tikzpicture}[scale=.2, baseline=(current bounding box.center)]
\foreach \x in {-1,0,1} \foreach \y in {-1,0,1} \fill(\x,\y) circle[radius=2pt];
\draw[thick,->](0,0)--(-1,1);
\draw[thick,->](0,0)--(1,1);
\draw[thick,->](0,0)--(-1,0);
\draw[thick,->](0,0)--(1,0);
\draw[thick,->](0,0)--(0,-1);
\end{tikzpicture}
&
\begin{tikzpicture}[scale=.2, baseline=(current bounding box.center)]
\foreach \x in {-1,0,1} \foreach \y in {-1,0,1} \fill(\x,\y) circle[radius=2pt];
\draw[thick,->](0,0)--(1,-1);
\draw[thick,->](0,0)--(1,1);
\draw[thick,->](0,0)--(-1,0);
\draw[thick,->](0,0)--(0,1);
\draw[thick,->](0,0)--(0,-1);
\end{tikzpicture}
&
\begin{tikzpicture}[scale=.2, baseline=(current bounding box.center)]
\foreach \x in {-1,0,1} \foreach \y in {-1,0,1} \fill(\x,\y) circle[radius=2pt];
\draw[thick,->](0,0)--(-1,1);
\draw[thick,->](0,0)--(0,1);
\draw[thick,->](0,0)--(1,1);
\draw[thick,->](0,0)--(0,-1);
\end{tikzpicture}
&
\begin{tikzpicture}[scale=.2, baseline=(current bounding box.center)]
\foreach \x in {-1,0,1} \foreach \y in {-1,0,1} \fill(\x,\y) circle[radius=2pt];
\draw[thick,->](0,0)--(1,1);
\draw[thick,->](0,0)--(1,0);
\draw[thick,->](0,0)--(-1,0);
\draw[thick,->](0,0)--(1,-1);
\end{tikzpicture}
\\
\begin{tikzpicture}[scale=.2, baseline=(current bounding box.center)]
\foreach \x in {-1,0,1} \foreach \y in {-1,0,1} \fill(\x,\y) circle[radius=2pt];
\draw[thick,->](0,0)--(-1,1);
\draw[thick,->](0,0)--(0,1);
\draw[thick,->](0,0)--(1,1);
\draw[thick,->](0,0)--(-1,0);
\draw[thick,->](0,0)--(1,0);
\draw[thick,->](0,0)--(0,-1);
\end{tikzpicture}
&
\begin{tikzpicture}[scale=.2, baseline=(current bounding box.center)]
\foreach \x in {-1,0,1} \foreach \y in {-1,0,1} \fill(\x,\y) circle[radius=2pt];
\draw[thick,->](0,0)--(1,-1);
\draw[thick,->](0,0)--(0,1);
\draw[thick,->](0,0)--(1,1);
\draw[thick,->](0,0)--(-1,0);
\draw[thick,->](0,0)--(1,0);
\draw[thick,->](0,0)--(0,-1);
\end{tikzpicture}
&
\begin{tikzpicture}[scale=.2, baseline=(current bounding box.center)]
\foreach \x in {-1,0,1} \foreach \y in {-1,0,1} \fill(\x,\y) circle[radius=2pt];
\draw[thick,->](0,0)--(-1,1);
\draw[thick,->](0,0)--(0,1);
\draw[thick,->](0,0)--(1,1);
\draw[thick,->](0,0)--(-1,-1);
\draw[thick,->](0,0)--(1,-1);
\end{tikzpicture}
& 
\begin{tikzpicture}[scale=.2, baseline=(current bounding box.center)]
\foreach \x in {-1,0,1} \foreach \y in {-1,0,1} \fill(\x,\y) circle[radius=2pt];
\draw[thick,->](0,0)--(-1,1);
\draw[thick,->](0,0)--(1,0);
\draw[thick,->](0,0)--(1,1);
\draw[thick,->](0,0)--(-1,-1);
\draw[thick,->](0,0)--(1,-1);
\end{tikzpicture}
& 
\begin{tikzpicture}[scale=.2, baseline=(current bounding box.center)]
\foreach \x in {-1,0,1} \foreach \y in {-1,0,1} \fill(\x,\y) circle[radius=2pt];
\draw[thick,->](0,0)--(-1,1);
\draw[thick,->](0,0)--(0,1);
\draw[thick,->](0,0)--(1,1);
\draw[thick,->](0,0)--(-1,0);
\draw[thick,->](0,0)--(1,0);
\draw[thick,->](0,0)--(-1,-1);
\draw[thick,->](0,0)--(1,-1);
\end{tikzpicture}
&
\begin{tikzpicture}[scale=.2, baseline=(current bounding box.center)]
\foreach \x in {-1,0,1} \foreach \y in {-1,0,1} \fill(\x,\y) circle[radius=2pt];
\draw[thick,->](0,0)--(-1,1);
\draw[thick,->](0,0)--(0,1);
\draw[thick,->](0,0)--(1,1);
\draw[thick,->](0,0)--(0,-1);
\draw[thick,->](0,0)--(1,0);
\draw[thick,->](0,0)--(-1,-1);
\draw[thick,->](0,0)--(1,-1);
\end{tikzpicture}
&
\begin{tikzpicture}[scale=.2, baseline=(current bounding box.center)]
\foreach \x in {-1,0,1} \foreach \y in {-1,0,1} \fill(\x,\y) circle[radius=2pt];
\draw[thick,->](0,0)--(0,1);
\draw[thick,->](0,0)--(-1,-1);
\draw[thick,->](0,0)--(0,-1);
\draw[thick,->](0,0)--(1,-1);
\end{tikzpicture}
&
\begin{tikzpicture}[scale=.2, baseline=(current bounding box.center)]
\foreach \x in {-1,0,1} \foreach \y in {-1,0,1} \fill(\x,\y) circle[radius=2pt];
\draw[thick,->](0,0)--(1,0);
\draw[thick,->](0,0)--(-1,-1);
\draw[thick,->](0,0)--(-1,0);
\draw[thick,->](0,0)--(-1,1);
\end{tikzpicture}
&
\begin{tikzpicture}[scale=.2, baseline=(current bounding box.center)]
\foreach \x in {-1,0,1} \foreach \y in {-1,0,1} \fill(\x,\y) circle[radius=2pt];
\draw[thick,->](0,0)--(0,1);
\draw[thick,->](0,0)--(-1,0);
\draw[thick,->](0,0)--(1,0);
\draw[thick,->](0,0)--(-1,-1);red
\draw[thick,->](0,0)--(0,-1);
\draw[thick,->](0,0)--(1,-1);
\end{tikzpicture}
&
\begin{tikzpicture}[scale=.2, baseline=(current bounding box.center)]
\foreach \x in {-1,0,1} \foreach \y in {-1,0,1} \fill(\x,\y) circle[radius=2pt];
\draw[thick,->](0,0)--(0,1);
\draw[thick,->](0,0)--(-1,0);
\draw[thick,->](0,0)--(1,0);
\draw[thick,->](0,0)--(-1,-1);red
\draw[thick,->](0,0)--(0,-1);
\draw[thick,->](0,0)--(-1,1);
\end{tikzpicture}
&
\begin{tikzpicture}[scale=.2, baseline=(current bounding box.center)]
\foreach \x in {-1,0,1} \foreach \y in {-1,0,1} \fill(\x,\y) circle[radius=2pt];
\draw[thick,->](0,0)--(-1,1);
\draw[thick,->](0,0)--(1,1);
\draw[thick,->](0,0)--(-1,-1);
\draw[thick,->](0,0)--(0,-1);
\draw[thick,->](0,0)--(1,-1);
\end{tikzpicture}
&
\begin{tikzpicture}[scale=.2, baseline=(current bounding box.center)]
\foreach \x in {-1,0,1} \foreach \y in {-1,0,1} \fill(\x,\y) circle[radius=2pt];
\draw[thick,->](0,0)--(-1,1);
\draw[thick,->](0,0)--(1,1);
\draw[thick,->](0,0)--(-1,-1);
\draw[thick,->](0,0)--(-1,0);
\draw[thick,->](0,0)--(1,-1);
\end{tikzpicture}
&
\begin{tikzpicture}[scale=.2, baseline=(current bounding box.center)]
\foreach \x in {-1,0,1} \foreach \y in {-1,0,1} \fill(\x,\y) circle[radius=2pt];
\draw[thick,->](0,0)--(-1,1);
\draw[thick,->](0,0)--(1,1);
\draw[thick,->](0,0)--(-1,0);
\draw[thick,->](0,0)--(1,0);
\draw[thick,->](0,0)--(-1,-1);
\draw[thick,->](0,0)--(0,-1);
\draw[thick,->](0,0)--(1,-1);
\end{tikzpicture}
&
\begin{tikzpicture}[scale=.2, baseline=(current bounding box.center)]
\foreach \x in {-1,0,1} \foreach \y in {-1,0,1} \fill(\x,\y) circle[radius=2pt];
\draw[thick,->](0,0)--(-1,1);
\draw[thick,->](0,0)--(1,1);
\draw[thick,->](0,0)--(-1,0);
\draw[thick,->](0,0)--(0,1);
\draw[thick,->](0,0)--(-1,-1);
\draw[thick,->](0,0)--(0,-1);
\draw[thick,->](0,0)--(1,-1);
\end{tikzpicture}
&
\begin{tikzpicture}[scale=.2, baseline=(current bounding box.center)]
\foreach \x in {-1,0,1} \foreach \y in {-1,0,1} \fill(\x,\y) circle[radius=2pt];
\draw[thick,->](0,0)--(0,1);
\draw[thick,->](0,0)--(1,-1);
\draw[thick,->](0,0)--(-1,-1);
\end{tikzpicture} 
&
\begin{tikzpicture}[scale=.2, baseline=(current bounding box.center)]
\foreach \x in {-1,0,1} \foreach \y in {-1,0,1} \fill(\x,\y) circle[radius=2pt];
\draw[thick,->](0,0)--(1,0);
\draw[thick,->](0,0)--(-1,1);
\draw[thick,->](0,0)--(-1,-1);
\end{tikzpicture} 
\\
\begin{tikzpicture}[scale=.2, baseline=(current bounding box.center)]
\foreach \x in {-1,0,1} \foreach \y in {-1,0,1} \fill(\x,\y) circle[radius=2pt];
\draw[thick,->](0,0)--(0,1);
\draw[thick,->](0,0)--(1,-1);
\draw[thick,->](0,0)--(-1,-1);
\draw[thick,->](0,0)--(1,0);
\draw[thick,->](0,0)--(-1,0);
\end{tikzpicture}
&
\begin{tikzpicture}[scale=.2, baseline=(current bounding box.center)]
\foreach \x in {-1,0,1} \foreach \y in {-1,0,1} \fill(\x,\y) circle[radius=2pt];
\draw[thick,->](0,0)--(0,1);
\draw[thick,->](0,0)--(-1,1);
\draw[thick,->](0,0)--(-1,-1);
\draw[thick,->](0,0)--(1,0);
\draw[thick,->](0,0)--(0,-1);
\end{tikzpicture}
&
\begin{tikzpicture}[scale=.2, baseline=(current bounding box.center)]
\foreach \x in {-1,0,1} \foreach \y in {-1,0,1} \fill(\x,\y) circle[radius=2pt];
\draw[thick,->](0,0)--(0,1);
\draw[thick,->](0,0)--(-1,0);
\draw[thick,->](0,0)--(1,-1);
\end{tikzpicture}
&
\begin{tikzpicture}[scale=.2, baseline=(current bounding box.center)]
\foreach \x in {-1,0,1} \foreach \y in {-1,0,1} \fill(\x,\y) circle[radius=2pt];
\draw[thick,->](0,0)--(1,0);
\draw[thick,->](0,0)--(0,-1);
\draw[thick,->](0,0)--(-1,1);
\end{tikzpicture}
&
\begin{tikzpicture}[scale=.2, baseline=(current bounding box.center)]
\foreach \x in {-1,0,1} \foreach \y in {-1,0,1} \fill(\x,\y) circle[radius=2pt];
\draw[thick,->](0,0)--(-1,1);
\draw[thick,->](0,0)--(0,1);
\draw[thick,->](0,0)--(-1,0);
\draw[thick,->](0,0)--(1,0);
\draw[thick,->](0,0)--(0,-1);
\draw[thick,->](0,0)--(1,-1);
\end{tikzpicture}
&
\begin{tikzpicture}[scale=.2, baseline=(current bounding box.center)]
\foreach \x in {-1,0,1} \foreach \y in {-1,0,1} \fill(\x,\y) circle[radius=2pt];
\draw[thick,->](0,0)--(-1,1);
\draw[thick,->](0,0)--(1,-1);
\draw[thick,->](0,0)--(1,0);
\draw[thick,->](0,0)--(-1,0);
\end{tikzpicture}
&
\begin{tikzpicture}[scale=.2, baseline=(current bounding box.center)]
\foreach \x in {-1,0,1} \foreach \y in {-1,0,1} \fill(\x,\y) circle[radius=2pt];
\draw[thick,->](0,0)--(-1,1);
\draw[thick,->](0,0)--(1,-1);
\draw[thick,->](0,0)--(0,1);
\draw[thick,->](0,0)--(0,-1);
\end{tikzpicture}
&
\begin{tikzpicture}[scale=.2, baseline=(current bounding box.center)]
\foreach \x in {-1,0,1} \foreach \y in {-1,0,1} \fill(\x,\y) circle[radius=2pt];
\draw[thick,->](0,0)--(0,1);
\draw[thick,->](0,0)--(1,0);
\draw[thick,->](0,0)--(-1,-1);
\draw[thick,->](0,0)--(0,-1);
\end{tikzpicture}
&
\begin{tikzpicture}[scale=.2, baseline=(current bounding box.center)]
\foreach \x in {-1,0,1} \foreach \y in {-1,0,1} \fill(\x,\y) circle[radius=2pt];
\draw[thick,->](0,0)--(0,1);
\draw[thick,->](0,0)--(1,0);
\draw[thick,->](0,0)--(-1,-1);
\draw[thick,->](0,0)--(-1,0);
\end{tikzpicture}
&
\begin{tikzpicture}[scale=.2, baseline=(current bounding box.center)]
\foreach \x in {-1,0,1} \foreach \y in {-1,0,1} \fill(\x,\y) circle[radius=2pt];
\draw[thick,->](0,0)--(0,1);
\draw[thick,->](0,0)--(1,0);
\draw[thick,->](0,0)--(-1,-1);
\draw[thick,->](0,0)--(1,-1);
\end{tikzpicture}
&
\begin{tikzpicture}[scale=.2, baseline=(current bounding box.center)]
\foreach \x in {-1,0,1} \foreach \y in {-1,0,1} \fill(\x,\y) circle[radius=2pt];
\draw[thick,->](0,0)--(0,1);
\draw[thick,->](0,0)--(1,0);
\draw[thick,->](0,0)--(-1,-1);
\draw[thick,->](0,0)--(-1,1);
\end{tikzpicture}
&
\begin{tikzpicture}[scale=.2, baseline=(current bounding box.center)]
\foreach \x in {-1,0,1} \foreach \y in {-1,0,1} \fill(\x,\y) circle[radius=2pt];
\draw[thick,->](0,0)--(0,1);
\draw[thick,->](0,0)--(-1,0);
\draw[thick,->](0,0)--(1,0);
\draw[thick,->](0,0)--(1,-1);
\end{tikzpicture}
&
\begin{tikzpicture}[scale=.2, baseline=(current bounding box.center)]
\foreach \x in {-1,0,1} \foreach \y in {-1,0,1} \fill(\x,\y) circle[radius=2pt];
\draw[thick,->](0,0)--(0,1);
\draw[thick,->](0,0)--(0,-1);
\draw[thick,->](0,0)--(1,0);
\draw[thick,->](0,0)--(-1,1);
\end{tikzpicture}
&
\begin{tikzpicture}[scale=.2, baseline=(current bounding box.center)]
\foreach \x in {-1,0,1} \foreach \y in {-1,0,1} \fill(\x,\y) circle[radius=2pt];
\draw[thick,->](0,0)--(0,1);
\draw[thick,->](0,0)--(1,0);
\draw[thick,->](0,0)--(-1,-1);
\draw[thick,->](0,0)--(0,-1);
\draw[thick,->](0,0)--(1,-1);
\end{tikzpicture}
&
\begin{tikzpicture}[scale=.2, baseline=(current bounding box.center)]
\foreach \x in {-1,0,1} \foreach \y in {-1,0,1} \fill(\x,\y) circle[radius=2pt];
\draw[thick,->](0,0)--(1,0);
\draw[thick,->](0,0)--(0,1);
\draw[thick,->](0,0)--(-1,-1);
\draw[thick,->](0,0)--(-1,0);
\draw[thick,->](0,0)--(-1,1);
\end{tikzpicture}
&
\begin{tikzpicture}[scale=.2, baseline=(current bounding box.center)]
\foreach \x in {-1,0,1} \foreach \y in {-1,0,1} \fill(\x,\y) circle[radius=2pt];
\draw[thick,->](0,0)--(-1,1);
\draw[thick,->](0,0)--(0,1);
\draw[thick,->](0,0)--(1,0);
\draw[thick,->](0,0)--(0,-1);
\draw[thick,->](0,0)--(1,-1);
\end{tikzpicture}
\\
\begin{tikzpicture}[scale=.2, baseline=(current bounding box.center)]
\foreach \x in {-1,0,1} \foreach \y in {-1,0,1} \fill(\x,\y) circle[radius=2pt];
\draw[thick,->](0,0)--(-1,1);
\draw[thick,->](0,0)--(0,1);
\draw[thick,->](0,0)--(1,0);
\draw[thick,->](0,0)--(-1,0);
\draw[thick,->](0,0)--(1,-1);
\end{tikzpicture}
&
\begin{tikzpicture}[scale=.2, baseline=(current bounding box.center)]
\foreach \x in {-1,0,1} \foreach \y in {-1,0,1} \fill(\x,\y) circle[radius=2pt];
\draw[thick,->](0,0)--(0,1);
\draw[thick,->](0,0)--(1,1);
\draw[thick,->](0,0)--(-1,0);
\draw[thick,->](0,0)--(0,-1);
\end{tikzpicture}
&
\begin{tikzpicture}[scale=.2, baseline=(current bounding box.center)]
\foreach \x in {-1,0,1} \foreach \y in {-1,0,1} \fill(\x,\y) circle[radius=2pt];
\draw[thick,->](0,0)--(1,0);
\draw[thick,->](0,0)--(1,1);
\draw[thick,->](0,0)--(-1,0);
\draw[thick,->](0,0)--(0,-1);
\end{tikzpicture}
&
\begin{tikzpicture}[scale=.2, baseline=(current bounding box.center)]
\foreach \x in {-1,0,1} \foreach \y in {-1,0,1} \fill(\x,\y) circle[radius=2pt];
\draw[thick,->](0,0)--(0,1);
\draw[thick,->](0,0)--(1,1);
\draw[thick,->](0,0)--(-1,0);
\draw[thick,->](0,0)--(-1,-1);
\draw[thick,->](0,0)--(0,-1);
\end{tikzpicture}
&
\begin{tikzpicture}[scale=.2, baseline=(current bounding box.center)]
\foreach \x in {-1,0,1} \foreach \y in {-1,0,1} \fill(\x,\y) circle[radius=2pt];
\draw[thick,->](0,0)--(1,0);
\draw[thick,->](0,0)--(1,1);
\draw[thick,->](0,0)--(-1,0);
\draw[thick,->](0,0)--(-1,-1);
\draw[thick,->](0,0)--(0,-1);
\end{tikzpicture}
&
\begin{tikzpicture}[scale=.2, baseline=(current bounding box.center)]
\foreach \x in {-1,0,1} \foreach \y in {-1,0,1} \fill(\x,\y) circle[radius=2pt];
\draw[thick,->](0,0)--(0,1);
\draw[thick,->](0,0)--(1,1);
\draw[thick,->](0,0)--(-1,0);
\draw[thick,->](0,0)--(1,0);
\draw[thick,->](0,0)--(-1,-1);
\end{tikzpicture}
& 
\begin{tikzpicture}[scale=.2, baseline=(current bounding box.center)]
\foreach \x in {-1,0,1} \foreach \y in {-1,0,1} \fill(\x,\y) circle[radius=2pt];
\draw[thick,->](0,0)--(0,1);
\draw[thick,->](0,0)--(1,1);
\draw[thick,->](0,0)--(0,-1);
\draw[thick,->](0,0)--(1,0);
\draw[thick,->](0,0)--(-1,-1);
\end{tikzpicture}
& 
\begin{tikzpicture}[scale=.2, baseline=(current bounding box.center)]
\foreach \x in {-1,0,1} \foreach \y in {-1,0,1} \fill(\x,\y) circle[radius=2pt];
\draw[thick,->](0,0)--(0,1);
\draw[thick,->](0,0)--(1,1);
\draw[thick,->](0,0)--(-1,0);
\draw[thick,->](0,0)--(1,0);
\draw[thick,->](0,0)--(0,-1);
\end{tikzpicture}
&
\begin{tikzpicture}[scale=.2, baseline=(current bounding box.center)]
\foreach \x in {-1,0,1} \foreach \y in {-1,0,1} \fill(\x,\y) circle[radius=2pt];
\draw[thick,->](0,0)--(-1,1);
\draw[thick,->](0,0)--(1,1);
\draw[thick,->](0,0)--(0,-1);
\draw[thick,->](0,0)--(1,-1);
\end{tikzpicture}
& 
\begin{tikzpicture}[scale=.2, baseline=(current bounding box.center)]
\foreach \x in {-1,0,1} \foreach \y in {-1,0,1} \fill(\x,\y) circle[radius=2pt];
\draw[thick,->](0,0)--(-1,1);
\draw[thick,->](0,0)--(1,1);
\draw[thick,->](0,0)--(-1,0);
\draw[thick,->](0,0)--(1,-1);
\end{tikzpicture}
& 
\begin{tikzpicture}[scale=.2, baseline=(current bounding box.center)]
\foreach \x in {-1,0,1} \foreach \y in {-1,0,1} \fill(\x,\y) circle[radius=2pt];
\draw[thick,->](0,0)--(-1,1);
\draw[thick,->](0,0)--(0,1);
\draw[thick,->](0,0)--(1,1);
\draw[thick,->](0,0)--(0,-1);
\draw[thick,->](0,0)--(1,-1);
\end{tikzpicture}
&
\begin{tikzpicture}[scale=.2, baseline=(current bounding box.center)]
\foreach \x in {-1,0,1} \foreach \y in {-1,0,1} \fill(\x,\y) circle[radius=2pt];
\draw[thick,->](0,0)--(-1,1);
\draw[thick,->](0,0)--(1,0);
\draw[thick,->](0,0)--(1,1);
\draw[thick,->](0,0)--(-1,0);
\draw[thick,->](0,0)--(1,-1);
\end{tikzpicture}
&
\begin{tikzpicture}[scale=.2, baseline=(current bounding box.center)]
\foreach \x in {-1,0,1} \foreach \y in {-1,0,1} \fill(\x,\y) circle[radius=2pt];
\draw[thick,->](0,0)--(-1,1);
\draw[thick,->](0,0)--(0,1);
\draw[thick,->](0,0)--(1,1);
\draw[thick,->](0,0)--(-1,0);
\draw[thick,->](0,0)--(1,-1);
\end{tikzpicture}
&
\begin{tikzpicture}[scale=.2, baseline=(current bounding box.center)]
\foreach \x in {-1,0,1} \foreach \y in {-1,0,1} \fill(\x,\y) circle[radius=2pt];
\draw[thick,->](0,0)--(1,-1);
\draw[thick,->](0,0)--(1,0);
\draw[thick,->](0,0)--(1,1);
\draw[thick,->](0,0)--(0,-1);
\draw[thick,->](0,0)--(-1,1);
\end{tikzpicture}
&
\begin{tikzpicture}[scale=.2, baseline=(current bounding box.center)]
\foreach \x in {-1,0,1} \foreach \y in {-1,0,1} \fill(\x,\y) circle[radius=2pt];
\draw[thick,->](0,0)--(-1,1);
\draw[thick,->](0,0)--(1,1);
\draw[thick,->](0,0)--(-1,0);
\draw[thick,->](0,0)--(0,-1);
\draw[thick,->](0,0)--(1,-1);
\end{tikzpicture}
&
\begin{tikzpicture}[scale=.2, baseline=(current bounding box.center)]
\foreach \x in {-1,0,1} \foreach \y in {-1,0,1} \fill(\x,\y) circle[radius=2pt];
\draw[thick,->](0,0)--(-1,1);
\draw[thick,->](0,0)--(0,1);
\draw[thick,->](0,0)--(1,1);
\draw[thick,->](0,0)--(-1,0);
\draw[thick,->](0,0)--(1,0);
\draw[thick,->](0,0)--(1,-1);
\end{tikzpicture}
\\
\begin{tikzpicture}[scale=.2, baseline=(current bounding box.center)]
\foreach \x in {-1,0,1} \foreach \y in {-1,0,1} \fill(\x,\y) circle[radius=2pt];
\draw[thick,->](0,0)--(-1,1);
\draw[thick,->](0,0)--(0,1);
\draw[thick,->](0,0)--(1,1);
\draw[thick,->](0,0)--(0,-1);
\draw[thick,->](0,0)--(1,0);
\draw[thick,->](0,0)--(1,-1);
\end{tikzpicture}
&
\begin{tikzpicture}[scale=.2, baseline=(current bounding box.center)]
\foreach \x in {-1,0,1} \foreach \y in {-1,0,1} \fill(\x,\y) circle[radius=2pt];
\draw[thick,->](0,0)--(-1,1);
\draw[thick,->](0,0)--(0,1);
\draw[thick,->](0,0)--(1,1);
\draw[thick,->](0,0)--(1,0);
\draw[thick,->](0,0)--(-1,-1);
\draw[thick,->](0,0)--(1,-1);
\end{tikzpicture}
&
\begin{tikzpicture}[scale=.2, baseline=(current bounding box.center)]
\foreach \x in {-1,0,1} \foreach \y in {-1,0,1} \fill(\x,\y) circle[radius=2pt];
\draw[thick,->](0,0)--(-1,1);
\draw[thick,->](0,0)--(0,1);
\draw[thick,->](0,0)--(1,1);
\draw[thick,->](0,0)--(-1,0);
\draw[thick,->](0,0)--(0,-1);
\draw[thick,->](0,0)--(1,-1);
\end{tikzpicture}
&
\begin{tikzpicture}[scale=.2, baseline=(current bounding box.center)]
\foreach \x in {-1,0,1} \foreach \y in {-1,0,1} \fill(\x,\y) circle[radius=2pt];
\draw[thick,->](0,0)--(-1,1);
\draw[thick,->](0,0)--(1,0);
\draw[thick,->](0,0)--(1,1);
\draw[thick,->](0,0)--(-1,0);
\draw[thick,->](0,0)--(0,-1);
\draw[thick,->](0,0)--(1,-1);
\end{tikzpicture}
&
\begin{tikzpicture}[scale=.2, baseline=(current bounding box.center)]
\foreach \x in {-1,0,1} \foreach \y in {-1,0,1} \fill(\x,\y) circle[radius=2pt];
\draw[thick,->](0,0)--(-1,1);
\draw[thick,->](0,0)--(0,1);
\draw[thick,->](0,0)--(1,1);
\draw[thick,->](0,0)--(-1,0);
\draw[thick,->](0,0)--(-1,-1);
\draw[thick,->](0,0)--(1,-1);
\end{tikzpicture}
&
\begin{tikzpicture}[scale=.2, baseline=(current bounding box.center)]
\foreach \x in {-1,0,1} \foreach \y in {-1,0,1} \fill(\x,\y) circle[radius=2pt];
\draw[thick,->](0,0)--(-1,1);
\draw[thick,->](0,0)--(1,0);
\draw[thick,->](0,0)--(1,1);
\draw[thick,->](0,0)--(0,-1);
\draw[thick,->](0,0)--(-1,-1);
\draw[thick,->](0,0)--(1,-1);
\end{tikzpicture}
&
\begin{tikzpicture}[scale=.2, baseline=(current bounding box.center)]
\foreach \x in {-1,0,1} \foreach \y in {-1,0,1} \fill(\x,\y) circle[radius=2pt];
\draw[thick,->](0,0)--(-1,1);
\draw[thick,->](0,0)--(1,1);
\draw[thick,->](0,0)--(-1,0);
\draw[thick,->](0,0)--(-1,-1);
\draw[thick,->](0,0)--(0,-1);
\draw[thick,->](0,0)--(1,-1);
\end{tikzpicture}
&
\begin{tikzpicture}[scale=.2, baseline=(current bounding box.center)]
\foreach \x in {-1,0,1} \foreach \y in {-1,0,1} \fill(\x,\y) circle[radius=2pt];
\draw[thick,->](0,0)--(-1,1);
\draw[thick,->](0,0)--(0,1);
\draw[thick,->](0,0)--(1,1);
\draw[thick,->](0,0)--(-1,0);
\draw[thick,->](0,0)--(1,0);
\draw[thick,->](0,0)--(0,-1);
\draw[thick,->](0,0)--(1,-1);
\end{tikzpicture}
&
\begin{tikzpicture}[scale=.2, baseline=(current bounding box.center)]
\foreach \x in {-1,0,1} \foreach \y in {-1,0,1} \fill(\x,\y) circle[radius=2pt];
\draw[thick,->](0,0)--(-1,1);
\draw[thick,->](0,0)--(1,1);
\draw[thick,->](0,0)--(-1,-1);
\draw[thick,->](0,0)--(0,-1);
\end{tikzpicture}
&
\begin{tikzpicture}[scale=.2, baseline=(current bounding box.center)]
\foreach \x in {-1,0,1} \foreach \y in {-1,0,1} \fill(\x,\y) circle[radius=2pt];
\draw[thick,->](0,0)--(1,-1);
\draw[thick,->](0,0)--(1,1);
\draw[thick,->](0,0)--(-1,-1);
\draw[thick,->](0,0)--(-1,0);
\end{tikzpicture}
&
\begin{tikzpicture}[scale=.2, baseline=(current bounding box.center)]
\foreach \x in {-1,0,1} \foreach \y in {-1,0,1} \fill(\x,\y) circle[radius=2pt];
\draw[thick,->](0,0)--(-1,1);
\draw[thick,->](0,0)--(1,1);
\draw[thick,->](0,0)--(-1,0);
\draw[thick,->](0,0)--(0,-1);
\end{tikzpicture}
&
\begin{tikzpicture}[scale=.2, baseline=(current bounding box.center)]
\foreach \x in {-1,0,1} \foreach \y in {-1,0,1} \fill(\x,\y) circle[radius=2pt];
\draw[thick,->](0,0)--(1,-1);
\draw[thick,->](0,0)--(1,1);
\draw[thick,->](0,0)--(-1,0);
\draw[thick,->](0,0)--(0,-1);
\end{tikzpicture}
&
\begin{tikzpicture}[scale=.2, baseline=(current bounding box.center)]
\foreach \x in {-1,0,1} \foreach \y in {-1,0,1} \fill(\x,\y) circle[radius=2pt];
\draw[thick,->](0,0)--(-1,1);
\draw[thick,->](0,0)--(0,1);
\draw[thick,->](0,0)--(1,1);
\draw[thick,->](0,0)--(-1,-1);
\draw[thick,->](0,0)--(0,-1);
\end{tikzpicture}
&
\begin{tikzpicture}[scale=.2, baseline=(current bounding box.center)]
\foreach \x in {-1,0,1} \foreach \y in {-1,0,1} \fill(\x,\y) circle[radius=2pt];
\draw[thick,->](0,0)--(1,-1);
\draw[thick,->](0,0)--(1,0);
\draw[thick,->](0,0)--(1,1);
\draw[thick,->](0,0)--(-1,-1);
\draw[thick,->](0,0)--(-1,0);
\end{tikzpicture}
&
\begin{tikzpicture}[scale=.2, baseline=(current bounding box.center)]
\foreach \x in {-1,0,1} \foreach \y in {-1,0,1} \fill(\x,\y) circle[radius=2pt];
\draw[thick,->](0,0)--(-1,1);
\draw[thick,->](0,0)--(0,1);
\draw[thick,->](0,0)--(1,1);
\draw[thick,->](0,0)--(-1,0);
\draw[thick,->](0,0)--(0,-1);
\end{tikzpicture}
&
\begin{tikzpicture}[scale=.2, baseline=(current bounding box.center)]
\foreach \x in {-1,0,1} \foreach \y in {-1,0,1} \fill(\x,\y) circle[radius=2pt];
\draw[thick,->](0,0)--(1,-1);
\draw[thick,->](0,0)--(1,0);
\draw[thick,->](0,0)--(1,1);
\draw[thick,->](0,0)--(-1,0);
\draw[thick,->](0,0)--(0,-1);
\end{tikzpicture}
\\
\begin{tikzpicture}[scale=.2, baseline=(current bounding box.center)]
\foreach \x in {-1,0,1} \foreach \y in {-1,0,1} \fill(\x,\y) circle[radius=2pt];
\draw[thick,->](0,0)--(-1,1);
\draw[thick,->](0,0)--(1,1);
\draw[thick,->](0,0)--(-1,0);
\draw[thick,->](0,0)--(-1,-1);
\draw[thick,->](0,0)--(0,-1);
\end{tikzpicture}
&
\begin{tikzpicture}[scale=.2, baseline=(current bounding box.center)]
\foreach \x in {-1,0,1} \foreach \y in {-1,0,1} \fill(\x,\y) circle[radius=2pt];
\draw[thick,->](0,0)--(1,-1);
\draw[thick,->](0,0)--(1,1);
\draw[thick,->](0,0)--(-1,0);
\draw[thick,->](0,0)--(-1,-1);
\draw[thick,->](0,0)--(0,-1);
\end{tikzpicture}
&
\begin{tikzpicture}[scale=.2, baseline=(current bounding box.center)]
\foreach \x in {-1,0,1} \foreach \y in {-1,0,1} \fill(\x,\y) circle[radius=2pt];
\draw[thick,->](0,0)--(-1,1);
\draw[thick,->](0,0)--(0,1);
\draw[thick,->](0,0)--(1,1);
\draw[thick,->](0,0)--(-1,0);
\draw[thick,->](0,0)--(-1,-1);
\draw[thick,->](0,0)--(0,-1);
\end{tikzpicture}
&
\begin{tikzpicture}[scale=.2, baseline=(current bounding box.center)]
\foreach \x in {-1,0,1} \foreach \y in {-1,0,1} \fill(\x,\y) circle[radius=2pt];
\draw[thick,->](0,0)--(1,-1);
\draw[thick,->](0,0)--(1,0);
\draw[thick,->](0,0)--(1,1);
\draw[thick,->](0,0)--(-1,0);
\draw[thick,->](0,0)--(-1,-1);
\draw[thick,->](0,0)--(0,-1);
\end{tikzpicture}
&
\begin{tikzpicture}[scale=.2, baseline=(current bounding box.center)]
\foreach \x in {-1,0,1} \foreach \y in {-1,0,1} \fill(\x,\y) circle[radius=2pt];
\draw[thick,->](0,0)--(-1,1);
\draw[thick,->](0,0)--(0,1);
\draw[thick,->](0,0)--(1,1);
\draw[thick,->](0,0)--(1,0);
\draw[thick,->](0,0)--(-1,-1);
\end{tikzpicture}
&
\begin{tikzpicture}[scale=.2, baseline=(current bounding box.center)]
\foreach \x in {-1,0,1} \foreach \y in {-1,0,1} \fill(\x,\y) circle[radius=2pt];
\draw[thick,->](0,0)--(1,-1);
\draw[thick,->](0,0)--(0,1);
\draw[thick,->](0,0)--(1,1);
\draw[thick,->](0,0)--(1,0);
\draw[thick,->](0,0)--(-1,-1);
\end{tikzpicture}
&
\begin{tikzpicture}[scale=.2, baseline=(current bounding box.center)]
\foreach \x in {-1,0,1} \foreach \y in {-1,0,1} \fill(\x,\y) circle[radius=2pt];
\draw[thick,->](0,0)--(-1,1);
\draw[thick,->](0,0)--(0,1);
\draw[thick,->](0,0)--(1,1);
\draw[thick,->](0,0)--(1,0);
\draw[thick,->](0,0)--(0,-1);
\end{tikzpicture}
&
\begin{tikzpicture}[scale=.2, baseline=(current bounding box.center)]
\foreach \x in {-1,0,1} \foreach \y in {-1,0,1} \fill(\x,\y) circle[radius=2pt];
\draw[thick,->](0,0)--(1,-1);
\draw[thick,->](0,0)--(0,1);
\draw[thick,->](0,0)--(1,1);
\draw[thick,->](0,0)--(1,0);
\draw[thick,->](0,0)--(-1,0);
\end{tikzpicture}
&
\begin{tikzpicture}[scale=.2, baseline=(current bounding box.center)]
\foreach \x in {-1,0,1} \foreach \y in {-1,0,1} \fill(\x,\y) circle[radius=2pt];
\draw[thick,->](0,0)--(-1,1);
\draw[thick,->](0,0)--(0,1);
\draw[thick,->](0,0)--(1,1);
\draw[thick,->](0,0)--(-1,0);
\draw[thick,->](0,0)--(1,0);
\draw[thick,->](0,0)--(-1,-1);
\draw[thick,->](0,0)--(0,-1);
\end{tikzpicture}
&
\begin{tikzpicture}[scale=.2, baseline=(current bounding box.center)]
\foreach \x in {-1,0,1} \foreach \y in {-1,0,1} \fill(\x,\y) circle[radius=2pt];
\draw[thick,->](0,0)--(1,-1);
\draw[thick,->](0,0)--(0,1);
\draw[thick,->](0,0)--(1,1);
\draw[thick,->](0,0)--(-1,0);
\draw[thick,->](0,0)--(1,0);
\draw[thick,->](0,0)--(-1,-1);
\draw[thick,->](0,0)--(0,-1);
\end{tikzpicture}
&
\begin{tikzpicture}[scale=.2, baseline=(current bounding box.center)]
\foreach \x in {-1,0,1} \foreach \y in {-1,0,1} \fill(\x,\y) circle[radius=2pt];
\draw[thick,->](0,0)--(0,1);
\draw[thick,->](0,0)--(1,1);
\draw[thick,->](0,0)--(-1,-1);
\draw[thick,->](0,0)--(1,-1);
\end{tikzpicture}
&
\begin{tikzpicture}[scale=.2, baseline=(current bounding box.center)]
\foreach \x in {-1,0,1} \foreach \y in {-1,0,1} \fill(\x,\y) circle[radius=2pt];
\draw[thick,->](0,0)--(1,0);
\draw[thick,->](0,0)--(1,1);
\draw[thick,->](0,0)--(-1,-1);
\draw[thick,->](0,0)--(-1,1);
\end{tikzpicture}
&
\begin{tikzpicture}[scale=.2, baseline=(current bounding box.center)]
\foreach \x in {-1,0,1} \foreach \y in {-1,0,1} \fill(\x,\y) circle[radius=2pt];
\draw[thick,->](0,0)--(0,1);
\draw[thick,->](0,0)--(1,1);
\draw[thick,->](0,0)--(-1,0);
\draw[thick,->](0,0)--(1,-1);
\end{tikzpicture}
&
\begin{tikzpicture}[scale=.2, baseline=(current bounding box.center)]
\foreach \x in {-1,0,1} \foreach \y in {-1,0,1} \fill(\x,\y) circle[radius=2pt];
\draw[thick,->](0,0)--(1,0);
\draw[thick,->](0,0)--(1,1);
\draw[thick,->](0,0)--(0,-1);
\draw[thick,->](0,0)--(-1,1);
\end{tikzpicture}
&
\begin{tikzpicture}[scale=.2, baseline=(current bounding box.center)]
\foreach \x in {-1,0,1} \foreach \y in {-1,0,1} \fill(\x,\y) circle[radius=2pt];
\draw[thick,->](0,0)--(0,1);
\draw[thick,->](0,0)--(1,1);
\draw[thick,->](0,0)--(-1,0);
\draw[thick,->](0,0)--(-1,-1);
\draw[thick,->](0,0)--(1,-1);
\end{tikzpicture}
&
\begin{tikzpicture}[scale=.2, baseline=(current bounding box.center)]
\foreach \x in {-1,0,1} \foreach \y in {-1,0,1} \fill(\x,\y) circle[radius=2pt];
\draw[thick,->](0,0)--(1,0);
\draw[thick,->](0,0)--(1,1);
\draw[thick,->](0,0)--(0,-1);
\draw[thick,->](0,0)--(-1,-1);
\draw[thick,->](0,0)--(-1,1);
\end{tikzpicture}
\\
\begin{tikzpicture}[scale=.2, baseline=(current bounding box.center)]
\foreach \x in {-1,0,1} \foreach \y in {-1,0,1} \fill(\x,\y) circle[radius=2pt];
\draw[thick,->](0,0)--(0,1);
\draw[thick,->](0,0)--(1,1);
\draw[thick,->](0,0)--(-1,-1);
\draw[thick,->](0,0)--(0,-1);
\draw[thick,->](0,0)--(1,-1);
\end{tikzpicture}
&
\begin{tikzpicture}[scale=.2, baseline=(current bounding box.center)]
\foreach \x in {-1,0,1} \foreach \y in {-1,0,1} \fill(\x,\y) circle[radius=2pt];
\draw[thick,->](0,0)--(1,0);
\draw[thick,->](0,0)--(1,1);
\draw[thick,->](0,0)--(-1,-1);
\draw[thick,->](0,0)--(-1,0);
\draw[thick,->](0,0)--(-1,1);
\end{tikzpicture}
&
\begin{tikzpicture}[scale=.2, baseline=(current bounding box.center)]
\foreach \x in {-1,0,1} \foreach \y in {-1,0,1} \fill(\x,\y) circle[radius=2pt];
\draw[thick,->](0,0)--(0,1);
\draw[thick,->](0,0)--(1,1);
\draw[thick,->](0,0)--(1,0);
\draw[thick,->](0,0)--(-1,-1);
\draw[thick,->](0,0)--(0,-1);
\draw[thick,->](0,0)--(1,-1);
\end{tikzpicture}
&
\begin{tikzpicture}[scale=.2, baseline=(current bounding box.center)]
\foreach \x in {-1,0,1} \foreach \y in {-1,0,1} \fill(\x,\y) circle[radius=2pt];
\draw[thick,->](0,0)--(0,1);
\draw[thick,->](0,0)--(1,1);
\draw[thick,->](0,0)--(1,0);
\draw[thick,->](0,0)--(-1,-1);
\draw[thick,->](0,0)--(-1,0);
\draw[thick,->](0,0)--(-1,1);
\end{tikzpicture}
&
\begin{tikzpicture}[scale=.2, baseline=(current bounding box.center)]
\foreach \x in {-1,0,1} \foreach \y in {-1,0,1} \fill(\x,\y) circle[radius=2pt];
\draw[thick,->](0,0)--(0,1);
\draw[thick,->](0,0)--(1,1);
\draw[thick,->](0,0)--(-1,0);
\draw[thick,->](0,0)--(1,0);
\draw[thick,->](0,0)--(-1,-1);
\draw[thick,->](0,0)--(1,-1);
\end{tikzpicture}
&
\begin{tikzpicture}[scale=.2, baseline=(current bounding box.center)]
\foreach \x in {-1,0,1} \foreach \y in {-1,0,1} \fill(\x,\y) circle[radius=2pt];
\draw[thick,->](0,0)--(0,1);
\draw[thick,->](0,0)--(1,1);
\draw[thick,->](0,0)--(0,-1);
\draw[thick,->](0,0)--(1,0);
\draw[thick,->](0,0)--(-1,-1);
\draw[thick,->](0,0)--(-1,1);
\end{tikzpicture}
&
 \begin{tikzpicture}[scale=.2, baseline=(current bounding box.center)]
\foreach \x in {-1,0,1} \foreach \y in {-1,0,1} \fill(\x,\y) circle[radius=2pt];
\draw[thick,->](0,0)--(0,1);
\draw[thick,->](0,0)--(1,1);
\draw[thick,->](0,0)--(-1,0);
\draw[thick,->](0,0)--(-1,-1);
\draw[thick,->](0,0)--(0,-1);
\draw[thick,->](0,0)--(1,-1);
\end{tikzpicture}
&
 \begin{tikzpicture}[scale=.2, baseline=(current bounding box.center)]
\foreach \x in {-1,0,1} \foreach \y in {-1,0,1} \fill(\x,\y) circle[radius=2pt];
\draw[thick,->](0,0)--(1,0);
\draw[thick,->](0,0)--(1,1);
\draw[thick,->](0,0)--(-1,0);
\draw[thick,->](0,0)--(-1,-1);
\draw[thick,->](0,0)--(0,-1);
\draw[thick,->](0,0)--(-1,1);
\end{tikzpicture}
&
\begin{tikzpicture}[scale=.2, baseline=(current bounding box.center)]
\foreach \x in {-1,0,1} \foreach \y in {-1,0,1} \fill(\x,\y) circle[radius=2pt];
\draw[thick,->](0,0)--(0,1);
\draw[thick,->](0,0)--(-1,0);
\draw[thick,->](0,0)--(1,0);
\draw[thick,->](0,0)--(0,-1);
\draw[thick,->](0,0)--(1,-1);
\end{tikzpicture}
&
\begin{tikzpicture}[scale=.2, baseline=(current bounding box.center)]
\foreach \x in {-1,0,1} \foreach \y in {-1,0,1} \fill(\x,\y) circle[radius=2pt];
\draw[thick,->](0,0)--(0,1);
\draw[thick,->](0,0)--(-1,0);
\draw[thick,->](0,0)--(1,0);
\draw[thick,->](0,0)--(0,-1);
\draw[thick,->](0,0)--(-1,1);
\end{tikzpicture}
&
\begin{tikzpicture}[scale=.2, baseline=(current bounding box.center)]
\foreach \x in {-1,0,1} \foreach \y in {-1,0,1} \fill(\x,\y) circle[radius=2pt];
\draw[thick,->](0,0)--(0,1);
\draw[thick,->](0,0)--(-1,0);
\draw[thick,->](0,0)--(1,0);
\draw[thick,->](0,0)--(-1,-1);
\draw[thick,->](0,0)--(0,-1);
\end{tikzpicture}
&
\begin{tikzpicture}[scale=.2, baseline=(current bounding box.center)]
\foreach \x in {-1,0,1} \foreach \y in {-1,0,1} \fill(\x,\y) circle[radius=2pt];
\draw[thick,->](0,0)--(-1,1);
\draw[thick,->](0,0)--(0,1);
\draw[thick,->](0,0)--(1,0);
\draw[thick,->](0,0)--(-1,-1);
\draw[thick,->](0,0)--(1,-1);
\end{tikzpicture}
&
\begin{tikzpicture}[scale=.2, baseline=(current bounding box.center)]
\foreach \x in {-1,0,1} \foreach \y in {-1,0,1} \fill(\x,\y) circle[radius=2pt];
\draw[thick,->](0,0)--(-1,1);
\draw[thick,->](0,0)--(0,1);
\draw[thick,->](0,0)--(-1,0);
\draw[thick,->](0,0)--(1,0);
\draw[thick,->](0,0)--(-1,-1);
\draw[thick,->](0,0)--(1,-1);
\end{tikzpicture}
&
\begin{tikzpicture}[scale=.2, baseline=(current bounding box.center)]
\foreach \x in {-1,0,1} \foreach \y in {-1,0,1} \fill(\x,\y) circle[radius=2pt];
\draw[thick,->](0,0)--(-1,1);
\draw[thick,->](0,0)--(0,1);
\draw[thick,->](0,0)--(0,-1);
\draw[thick,->](0,0)--(1,0);
\draw[thick,->](0,0)--(-1,-1);
\draw[thick,->](0,0)--(1,-1);
\end{tikzpicture}
&
\begin{tikzpicture}[scale=.2, baseline=(current bounding box.center)]
\foreach \x in {-1,0,1} \foreach \y in {-1,0,1} \fill(\x,\y) circle[radius=2pt];
\draw[thick,->](0,0)--(-1,1);
\draw[thick,->](0,0)--(0,1);
\draw[thick,->](0,0)--(-1,0);
\draw[thick,->](0,0)--(1,0);
\draw[thick,->](0,0)--(-1,-1);
\draw[thick,->](0,0)--(0,-1);
\draw[thick,->](0,0)--(1,-1);
\end{tikzpicture}
&
\begin{tikzpicture}[scale=.2, baseline=(current bounding box.center)]
\foreach \x in {-1,0,1} \foreach \y in {-1,0,1} \fill(\x,\y) circle[radius=2pt];
\draw[thick,->](0,0)--(0,1);
\draw[thick,->](0,0)--(-1,0);
\draw[thick,->](0,0)--(0,-1);
\draw[thick,->](0,0)--(1,-1);
\end{tikzpicture}
\\
\begin{tikzpicture}[scale=.2, baseline=(current bounding box.center)]
\foreach \x in {-1,0,1} \foreach \y in {-1,0,1} \fill(\x,\y) circle[radius=2pt];
\draw[thick,->](0,0)--(1,0);
\draw[thick,->](0,0)--(-1,0);
\draw[thick,->](0,0)--(0,-1);
\draw[thick,->](0,0)--(-1,1);
\end{tikzpicture}
&
\begin{tikzpicture}[scale=.2, baseline=(current bounding box.center)]
\foreach \x in {-1,0,1} \foreach \y in {-1,0,1} \fill(\x,\y) circle[radius=2pt];
\draw[thick,->](0,0)--(0,1);
\draw[thick,->](0,0)--(-1,0);
\draw[thick,->](0,0)--(-1,-1);
\draw[thick,->](0,0)--(1,-1);
\end{tikzpicture}
&
\begin{tikzpicture}[scale=.2, baseline=(current bounding box.center)]
\foreach \x in {-1,0,1} \foreach \y in {-1,0,1} \fill(\x,\y) circle[radius=2pt];
\draw[thick,->](0,0)--(1,0);
\draw[thick,->](0,0)--(0,-1);
\draw[thick,->](0,0)--(-1,-1);
\draw[thick,->](0,0)--(-1,1);
\end{tikzpicture}
&
\begin{tikzpicture}[scale=.2, baseline=(current bounding box.center)]
\foreach \x in {-1,0,1} \foreach \y in {-1,0,1} \fill(\x,\y) circle[radius=2pt];
\draw[thick,->](0,0)--(-1,1);
\draw[thick,->](0,0)--(0,1);
\draw[thick,->](0,0)--(-1,-1);
\draw[thick,->](0,0)--(1,-1);
\end{tikzpicture}
&
\begin{tikzpicture}[scale=.2, baseline=(current bounding box.center)]
\foreach \x in {-1,0,1} \foreach \y in {-1,0,1} \fill(\x,\y) circle[radius=2pt];
\draw[thick,->](0,0)--(-1,1);
\draw[thick,->](0,0)--(1,0);
\draw[thick,->](0,0)--(-1,-1);
\draw[thick,->](0,0)--(1,-1);
\end{tikzpicture}
&
\begin{tikzpicture}[scale=.2, baseline=(current bounding box.center)]
\foreach \x in {-1,0,1} \foreach \y in {-1,0,1} \fill(\x,\y) circle[radius=2pt];
\draw[thick,->](0,0)--(-1,1);
\draw[thick,->](0,0)--(0,1);
\draw[thick,->](0,0)--(-1,0);
\draw[thick,->](0,0)--(1,-1);
\end{tikzpicture}
&
\begin{tikzpicture}[scale=.2, baseline=(current bounding box.center)]
\foreach \x in {-1,0,1} \foreach \y in {-1,0,1} \fill(\x,\y) circle[radius=2pt];
\draw[thick,->](0,0)--(-1,1);
\draw[thick,->](0,0)--(1,0);
\draw[thick,->](0,0)--(0,-1);
\draw[thick,->](0,0)--(1,-1);
\end{tikzpicture}
&
\begin{tikzpicture}[scale=.2, baseline=(current bounding box.center)]
\foreach \x in {-1,0,1} \foreach \y in {-1,0,1} \fill(\x,\y) circle[radius=2pt];
\draw[thick,->](0,0)--(0,1);
\draw[thick,->](0,0)--(-1,0);
\draw[thick,->](0,0)--(-1,-1);
\draw[thick,->](0,0)--(0,-1);
\draw[thick,->](0,0)--(1,-1);
\end{tikzpicture}
&
\begin{tikzpicture}[scale=.2, baseline=(current bounding box.center)]
\foreach \x in {-1,0,1} \foreach \y in {-1,0,1} \fill(\x,\y) circle[radius=2pt];
\draw[thick,->](0,0)--(1,0);
\draw[thick,->](0,0)--(0,-1);
\draw[thick,->](0,0)--(-1,-1);
\draw[thick,->](0,0)--(-1,0);
\draw[thick,->](0,0)--(-1,1);
\end{tikzpicture}
&
\begin{tikzpicture}[scale=.2, baseline=(current bounding box.center)]
\foreach \x in {-1,0,1} \foreach \y in {-1,0,1} \fill(\x,\y) circle[radius=2pt];
\draw[thick,->](0,0)--(-1,1);
\draw[thick,->](0,0)--(0,1);
\draw[thick,->](0,0)--(-1,-1);
\draw[thick,->](0,0)--(0,-1);
\draw[thick,->](0,0)--(1,-1);
\end{tikzpicture}
&
\begin{tikzpicture}[scale=.2, baseline=(current bounding box.center)]
\foreach \x in {-1,0,1} \foreach \y in {-1,0,1} \fill(\x,\y) circle[radius=2pt];
\draw[thick,->](0,0)--(-1,1);
\draw[thick,->](0,0)--(1,0);
\draw[thick,->](0,0)--(-1,-1);
\draw[thick,->](0,0)--(-1,0);
\draw[thick,->](0,0)--(1,-1);
\end{tikzpicture}
&
 \begin{tikzpicture}[scale=.2, baseline=(current bounding box.center)]
\foreach \x in {-1,0,1} \foreach \y in {-1,0,1} \fill(\x,\y) circle[radius=2pt];
\draw[thick,->](0,0)--(-1,1);
\draw[thick,->](0,0)--(0,1);
\draw[thick,->](0,0)--(-1,0);
\draw[thick,->](0,0)--(0,-1);
\draw[thick,->](0,0)--(1,-1);
\end{tikzpicture}
&
 \begin{tikzpicture}[scale=.2, baseline=(current bounding box.center)]
\foreach \x in {-1,0,1} \foreach \y in {-1,0,1} \fill(\x,\y) circle[radius=2pt];
\draw[thick,->](0,0)--(-1,1);
\draw[thick,->](0,0)--(1,0);
\draw[thick,->](0,0)--(-1,0);
\draw[thick,->](0,0)--(0,-1);
\draw[thick,->](0,0)--(1,-1);
\end{tikzpicture}
&
 \begin{tikzpicture}[scale=.2, baseline=(current bounding box.center)]
\foreach \x in {-1,0,1} \foreach \y in {-1,0,1} \fill(\x,\y) circle[radius=2pt];
\draw[thick,->](0,0)--(-1,1);
\draw[thick,->](0,0)--(0,1);
\draw[thick,->](0,0)--(-1,0);
\draw[thick,->](0,0)--(-1,-1);
\draw[thick,->](0,0)--(1,-1);
\end{tikzpicture}
&
 \begin{tikzpicture}[scale=.2, baseline=(current bounding box.center)]
\foreach \x in {-1,0,1} \foreach \y in {-1,0,1} \fill(\x,\y) circle[radius=2pt];
\draw[thick,->](0,0)--(-1,1);
\draw[thick,->](0,0)--(1,0);
\draw[thick,->](0,0)--(0,-1);
\draw[thick,->](0,0)--(-1,-1);
\draw[thick,->](0,0)--(1,-1);
\end{tikzpicture}
&
\begin{tikzpicture}[scale=.2, baseline=(current bounding box.center)]
\foreach \x in {-1,0,1} \foreach \y in {-1,0,1} \fill(\x,\y) circle[radius=2pt];
\draw[thick,->](0,0)--(-1,1);
\draw[thick,->](0,0)--(0,1);
\draw[thick,->](0,0)--(-1,0);
\draw[thick,->](0,0)--(-1,-1);
\draw[thick,->](0,0)--(0,-1);
\draw[thick,->](0,0)--(1,-1);
\end{tikzpicture}
\\
\begin{tikzpicture}[scale=.2, baseline=(current bounding box.center)]
\foreach \x in {-1,0,1} \foreach \y in {-1,0,1} \fill(\x,\y) circle[radius=2pt];
\draw[thick,->](0,0)--(-1,1);
\draw[thick,->](0,0)--(1,0);
\draw[thick,->](0,0)--(-1,0);
\draw[thick,->](0,0)--(-1,-1);
\draw[thick,->](0,0)--(0,-1);
\draw[thick,->](0,0)--(1,-1);
\end{tikzpicture}
&
\begin{tikzpicture}[scale=.2, baseline=(current bounding box.center)]
\foreach \x in {-1,0,1} \foreach \y in {-1,0,1} \fill(\x,\y) circle[radius=2pt];
\draw[thick,->](0,0)--(1,1);
\draw[thick,->](0,0)--(-1,0);
\draw[thick,->](0,0)--(-1,-1);
\draw[thick,->](0,0)--(0,-1);
\end{tikzpicture}
&
\begin{tikzpicture}[scale=.2, baseline=(current bounding box.center)]
\foreach \x in {-1,0,1} \foreach \y in {-1,0,1} \fill(\x,\y) circle[radius=2pt];
\draw[thick,->](0,0)--(0,1);
\draw[thick,->](0,0)--(1,1);
\draw[thick,->](0,0)--(1,0);
\draw[thick,->](0,0)--(-1,-1);
\end{tikzpicture}
& 
&
& 
&
&
&
&
&
&
&
&
&
&
\end{array}
\end{equation}

The ring $\CX[x,y]/(K(x,y,t))$ is an integral domain and we will denote its quotient field by $\CX(\Etproj)$. We will abuse notation and use $x$ and $y$ to denote the image of these variables in this field as well.  From the context it will be clear which sense is being used. 
\subsection{The Group} Since the polynomial $K(x,y)$  has degree $2$ in each variable, we can define two automorphisms of its zero set. Let $P = (a,b)$ satisfy $K(a,b) = 0$. The polynomial $K(a, y)$ has at most two roots $b, \tilde{b}$ (possibly $b = \tilde{b}$). We define $\iota_1(P) = (a,\tilde{b})$. Similarly, one can define $\iota_2(P) = (\tilde{a}, b)$ where $a, \tilde{a}$ are the roots of $K(x,b) = 0$. The maps $\iota_1, \iota_2$ are involutions which are induced by rational maps on $\CX \times \CX$ (formulas are given in \cite{BMM} and \cite{DHRS}) and so can be extended to involutions of $\Etproj(\CX)$, i.e., for any $P \in \Etproj$ we have
$$
\{P,\iota_1(P)\} = \Etproj \cap (\{x\} \times \P1(\C))
\text{ and }
\{P,\iota_2(P)\} = \Etproj \cap (\P1(\C) \times \{y\}).
$$
We furthermore define an an automorphism $\tau : \Etproj \rightarrow \Etproj$ by the formula
\[\tau := \iota_2\circ\iota_1. \]
\begin{defn} The {\rm group of the walk} is the group generated by $\iota_1, \iota_2$.\end{defn}
\begin{rmk}\label{rmk:automorphismonfunctionfield} The map $\iota_1$  induces an automorphism  of $\CX(\Etproj)$   via $\iota_1(f(Q)) = f(\iota_1(Q))$ for $Q \in \Etproj$ (we are abusing notation and using the same symbol for the map on $\Etproj$ and $\CX(\Etproj))$.  Similarly, $\iota_2$ and $\tau$ induce automorphisms of $\CX(\Etproj))$. One needs to be careful of the context when using these symbols.  In particular, $\tau = \iota_2\circ\iota_1$ on $\Etproj$ but $\tau = \iota_1\circ\iota_2$ on $\CX(\Etproj)$.\

The subfields $\CX(x)$ and $\CX(y)$ of $\CX(\Etproj)$ are pure transcendental extensions and are the fixed fields of $\ita$ and $\itb$ respectively. 
\end{rmk}

In {\cite{BBMR17}, the authors show that the group of the walk is finite if and only if there exists a nonconstant $g \in (\CX(x) \cap \CX(y))\subset \CX(\Etproj)$. When such a $g$ exists one says that the walk admits {\it invariants}. We give an equivalent property.

\begin{lem}\label{lem:inforder} 1.~The group $G$ of the walk is finite if and only if $\tau$ has finite order.\\
2.~The element $\tau$ has finite order if and only if there exists $f \in \CX(\Etproj)\backslash \CX$ such that $\tau(f) = f$.\\
3.~There exists a nonconstant $g \in (\CX(x) \cap \CX(y))\subset \CX(\Etproj)$ if and only if there exists $f \in \CX(\Etproj)\backslash \CX$ such that $\tau(f) = f$ \end{lem}
\begin{proof} 1.~This follows from the fact that the group generated by $\tau$ has index $2$ in the group of the walk.\\
2.~Assume $\tau$ has finite order $n$. Let $\CX(\Etproj)^\tau$ be the field of invariants of $\tau$.  For any $f \in \CX(\Etproj)$, the polynomial 
\[P_f(X) = \prod_{i=0}^{n-1} (X-\tau^i(f))\]
has coefficients in $\CX(\Etproj)^\tau$ and therefore any element of $\CX(\Etproj)$ is algebraic over $\CX(\Etproj)^\tau$. Since $\CX(\Etproj)$ has transcendence degree $1$ over $\CX$, there must be an element in $ \CX(\Etproj)^\tau\backslash \CX$.

Now assume that there exists an $f \in \CX(\Etproj)\backslash \CX$ such that $\tau(f) = f$. Since $\CX(\Etproj)$ has transcendence degree $1$ over $\CX$, $x$ and $y$ must be algebraic over $\CX(f)$. Let  $P_x(X)\in \CX(f)[X]$ (resp. $P_y(X)\in \CX(f)[X]$) be the {monic} minimal polynomial of $x$ (resp. $y$) over $\CX(f)$ and let $S_x =\{\alpha \in \CX(\Etproj) \ | \ P_x(\alpha) = 0\}$ and $S_y =\{\alpha \in \CX(\Etproj) \ | \ P_y(\alpha) = 0\}$. The automorphism  $\tau$ permutes the elements of $S_x$ and the elements of $S_y$. Since these sets are finite sets, there is some positive integer $n$ such that $\tau^n$ leaves all the elements of these sets fixed.  In particular, $\tau^n$ leaves $x$ and $y$ fixed and so must be the identity.\\
3.~ Of course, \cite[Theorem 7]{BBMR17} and 2.~above yield this equivalence but we give a direct proof. If $g \in (\CX(x)\cap \CX(y))\backslash \CX$ then $\ita(g) = g$ and $\itb(g) = g$, so $\tau(g) = \ita\itb(g) = \ita(g) =  g$. Conversely assume that $f \in \CX(\Etproj)\backslash \CX$ such that $\tau(f) = f$. We then have that $f$ is transcendental over $\CX$ so $x$ is algebraic over $\CX(f) \subset \CX(\Etproj)^\tau$. Let $P(X)\in \CX(\Etproj)^\tau[X]$ be the minimal polynomial of $x$ and denote by $P^{\ita}(X)$ the polynomial resulting from applying $\ita$ to the coefficients of $P(X)$. One sees that the coefficients of $P^{\ita}(X)$ again lie in $\CX(\Etproj)^\tau$ so we must have that $P^{\ita}(X) = P(X)$ since they both have $x$ as a root. Therefore the coefficients of $P(x)$ are left fixed by $\ita$ (as well as by $\tau$)  {and thus} lie in $\CX(x)$. Not all of these coefficients lie in $\CX$ since $x$ is not algebraic over $\CX$ so  there exists $g \in \CX(x)$ such that $\tau(g) = g$.  We then have that $g = \ita(g) = \ita(\tau(g)) = \itb(g)$ so $g\in (\CX(x) \cap \CX(y))\backslash \CX$. \end{proof}

We have the following additional facts concerning the group of the walk  and its relation to the kernel curve.

\begin{itemize}
\item For a dense set of values of $t \in [0,1]$, this group is  finite for $23$ unweighted models  (as well as some of these models with weights). These have been shown to have generating series that are holonomic (or even algebraic).  \cite{BMM, BostanKauersTheCompleteGenerating,BBMR17}. 
\item For a dense set of values of $t \in [0,1]$, this group is infinite for the remaining $56$ weighted models. Furthermore,
\begin{itemize}
\item for the $51$  models with associated curve of genus $1$, there exists a point $P \in \Etproj$ such that the element $\tau$ of the group is given by 
\[ \tau(Q) = Q  \oplus P\]
where $\oplus$ denotes addition on the elliptic cuve $\Etproj$ {[Proposition 2.5.2 in \cite{DuistQRT}]. If $\tau^n(Q)=Q$ for some point $Q \in \Etproj$ and some integer $n \in \Z$, the automorphism $\tau^n$ is the identity. The fact that the group is infinite is also equivalent to  the point $P$  having infinite order in the group structure on $\Etproj$.  }
\item  for the $5$ weighted models with associated curve of genus $0$, there exists a rational map $\phi : \PX^1(\CX) \rightarrow \Etproj$ such that the pullback of $\tau$ is a $q$-dilation $z\mapsto qz$ for some $q\in \CX, |q| \neq1$.  
\end{itemize}
\end{itemize}

 A remaining question is: for which values of the weights are the models attached to the set of steps \ref{stepgenus1} differentially algebraic or $D$-algebraic for short.  If the group of the walk is finite, \cite[Theorem 42]{DreyRasc} shows that the generating series is holonomic. 
When the group is infinite and the models unweighted, the question was solved  case by case in \cite{BBMR17} and \cite{DHRS}. In the next sections of this paper, we will show that the $D$-algebraicity of weighted models with  genus one kernel curve is encoded by the position of the  \emph{base points} of a pencil of elliptic curves. This  gives a more geometric understanding of the differential behavior of the weighted models  and allows one to  produce an algorithm to test their $D$-algebraicity. 
\section{Decoupling pairs and certificates}\label{certificate}  In this section we compare the criteria  presented in \cite{BBMR17} and \cite{DHRS} ensuring that the generating series of a quadrant model is $D$-algebraic. We shall assume that the curve $\Etproj$ defined by $K(x,y,t) = 0$ is  an irreducible curve. 


\subsection{Decoupling pairs}In \cite[Definition 8]{BBMR17}, the authors  introduce the notion of a  {\it decoupling}.

\begin{defn} A quadrant model is {\rm decoupled} if there exist $f(x) \in k(x)$ and $g(x) \in k(y)$ such that $xy = f(x) + g(y)$ in $k(\Etproj)$. The functions $f(x)$ and $g(y)$ are said to form a {\rm decoupling pair} for $h(x,y)$. \end{defn}  

A main result  of \cite{BBMR17} is that, of the $79$ relevant unweighted quadrant models,  precisely $13$ are decoupled.  Of these, $9$, as in Figure~\ref{figure:ex},  correspond to models with infinite group and an additional $4$ have finite group.   The authors further  show that  those  models admitting an invariant and  having a decoupling pair   are precisely  the models having algebraic generating series.   For the $9$ decoupled unweighted models with infinite group, the authors give explicit expressions for the generating series and show that these series are $D$-algebraic.

  {\mfs The strategy of \cite{BBMR17} is to give an explicit expression of the  generating series in terms of a certain \emph{weak invariant}, which is written in terms of the  elliptic functions. This explicit expression allows one to find explicit differential algebraic equation for the  generating series. The approach of \cite{BBMR17}  should also work   for decoupled weighted model.}
}

 Without being as explicit as \cite{BBMR17}, we can indicate why these expressions exist. 
  Note that when the kernel curve has genus one, {the elliptic curve $\Etproj$ admits an uniformization of the form $\{ (x(\omega),y(\omega)) \mbox{ with } \omega \in \C/(\Z\omega_1 +\Z \omega_2) \}$ where $\omega_1, \omega_2$ are two  $\Z$-linearly independent  complex numbers. The functions $x(\omega),y(\omega)$ are rational functions of the Weierstrass functions $\wp_{1,2}, \wp_{1,2}'$ attached to the elliptic curve $\C/(\Z\omega_1 +\Z \omega_2)$.  The automorphism $\tau$ then lifts to $\C$ as a translation $\omega_3$. By \cite{DreyRasc}, the   generating series $F^1(x,t)$ and $F^2(x,t)$ can be lifted to the universal cover of $\Etproj$ as meromorphic function denoted by $r_x(\omega)$ and $r_y(\omega)$.  When the model is decoupled, one can express
$r_x(\omega)$  in terms of elliptic functions as follows. }
\begin{lemma}\label{lem:analyticcontinuationdecoupledexpression}
Assume that the weighted model is decoupled and has  a genus one kernel curve and infinite group of the walk. Let $f(x)$ and $g(y)$ be a decoupling pair for $xy$. Then, there exist a unique rational function $G(X,Y) \in \C(X,Y)$ such that 
$$r_x(\omega)= f(x(\omega)) + G(\wp_{1,3}(\omega), \wp'_{1,3}(\omega)),$$
where $\wp_{1,3}$ is the Weierstrass function attached to the elliptic curve $\C/(\Z\omega_1 +\Z \omega_3)$. 
\end{lemma}
\begin{proof}
Since the group of the walk is infinite, the automorphism $\tau$ had infinite order and the complex number $\omega_3$  is $\Z$-linearly independent with $\omega_1$ so that they both form a  $\Z$-lattice in $\C$. An easy computation shows that 
$y(\omega +\omega_3)(x(\omega +\omega_3) -x (\omega))= f(x(\omega+\omega_3)) -f(x(\omega)).$ Since $f(x(\omega)$ is $\omega_1$-periodic, we deduce from the functional equations satisfied by $r_x(\omega)$ that $r_x (\omega) - f(x(\omega))$ is a meromorphic functions that is $\omega_1,\omega_3$-periodic. It is therefore an elliptic function with respect to the elliptic curve $\C//(\Z\omega_1 +\Z \omega_3)$. We conclude the proof via the characterization of elliptic functions in terms of Weierstrass functions.
\end{proof}


\subsection{Certificates}\label{sec:certificates} 
\begin{defn}\label{def:certificate} Let $K$ be a field, $\tau$ an automorphism of $K$ and $f \in K$.  We say that $g$ is a {\rm certificate} for $f$ if
\[ f = \tau(g) - g.\]
\end{defn}
This terminology comes from a similar term used in the theory of telescopers and certificates for deriving and verifying combinatorial identities \cite{BCCL10,WZ90}.
In \cite[Section 2.2]{DHRS19} the authors, using a result of Ishizaki \cite{Ishi}, show
\begin{prop} \label{cor:genus0} Assume that the kernel curve of a weighted quadrant model  $\Etproj$ has genus $0$ and has infinite group. $F^1(x,t) = -K(x,0,t)Q(x,0,t)$ and $F^2(y,t) = -K(0,y,t)Q(0,y,t)$ are $D$-algebraic if and only if   the element $b = x(\iota_1(y) - y) \in \CX(\Etproj)$  has a certificate in $\CX(\Etproj)$, i.e., there exists  $g\in \CX(\Etproj)$ such that 
\begin{align}
b &= \tau(g) - g.
\end{align}
\end{prop}
In the genus $1$ case and for unweighted models, the authors of \cite{DHRS}  proved a  slightly weaker result   {  The following proposition shows how this latter result can be reproduced word for word for weighted models.  We  will just  sketch the  proof since its only new ingredient relies on the uniformization results of \cite{DreyRasc} for weighted models, which  allows the direct use of the Galois theoretic tools of \cite{DHRS}.}

 If $D$ is a divisor of $\Etproj$, we will denote by $\calL(D)$ the finite dimensional $\CX$-space $\{ f \ | \ (f) + D \geq 0\}$ where $(f)$ is the divisor of $f$. Recall that there exists a point $P \in \Etproj(\CX)$ such that $\tau(Q) = Q \oplus P$ for all $Q\in \Etproj(\CX)$.

\begin{prop}\label{cor:genus1} Assume that the kernel curve $\Etproj$ of a weighted quadrant model  is of genus $1$ and has infinite group. We then have   that  $F^1(x,t) = -K(x,0,t)Q(x,0,t)$ and $F^2(y,t) = -K(0,y,t)Q(0,y,t)$ are $D$-algebraic if and only if   there exits  $g\in \CX(\Etproj)$, a $Q \in \Etproj$  and an $h \in \calL(Q + \tau(Q))$

such that
\begin{align}\label{eq:gen1dalg}
b &= \tau(g) - g + h.
\end{align}
where $b = x(\iota_1(y) - y) \in \CX(\Etproj)$.
\end{prop}
\begin{proof}
 Let us assume that $F^1(x,t)$ and $F^2(y,t)$ are  $D$-algebraic over $\C$
 By \cite[Proposition 2.8]{DreyRasc}, there exist 
$\omega_1,\omega_2$ two $\Z$-linearly independent  complex numbers and two meromorphic functions $x(\omega), y(\omega)$ such that the elliptic curve $\Etproj$ admits a uniformization of the form $\{ (x(\omega),y(\omega)) \mbox{ with } \omega \in \C/(\Z\omega_1 +\Z \omega_2) \}$. The functions $x(\omega),y(\omega)$ are rational functions of the Weierstrass functions $\wp_{1,2}, \wp_{1,2}'$ attached to the elliptic curve $\C/(\Z\omega_1 +\Z \omega_2)$. Therefore, $y(\omega)$ is $D$-algebraic over $\C$.
Evaluating \eqref{F1} on a certain open subset of $\Etproj$, the authors  of \cite{DreyRasc} were able to show that the function $F^2(y,t)$ can be lifted to a meromorphic function $r_y(\omega)$ on  the universal cover of $\Etproj$  such that 
\begin{itemize}
\item $r_y(\omega)$ coincides with $F^2(y(\omega),t)$ on a nonempty  open subset $\mathcal{D}_{x,y}$ (\cite[Lemma 24]{DreyRasc});
\item $r_y(\omega +\omega_1)= r_y(\omega)$
\item $r_y(\omega +\omega_3)= r_y(\omega) + b \circ (x( \omega),y(\omega))$ 
\end{itemize} where $\omega_3$ is a complex number such that the automorphism $\tau$ lifts to the universal cover as the translation by $\omega_3$\footnote{There is a discrepency in signs between \cite{DreyRasc} and us. We choose $F^1(x,t)= -Q(x,0,t)K(x,0,t)$ and they choose the opposite.}. By \cite[Lemma 6.3]{DHRS}, the function $r_y(\omega)$, which coincides with $F^2 \circ y(\omega)$ on some open set is $\omega$-$D$-algebraic. 
By \cite[Proposition  3.6 and Proposition B.5]{DHRS}, there exits  $g\in \CX(\Etproj)$, a $Q \in \Etproj$  and an $h \in \calL(Q + \tau(Q))$
such that $b = \tau(g) - g + h.$ Conversely, if $b=\tau(g)-g +h$ then \cite[Proposition B.5]{DHRS} implies the existence of $L \in \C[\frac{d}{d\omega}]$ such that $L (b \circ (x( \omega),y(\omega))=g(\omega +\omega_3) -g(\omega)$ for some $g \in \C(\Etproj)$, the latter field being identified 
with the field of meromorphic functions that are $(\omega_1, \omega_2)$-periodic. From the functional equations satisfied by $r_y$, one  obtains that the function $L(r_y)-g$ is $(\omega_1,\omega_3)$-periodic. Since elliptic functions are differentially algebraic over $\C$, the functions
$L(r_y)-g$ and $g$ are differentially algebraic over $\C$ and so is $r_y$. \cite[Lemma 6.4]{DHRS} allows one to conclude that, since 
$F^2(y,t)=r_y(y^{-1}(\omega))$ on some open set, the function $F^2(y,t)$ is $y$-$D$-algebraic over $\C$. By \cite[Proposition 3.10]{DHRS}, the function
$F^1(x,t)$ is also $x$-$D$-algebraic over $\C$.
\end{proof}

\begin{rem}  In \cite{dreyfus2019length}, the authors show that if a {weighted} quadrant model  has a generating series that is neither $x$- nor  $y$-$D$-algebraic, then the generating series is also $t$-$D$-transcendental. \end{rem}

In fact, one can further improve Proposition~\ref{cor:genus1} so that the condition \eqref{eq:gen1dalg} is replaced with the simpler {\it $b$ has a certificate in $\CX(\Etproj)$}, making the condition uniform for genus $0$ and $1$. 

 Note that $\iota_1(x) = x$ so for $b = x(\iota_1(y) - y)$, { one has $\iota_1(b) = -b$}. We refer to Appendix~\ref{appendix:PR} for the required facts concerning poles and residues.

\begin{lem}\label{cor:genus1b} Let $\Etproj$ be of genus $1$ and $b \in \CX(\Etproj)$ such that $\iota_1(b) = -b$. {Assume that the group of the walk is infinite.}  If there exist a $g\in \CX(\Etproj)$,  a $Q \in \Etproj$  and an $h \in \calL(Q + \tau(Q))$ such that
\begin{align}
b &= \tau(g) - g + h.
\end{align} 
then there exists a $\tilde{g} \in \CX(\Etproj)$ such that 
\begin{align}
b &= \tau(\tilde{g}) - \tilde{g}. 
\end{align}

\end{lem}
\begin{proof} Note that $\tau = \iota_1 \iota_2, \iota_1\tau = \iota_2$, and $\tau\iota_2 = {\iota_1}$ on $\CX(\Etproj)$ so
\begin{align}\label{eq:b}
2b= b-\iota_1(b) = \tau(g + \iota_2(g))  - (g + \iota_2(g)) + (h - \iota_1(h))
\end{align}
If $h \in \CX$, we have that $b = \tau(\tilde{g}) - \tilde{g}$ where $\tilde{g} = \frac{g+\ita(g)}{2}$.

If $h \notin \CX$, then it will be sufficient to prove that there exists an $\htilde \in \CX(\Etproj)$ such that $h- \ita(h) = \tau(\htilde) - \htilde$.  Lemma~\ref{resthm} implies that the configuration of poles and residues of $h$ is the following

$$
\begin{array}{|c|c|c|}
\hline
\mbox{Divisor} &  Q & \tau(Q) \\ \hline
\mbox{Residues of order 1} &  \alpha & -\alpha  \\ \hline
\end{array}
$$
for some $a\in \CX^*$.  Since $\ita$ is an involution of the curve, Lemma~\ref{lem:horesidues}.1 implies that the configuration of poles and residues of $-\ita(h)$ is
$$
\begin{array}{|c|c|c|}
\hline
\mbox{Divisor} &  \tau^{-1}(\ita(Q)) & \ita(Q) \\ \hline
\mbox{Residues of order 1} &  -\alpha & \alpha  \\ \hline
\end{array}
$$
If $\ita(Q) = \tau(Q)$, then  the function $\hat h = h - \ita(h)$ {has no poles and is therefore } constant. Note that $a$ may not equal  zero but the poles of $h$ and $\ita(h)$ cancel. Since $\hat h = -\ita(\hat h)$,{the constant} $\hat h$ must be zero and so from (\ref{eq:b}) we can conclude that $b = \tau(\gtilde) - \gtilde$ where $\gtilde = \frac{g+\itb(g)}{2}$. 

If $\ita(Q) \neq \tau(Q)$  the configuration of poles and residues of $h -\ita(h)$ is
$$
\begin{array}{|c|c|c|c|c|}
\hline
\mbox{Divisor} &  \tau^{-1}(\ita(Q)) & \ita(Q) & Q & \tau(Q)\\ \hline
\mbox{Residues of order 1} &  -\alpha & \alpha &\alpha &-\alpha \\ \hline
\end{array}
$$
The point $Q$ may coincide with $\ita(Q)$ and so { the residue there may be $2\alpha$} but this will not change the reasoning below.  Since  $\ita(Q) \neq \tau(Q)$, the Riemann-Roch Theorem implies that there exist an $f\in \CX(\Etproj)$ with simple poles at these points and whose configuration of poles and residues is
$$
\begin{array}{|c|c|c|}
\hline
\mbox{Divisor} &  \ita(Q) & \tau(Q) \\ \hline
\mbox{Residues of order 1} & - \alpha &\alpha  \\ \hline
\end{array}
$$
The configuration of poles and residues of $\tau(f) - f$ and of $h-\ita(h)$ are the same. Therefore $\hat h := h-\ita(h) = \tau(f) - f + d$ for some $d\in \CX$.  Since  $\ita(\hat h) = -\hat h$, we have, via an argument similar to the argument involving (\ref{eq:b}), that $\hat h = h - \ita(h) = \tau(\htilde) -\htilde$ where {$\htilde = \frac{ f + \itb(f)}{2}$}. \end{proof}

We therefore can give a uniform statement for the generating series of weighted quadrant models
\begin{thm}  Assume that the kernel curve of a weighted quadrant model  $\Etproj$   has infinite group. $F^1(x,t) = -K(x,0,t)Q(x,0,t)$ and $F^2(y,t) = -K(0,y,t)Q(0,y,t)$ are $D$-algebraic if and only if   the element $b = x(\iota_1(y) - y) \in \CX(\Etproj)$  has a certificate in $\CX(\Etproj)$, i.e., there exists  $g\in \CX(\Etproj)$ such that 
\begin{align}
b &= \tau(g) - g.
\end{align}
\end{thm}
\subsection{The relation between decoupling pairs and certificates}\label{sec:visavis} We now turn to showing that, for quadrant models with infinite group, being decoupled is equivalent to the existence of $g \in \CX(\Etproj)$ such that $x(\ita(y) - y) = \tau(g) - g$.  The following handles both the genus $0$ and genus $1$ cases in a uniform way.
\begin{prop}\label{prop:equiv} Assume that the quadrant model has an infinite group. The following are equivalent
\begin{enumerate}
\item The model is decoupled.
\item The element $b=x(\ita(y) - y)$ has a certificate in  $\CX(\Etproj)$.
\end{enumerate}
In fact, if $(f(x), g(y))$ is a decoupling pair for $xy$ then $g(y)$ is a certificate for $b$ and if $g$ is a certificate for $b$, then $(f= xy-g, g)$ is a decoupling pair for $xy$.
\end{prop}
\begin{proof} Recall that the fixed field of $\ita$ is  $\CX(x) \subset \CX(\Etproj)$  and the fixed field of $\itb$  is $\CX(y) \subset \CX(\Etproj)$.  

Assume {\it (1),} that the quadrant model  is decoupled.  We then have that 
\begin{align} \label{equiv:1}
xy = f(x) + g(y)
\end{align}
 for some $f(x) \in k(x)$ and $g(x) \in k(y)$. Applying $\ita$ to this equation, we have that  
 \begin{align}\label{equiv:2}
 x\ita(y) = f(x)  + \ita(g(y)). 
\end{align}
Subtracting (\ref{equiv:1}) from (\ref{equiv:2}) we have $x\ita(y) -xy = x(\ita(y) -y) = \ita(g(y)) - g(y)$. Since $\itb(g(y)) = g(y)$, we have
\begin{align}
x(\ita(y) -y) = \tau(g(y)) - g(y)
\end{align}
yielding {\it(2)}.

Now assume {\it(2)}, that there exists  $g \in \CX(\Etproj)$ such that $x(\ita(y) - y) = \tau(g) - g$. We let $b_1:=x(\ita(y) - y) = x(\tau(y) - y)$ and $b_2:= \tau(y)( \tau(x)-x)$. We then have 
\begin{align}
b_1 + b_ 2  &= \tau(y)( \tau(x)-x) + x(\tau(y) - y) + \tau(y)( \tau(x)-x)  = \tau(xy) - xy.
\end{align}
  We therefore have $b_2 = \tau(f) - f$ where $f = xy- g$.  We shall show that $f \in \CX(x)$ and $g\in \CX(y)$, which implies that {\it(1)} holds. 

To see that $f \in \CX(x)$, note that $\ita\itb\ita(b_2) = - b_2$. Combining this with $b_2 = \tau(f) - f$ yields
\begin{align*}
\ita(f) - \ita\itb\ita(f) &= f - \ita\itb(f).
\end{align*}
This implies that $\tau(\ita(f) - f) = \ita(f) - f$. Lemma~\ref{lem:inforder}.2 implies that $\ita(f) - f = c \in \CX$.  Applying $\ita$ to this last equation implies that $f - \ita(f) = c$ so $c=0$.  Therefore $f$ is left fixed by $\ita$ and so must belong to $\CX(x)$.

To see that $g \in \CX(y)$, note that $\ita(b_1) = -b_1$. Combining this with $b_1 = \tau(g) - g$, we have
\begin{align*}
g - \ita\itb(g) &= \itb(g) - \ita(g). 
\end{align*}
This implies that $\tau(\itb(g) - g) = \itb(g) - g$ so, as before $\itb(g) - g = c \in \CX$. Applying $\itb$ to this last equation implies $g - \itb(g) = c$ so $c = 0$.  Therefore $g$ is left fixed by $\itb$ and so must belong to $\CX(y)$. \end{proof}

\section{The orbit residue criterion}\label{sec:rescrit}

In  Section~\ref{sec:certificates} we reviewed and refined results from \cite{DHRS} and \cite{DHRS19}, to conclude  that to determine if a generating series of a quadrant model with infinite group is $x$- and $y$-$D$-algebraic it is enough to determine if the element $b = x(\ita(y) - y)$ has a certificate  $g \in \CX(\Etproj)$.    This condition is equivalent to the cancellation of the orbit residues of the function $b$ (see Proposition~\ref{Prop2}).   The definition of the orbit residues of $b$ involves the computation of  the poles of $b$ and their orbits with respect to $\tau$ as well as  various residues at these points.  Nevertheless, we show below that there are {\it a priori} criteria that allow us to avoid {these calculations}. In Proposition~\ref{prop:fixedbytwoinvolutions} we show that if the poles of $b$ behave in a certain way with respect to the involutions $\ita, \itb$ then the orbit residues are never zero.  In Proposition~\ref{prop:sameorbitQiequalsPj} and \ref{prop:sameorbitQjneqP_i}  we show that for the remaining cases $b$ has a certificate if  and only if two distinguished poles lie in the same $\tau$-orbit. This simplifies the application of Proposition~\ref{Prop2}   and is exploited in our considerations of weighted quadrant {walks} having $D$-algebraic generating series.\\

The potential poles of $b = x(\ita(y) - y)$ are the poles of $x, y,$ and $\ita(y)$ in $\mathbb{P}^1\times\mathbb{P}^1$:
\begin{itemize}
\item $P_i = (\infty, b_i)$ where $\infty = [1:0]$ and $b_i = [b_{i,0}, b_{i,1}]$, $i=0,1$,
\item $Q_i = (a_i, \infty)$ where  $a_i = [a_{i,0}, a_{i,1}]$, $i=0,1$,
\item $\ita(Q_i) = (a_i,c_i)$ where $c_i = [c_{i,0}, c_{i,1}]$, $i=0,1$. 
\end{itemize}
{In the rest of the paper, we make the following convention: the indexes of the points $P_i,Q_k, \ita(Q_l)$ have to be considered modulo $2$. For instance, if $Q_l=Q_1$ the point $Q_{l+1}$ corresponds to $Q_0$. }

\subsection{Symmetries and positions of the poles}
Note that $P_i = \ita(P_j)$ and $Q_i = \itb(Q_j)$ for $i \neq j$. We collect some useful facts concerning these points in the following Lemma. The notation $R\sim S$ for $R,S$ points of $\Etproj$ is used to denote the fact that there exists an $n \in \ZX$ such that $R = \tau^n(S)$.

\begin{lem}\label{lem:useful} \begin{enumerate}
\item $\ita(Q_i) = \tau^{-1}(Q_j)$ for $i\neq j$.
\item If $Q_i \sim P_j$ then $Q_{i+1} \sim P_{j+1}$.
\item If the point $Q_i$ is fixed by $\ita$ then $Q_i = P_j$ for some $j$  or $Q_i = (0,\infty) := ([0:1],[1:0])$.
\item If the point $P_i$ is fixed by $\itb$ then $P_i = Q_j$ for some $j$  or $P_i = (\infty,0) := ([1:0],[0:1])$.
\end{enumerate}
\end{lem}
\begin{proof}1.~The result follows from the facts that $\tau = \itb\ita$ and $Q_i = \itb(Q_j)$ with $i\neq j$.\\
2.~Note that $\ita\tau^n = \tau^{-n}\ita$. For simplicity, assume $i=j=1$. If $P_1 = \tau^n(Q_1)$, then $P_0 =\ita(P_1) = \tau^{-n}(\ita(Q_1)) = \tau^{-n-1}(Q_0)$ since $\ita(Q_1) = \tau^{-1}(Q_0)$.  \\%
3.~Since $\overline{K}(a_{i,0},a_{i,1}, y_0, y_1) = 0 $ has $y_1 = 0$ as a solution, we see that $\overline{K}(a_{i,0},a_{i,1}, y_0, y_1)$ has no $y_0^2$ term, that is, 
\begin{align*}
\overline{K}(a_{i,0},a_{i,1}, y_0, y_1)= &(a_{i,0}a_{i,1} -t\sum_{\ell=0}^{2} d_{\ell-1,0} a_{i,0}^{\ell}a_{i,1}^{2-\ell})y_0y_1\\ &+ t(\sum_{\ell=0}^2 d_{\ell-1,-1}a_{i,0}^{\ell} a_{i,1}^{2-\ell})y_1^2. \nonumber
\end{align*}

If $Q_i$ is fixed by $\ita$ this expression must have no  $y_0y_1$ term so 
$a_{i,0}a_{i,1} -t\sum d_{\ell-1,0} a_{i,0}^{\ell}a_{i,1}^{2-\ell} = 0$. 
%
Since  $t$ is transcendental over $\QX$, we have 
$a_{i,0}a_{i,1} = 0 = \sum d_{\ell-1,0} a_{i,0}^{\ell}a_{i,1}^{2-\ell}$ 

which implies that either $Q_i = P_j$ or $Q_i =  (0,\infty)$. 

Claim 4.~is entirely symmetric to 3.\end{proof}

We will use the following alternative expression for $b$ (c.f., \cite[Lemma 4.11]{DHRS}):
\begin{align}\label{eq:b2}
b^2&= \frac{x_0^2 \Delta_{[x_0:x_1]}^x}{x_1^2 (\sum_{i=1}^2 x_0^i x_1^{2-i}td_{i-1,1})^2}
\end{align}
where $\Delta_{[x_0:x_1]}^x$ is the discriminant of the polynomial $y \mapsto K(x_0,x_1,y,t)$,

\begin{align}\label{eq:disc}
\Delta^x_{[x_0:x_1]}= & t^2 \Big[ (d_{-1,0} x_1^2 -  \frac{1}{t} x_0x_1 + d_{1,0}x_0^2)^2  \\
&- 4(d_{-1,1} x_1^2 + d_{0,1} x_0x_1 + d_{1,1}x_0^2)(d_{-1,-1} x_1^2 + d_{0,-1} x_0x_1 + d_{1,-1}x_0^2) \Big].\nonumber
\end{align}

{Let us first give a symmetry argument which will allow us to simplify the enumeration of the distinct poles configurations. Let $d_{i,j}$ be a set of weights and let us denote by $K(x,y)$ the associated kernel polynomial and by $\Etproj$ the kernel curve. Let us consider now the polynomial $\widetilde{K}(\tx,\ty)=\tx\ty - t \sum_{i,j}d_{j,i}{\tx}^i{\ty}^j$ and the corresponding projective curve $\tEtproj$. These objects are obtained by exchanging the roles of  $x$ and $y$. Let us denote by 
$\widetilde{\ita},\widetilde{\itb},\widetilde{\tau}$ the horizontal, vertical switches and the automorphism of the walk on $\tEtproj$. Moreover, we denote by $\widetilde{b}$ the element of $\C(\tEtproj)=\C(\tx,\ty)$ defined by 
$\tx (\widetilde{\ita}(\ty)-\ty)$. The following holds.} 

\begin{lemma}\label{lemma:symmetryxandy}
The morphism $\phi :\Etproj \rightarrow \tEtproj, (a,b) \mapsto (b,a)$ is an isomorphism such that 
\begin{itemize}
\item $\widetilde{\itb} \circ \phi= \phi \circ \ita$,
\item $\widetilde{\ita} \circ \phi=\phi \circ \itb$,
\item ${\widetilde{\tau}^{-1} }\circ \phi = \phi \circ \tau$.
\end{itemize}
In particular $\Etproj$ is a curve of genus one and $\tau$ has infinite order if and only if $\tEtproj$ is a curve of genus one and $\widetilde{\tau}$ has infinite order. 
Moreover, $b$ has a certificate $g$ if and only if $\widetilde{b}$ has a certificate $\widetilde{g}$. 
\end{lemma}
\begin{proof}
The first part of the Lemma is obvious since the inverse of $\phi$ is given by ${\phi^{-1}}( (c,d))= (d,c)$. The equivalence is entirely symmetric so that one just has to prove one direction. Let us assume that 
$\widetilde{b}$ has a certificate $\widetilde{g}$, that is, \begin{equation}\label{eq:widetildebcertificates}
\widetilde{b}= \widetilde{\tau}(\widetilde{g})-\widetilde{g}.
\end{equation} The isomorphism $\phi$ induces an isomorphism   $\psi: \C(\tEtproj) \rightarrow \C(\Etproj), f \mapsto f \circ \phi$ such that $\iota_1 \circ \psi=\psi \circ\widetilde{\itb}$, $\itb \circ \psi =\psi \circ \widetilde{\ita}$ and 
${ \tau^{-1}} \circ \psi =\psi \circ \widetilde{\tau}$.  Applying $\psi$ to \eqref{eq:widetildebcertificates} and noting that 
$\psi(\tx)= y$ and $\psi(\ty)=x$ yields  
\begin{align*}
\psi({\widetilde{b}}) =  \psi(\tx)(\psi\widetilde{\ita}(\ty) - \psi(\ty) &   = \psi\widetilde{\tau}(\widetilde{g})-\widetilde{g} \\  
= y (\itb(x)-x)& ={ \tau^{-1}} (\psi(\widetilde{g})) -\psi(\widetilde{g}).
\end{align*}
Setting ${g=-\tau^{-1}(\psi(\widetilde{g}))}$, one finds $y(\iota_2(x)-x)= \tau(g)-g$. We apply $\iota_1$ to the latter equation and, noting that $\tau=\iota_1 \iota_2$ by Remark \ref{rmk:automorphismonfunctionfield} find
$$\iota_1(y)(\tau(x)-x)=\iota_1 \tau (g) -\iota_1(g)=\iota_2(g) -\iota_1(g)= \tau(h)-h,$$ 
where $h=-\itb(g)$. Thus the function $c=\iota_1(y)(\tau(x)-x)$ has a certificate. Since  $b+c=\tau(xy)-xy$, we conclude that 
$b$ has also a certificate. This ends the proof.
\end{proof}

\subsection{The involutions}

In this section, we study the behavior of the orbit residues {of $b$ } when its poles are fixed by involutions. Our  proof proceeds by considering the various configurations and orders of the poles. To do this one  determines the order of vanishing of the numerators and denominators of the expression on the right hand side of \eqref{eq:b2}. Useful facts for carrying out this task are:
\begin{itemize}
\item As noted in the proof of Lemma~\ref{lem:useful}(2),  at the points $Q_i = (a_i, \infty)$ where $y_0 = 0$, we have that $\sum_{i=1}^2 x_0^i x_1^{2-i}td_{i-1,1}$ vanishes. If $Q_0 = Q_1$, we have that his latter expression has a double zero.
\item  If we have a point $R$ where $\ita(R) = R$, then $\Delta_{[x_0:x_1]}^x = 0$ at this point.  In particular this happens when $P_1 = P_0$ or $Q_i = \ita(Q_i)$.  Furthermore, at this point one has ramification and the order of $x = [x_0:x_1]$ is $2$.
\end{itemize}
In what follows we will state the polar divisor $(b)_\infty$ and residue configurations and rely on the reader to do the simple verification using the facts.
\begin{prop}\label{prop:fixedbytwoinvolutions}
Assume that  $\Etproj$ is a curve of genus one and that  the automorphism of the walk is  not of finite order. 
If one of the $P_i$'s and one of the $Q_j$'s is fixed by an involution then  the function $b$ has no certificate. 
\end{prop}

\begin{proof}
By Proposition \ref{Prop2}, the function $b$ has no certificate if and only if one of its orbit residues is non-zero. We shall frequently use the fact that since  $\tau$  has infinite order, if  $\tau^n(Q)=Q$ for some point $Q$ then $n = 0$.  This follows from the fact that $\tau(Q) = Q\oplus P$ where $P$ has infinite order in the groups structure  on $\Etproj$ (see the remarks following Lemma~\ref{lem:inforder}).\\

We now use a case-by-case argument to prove this proposition.\\

\noindent \textit{Case a: $P_j$ is fixed by $\ita$ and $Q_i$ is fixed by $\ita$.} \\[0.1in]
By Lemma \ref{lem:useful}, we find that either $Q_i=P_0=P_1$ or $Q_i=(0,\infty)$. Moreover, $Q_i \neq Q_{i+1}$ since otherwise $\tau(Q_i)=Q_i$ and $\tau$ would be the identity. 
 \begin{itemize}  
 \item \textit{ Case a.1: $Q_i=P_0=P_1$.}  Then, the polar divisor $(b)_\infty$ of $b$ is $3P_1 +\epsilon Q_{i+1} + \epsilon \tau^{-1}(Q_i)$ where $\epsilon$ is zero if $Q_{i+1}=(0,\infty)$ and { otherwise $\epsilon = 1$}. It is easily seen that the orbit  residue of order $3$ of $P_1$ is never zero. 
 \item  \textit{ Case a.2: $P_0=P_1$ and $Q_i=(0,\infty) \neq Q_{i+1}$.} In that situation, $Q_{i+1}= \tau(Q_i)$ and $\iota_1(Q_{i+1})=\tau^{-1}(Q_i)$, Lemma \ref{lem:horesidues} allows one  to show that  the residues of $b$ are as follows
 $$
\begin{array}{|c|c|c|c|}
\hline
\mbox{Points} &  P_0 &\tau(Q_i)& \tau^{-1}(Q_i)    \\ \hline
\mbox{Residues of order 1} &  \a & \b &\b  \\ \hline
\end{array}
$$
with $\a +2\b=0$ and $\a,\b \neq 0$. Then, the orbit residues of $b$ are all zero if and only if $P_0 \sim Q_i$. This last condition will never happen. Suppose to the contrary that $P_0=\tau^n(Q_i)$ then $\ita(P_0)=P_0 =\tau^{-n}(\ita(Q_i))=\tau^{-n}(Q_i)$. Thus $\tau^{2n}(Q_i)=Q_i$ which implies $n=0$. This is absurd since $Q_i=(0, \infty)$ and $P_0=(\infty, [b_{0,0}:b_{0,1}])$.\\
 \end{itemize}
\textit{Case b: $P_j$ is fixed by $\itb$ and $Q_i$ is fixed by $\itb$.} \\
This case is symmetric with Case a by exchanging $x$ and $y$. Lemma \ref{lemma:symmetryxandy} allows to conclude that $b$ has no certificate in that case either.  \\[0.1in]
\textit{Case c: $Q_j$ fixed by $\iota_2$ and $P_i$ fixed by $\ita$}\\
In that case, note that $P_0=P_1$ and $Q_0=Q_1$. Moreover, since $\tau$ is not the identity, one has $Q_0 \neq P_0$.
 Lemma \ref{lem:horesidues} allows to show that $(b)_\infty=P_0 +\epsilon Q_0 + \epsilon \tau^{-1}(Q_0)$ with $\epsilon=1 $ if $Q_0=(0, \infty)$ and $\epsilon=2$ if $Q_0 \neq (0,\infty)$. Thus,  the residues of $b$ are as follows
 $$
\begin{array}{|c|c|c|c|}
\hline
\mbox{Points} &  P_0 &Q_0& \tau^{-1}(Q_0)    \\ \hline
\mbox{Residues of order 1} &  \a & \b &\b  \\ \hline
\mbox{Residues of order 2} & 0 & \c& - \c \\ \hline
\end{array}
$$
with $\a +2\b=0$ and $\a \neq 0$, $\b \neq 0$ if  $Q_0 = (0, \infty)$ and  $\c \neq 0$ if and only if $Q_0 \neq (0,\infty)$. Thus the orbit residues are zero if and only if $P_0 \sim Q_0$. The latter condition is never true. Indeed, if $P_0=\tau^n(Q_0)$ then $$\ita(P_0)=P_0=\tau^{-n}(\ita(Q_0))=\tau^{-n}(\ita( \itb(Q_0))=\tau^{-n-1}(Q_0)=\tau^n(Q_0).$$ Since $\tau$ is of infinite order, we must have $n=-n-1$ which is absurd since $n \in \Z$.\\[0.1in]
\textit{Case d: $Q_j$ fixed by $\ita$ and $P_i$ fixed by $\itb$.}\\
Using Lemma \ref{lem:useful}, we see that if $Q_j$ is fixed by $\ita$ then $Q_j=(0,\infty)$ or $(\infty,\infty)$. Moreover, if  $P_i$ is  fixed by $\itb$ then $P_i=(\infty,0)$ or $(\infty,\infty)$. Some of these possibilities will never occur:
\begin{itemize}
\item if $P_i=(\infty,\infty)$ is fixed by $\iota_2$ then $P_i=Q_0=Q_1$. Thus, none of the $Q_j$'s can be fixed by $\ita$. Otherwise, $P_i$ would be fixed by $\tau$.
\item if $P_i=(\infty,0)$ is fixed by $\itb$ then $P_{i+1} \notin \{Q_j, Q_{j+1} \}$. Indeed if $P_{i+1}=Q_j$ then $P_{i+1}=P_i=Q_j$ because $Q_j$ is fixed by $\ita$. This is absurd since $P_i=(\infty,0)$ and $Q_j=(a,\infty)$. If $P_{i+1}=Q_{j+1}$ then $\tau^3(Q_j)=Q_j$ which is absurd since $\tau$ has infinite order. 
\end{itemize}
Thus the only possibility is $Q_j=(0,\infty)$ fixed by $\ita$, $Q_{j +1} \notin \{P_i,P_{i+1} \}$ and $P_i=(\infty,0)$ fixed by $\itb$. The polar divisor of $(b)_\infty$ is $P_0+P_1+ \tau(Q_j) +\tau^{-1}(Q_j)$ and using Lemma \ref{lem:horesidues}, one gets 
$$
\begin{array}{|c|c|c|c|c|}
\hline
\mbox{Points} &  P_0 &P_1 &\tau(Q_j)& \tau^{-1}(Q_j)    \\ \hline
\mbox{Residues of order 1}& \a &  \a & \b &\b  \\ \hline

\end{array}
$$
where $2\a+2\b=0$ and $\a,\b \neq 0$.
Noting that $P_0 \sim Q_j$ if and only if $P_{1} \sim Q_{j}$, one sees that $b$ has orbit residues zero (in one or two orbits) if and only if $P_i \sim Q_j$. The latter condition is never true. Indeed, if $Q_j =\tau^n(P_i)$ then $\ita(Q_j)=Q_j=\tau^{-n}(\ita(P_i))=\tau^{-n-1}(P_i)=\tau^n(P_i)$. Since $\tau$ is not of finite order, we must have $n=-n-1$. Absurd since $n \in \Z$.
\end{proof}

\subsection{Remaining cases}
In this section, we shall consider the cases where  one of  the $P_i$'s and one  the $Q_j$'s are not simultaneously fixed by an involution. We shall prove that $b$ has orbit residues zero if and only if two precise points of the polar divisor  are in the same orbit.

 We distinguish two cases: $d_{1,1}= 0$ and $d_{1,1}\neq 0$. They corresponds to the fact that the point $(\infty,\infty)$ belongs to the curve or not.

\begin{prop}\label{prop:sameorbitQiequalsPj}
Assume that $d_{1,1}=0$, $\Etproj$ is a genus one curve and $\tau$ is of infinite order. Assume moreover that one of the $P_i$'s and one of the $Q_j$'s are not simultaneously fixed by an involution.  Then,
$b$ has a certificate if and only if $P_0 \sim P_1$.
\end{prop}

\begin{proof}
{Note that  $P_j = Q_k = (\infty,\infty)$ for some $j,k$}. {Moreover since we assume that one of the $P_i$'s and one the $Q_j$'s are not simultaneously fixed by an involution, we have $P_0 \neq P_1$ and $Q_0 \neq Q_1$.}  We shall prove the statement case by case according to the  configuration of poles of $b$.  \\[0.1in]
\textit{Case a. $P_j=Q_k$ and nothing else:}
Then the polar divisor $(b)_\infty$ of $b$ equals $2P_j +2P_{j+1} + { \tau(P_{j+1})} +  \tau^{-1}(P_j)$. Lemma \ref{lem:horesidues} shows that  the residues of 
$b$ are as follows
$$
\begin{array}{|c|c|c|c|c|}
\hline
\mbox{Points} &  P_j &P_{j+1} & {\tau(P_{j+1})} & \tau^{-1}(P_j)   \\ \hline
\mbox{Residues of order 1}& \a &  \a & \b &\b  \\ \hline
\mbox{Residues of order 2}& \c &  -\c & 0 &0  \\ \hline
\end{array}
$$
with $2a+2b=0$, $b,c \neq 0$. Then the orbit residues are zero if and only if $P_j \sim P_{j+1}$. \\ 

\noindent\textit{Case b. $P_j=Q_k$ and $Q_{k+1}=(0, \infty)$}

\begin{itemize}
\item  \textit{Case b.1: and nothing else}
Then the polar divisor $(b)_\infty$ of $b$ equals $2P_j +2P_{j+1}$. Lemma \ref{lem:horesidues} shows that  the residues of 
$b$ are as follows
$$
\begin{array}{|c|c|c|}
\hline
\mbox{Points} &  P_j &P_{j+1}   \\ \hline
\mbox{Residues of order 1}& \a &  \a  \\ \hline
\mbox{Residues of order 2}& \c &  -\c  \\ \hline
\end{array}
$$
with $2\a=0$, $\c \neq 0$. Then the orbit residues are zero if and only if $P_j \sim P_{j+1}$.

  \item \textit{Case b.2: and $Q_{k+1}$ is fixed by $\iota_1$} Note that $P_0 \neq P_1$ by the above. Moreover, since $Q_{k+1}$ is fixed by $\iota_1$, we get that $P_j=\tau^2(P_{j+1})$ so that $P_0 \sim P_1$. The divisor is the same than in Case b.1. and and since $P_0$ is in the same orbit than $P_1$, the orbit residues are always zero.  
   
\end{itemize}

{There are no other cases since the remaining configurations will correspond to the situations where one the  the $P_i$'s and one the $Q_j$'s are  simultaneously fixed by an involution.}\end{proof}
\begin{rmk}
In the proof of Proposition \ref{prop:sameorbitQiequalsPj}, we prove that if $P_j=Q_k$ and $Q_{k+1}=(0,\infty)$ is fixed by $\iota_1$ the function $b$ always has  a certificate. This corresponds to walks where the directions North East, North West and West do not belong to the steps set. The models of such walks are as follows
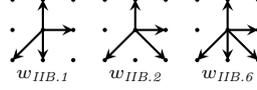
\begin{figure}[h!]
$$
\underset{\walk{IIB.1}}{\begin{tikzpicture}[scale=.4, baseline=(current bounding box.center)]
\foreach \x in {-1,0,1} \foreach \y in {-1,0,1} \fill(\x,\y) circle[radius=2pt];
\draw[thick,->](0,0)--(0,1);
\draw[thick,->](0,0)--(1,0);
\draw[thick,->](0,0)--(-1,-1);
\draw[thick,->](0,0)--(0,-1);
\end{tikzpicture}}
\quad 
\underset{\walk{IIB.2}}{\begin{tikzpicture}[scale=.4, baseline=(current bounding box.center)]
\foreach \x in {-1,0,1} \foreach \y in {-1,0,1} \fill(\x,\y) circle[radius=2pt];
\draw[thick,->](0,0)--(0,1);
\draw[thick,->](0,0)--(1,0);
\draw[thick,->](0,0)--(-1,-1);
\draw[thick,->](0,0)--(1,-1);
\end{tikzpicture}
}
\quad
\underset{\walk{IIB.6}}{\begin{tikzpicture}[scale=.4, baseline=(current bounding box.center)]
\foreach \x in {-1,0,1} \foreach \y in {-1,0,1} \fill(\x,\y) circle[radius=2pt];
\draw[thick,->](0,0)--(0,1);
\draw[thick,->](0,0)--(1,0);
\draw[thick,->](0,0)--(-1,-1);
\draw[thick,->](0,0)--(0,-1);
\draw[thick,->](0,0)--(1,-1);
\end{tikzpicture}}
$$
\end{figure}

That is we prove that among the $9$  models of walks that were differentially algebraic when unweighted, the three models above remain differentially algebraic with weights. \end{rmk}

\begin{prop}\label{prop:sameorbitQjneqP_i} 
Assume that $d_{1,1}\neq 0$, $\Etproj$ is a genus one curve and $\tau$ is of infinite order. Assume moreover that the $P_i$'s and the $Q_j$'s are not simultaneously fixed by an involution.  Then,
$b$ has a certificate if and only if $P_j \sim Q_k$.

\end{prop}
\begin{proof}
Since $d_{1,1}\neq 0$, the sets $\{P_0,P_1 \}$ and $\{Q_0,Q_1\}$ have empty intersection. \\

\textit{ Case a.: Assume the  six points $P_i, Q_i, \ita(Q_i), \ i=0,1$ are all distinct}\\

\begin{itemize}\item \textit{ Case a.1: and $Q_i \neq (0,\infty)$:} Then,  $(b)_\infty = P_0 + P_1 +Q_0 + \tau^{-1}(Q_1) + Q_1 + \tau^{-1}(Q_0).$\\[0.05in]
  Since $\ita(b) = -b$, Lemma~\ref{lem:horesidues} implies that  the residues are given by
  $$
\begin{array}{|c|c|c|c|c|c|c|}
\hline
\mbox{Points} &  P_0 &P_1& Q_0 & \tau^{-1}(Q_1) & Q_1 & \tau^{-1}(Q_0) \\ \hline
\mbox{Residues of order 1} &  \a & \a &\b&\b &\c &\c  \\ \hline
\end{array}
$$
with $2 \a +2 \b +2 \c=0$ and $\a,\b,\c \neq 0$.
 Assume that all  the orbit residue are  zero. Since $\a\neq 0$ the set $\{P_0,P_1\}$ cannot form a single $\tau$-orbit. Therefore $P_i \sim Q_j$ for some $i,j$. Conversely assume that $P_i \sim Q_j$.  Then,  by Lemma \ref{lem:useful}2.), we have  $P_{i+1} \sim Q_{j+1}$ .   We then have that either there are two $\tau$-orbits $\{P_{i +\epsilon}, Q_{j+ \epsilon}, \tau^{-1}(Q_{j+ \epsilon})\}, \epsilon =0,1$, each of whose orbit residues are $\a+\b+\c=0$ or there is one $\tau$-orbit $\{P_0,P_1, Q_0 , \tau^{-1}(Q_1) , Q_1 , \tau^{-1}(Q_0)\}$ whose orbit residue is $2\a+2\b+2\c=0$.  Thus the orbit residues are all zero. 

\item \textit{Case a.2: $Q_i = (0,\infty)$:} For simplicity assume $Q_1 = (0,\infty)$. In this case,$[0:1]$ is a double zero of both the numerator and denominator of \eqref{eq:b2} so $Q_1$ and $\ita(Q_1)$ are not poles. Therefore 
$(b)_\infty = P_0 + P_1 + Q_0 + \tau^{-1}(Q_1).$
  Since $\ita(b) = -b$, Lemma~\ref{lem:horesidues}.2 implies that  the residues are given by
  $$
\begin{array}{|c|c|c|c|c|}
\hline
\mbox{Points} &  P_0 &P_1 & Q_0 & \tau^{-1}(Q_1) \\ \hline
\mbox{Residues of order 1} &  \a &\a & \b & \b    \\ \hline
\end{array}
$$

 One easily modifies the argument above to prove the Proposition in this case.\\  
\end{itemize}

We now examine all of the cases when at least two of the putative poles coincide. Notice that we always have that $\tau^{-1}(Q_i)\neq Q_i$  since $\tau$ has infinite order (see the remarks following Lemma~\ref{lem:inforder}).\\

\textit{Case b: $Q_0 = Q_1$}\\

\begin{itemize}
\item \textit{Case b.1: and $P_0, P_1, Q_1, \ita(Q_1)$ distinct; $Q_1 \neq (0,\infty)$:}  $(b)_\infty = P_0 + P_1 + 2Q_1 + 2\tau^{-1} (Q_0).$
Lemma~\ref{lem:horesidues} implies that the configuration of residues is
$$
\begin{array}{|c|c|c|c|c|}
\hline
\mbox{Points} &  P_0 &P_1& 2Q_1 & 2\tau^{-1}(Q_0)   \\ \hline
\mbox{Residues of order 1} &  \alpha & \alpha &\beta&\beta   \\ \hline
\mbox{Residues of order 2} &     &    & \gamma & -\gamma \\ \hline  
\end{array}
$$
with $2 \alpha +2 \beta  =0$ and $\alpha, \gamma \neq 0$. 
If the orbit sums are zero then $\{P_0, P_1\}$ cannot be an orbit so for some $i,j$ some $P_i \sim Q_j$ ($P_i \neq Q_j$ by assumption). If $P_i\sim Q_j$ then, since $Q_0=Q_1$, Lemma~\ref{lem:useful} implies that all the poles must lie in the same orbit. Lemma~\ref{lem:horesidues} implies that all orbit sums are zero.

\item \textit{Case b.2: and $P_0, P_1, Q_1, \ita(Q_1)$ distinct; $Q_1 = (0,\infty)$:}  $(b)_\infty = P_0 + P_1 + Q_1 + \tau^{-1} (Q_1).$
$$
\begin{array}{|c|c|c|c|c|}
\hline
\mbox{Points} &  P_0 &P_1& Q_1 & \tau^{-1}(Q_1)   \\ \hline
\mbox{Residues of order 1} &  \alpha & \alpha &\beta&\beta  \\ \hline
\end{array}
$$
The argument is similar to Case a.1).\\

\end{itemize}

\textit{Case c :$P_0=P_1$} This case is obtained by symmetry exchanging $x$ and $y$ from Case b. Lemma \ref{lemma:symmetryxandy} allows to conclude. Note that the condition $P_i \sim Q_j$ becomes $Q_i \sim P_j$. That is, this condition remains unchanged by symmetry.  \\

\textit{Case d : $Q_i=(0, \infty)$} \\
\begin{itemize}
\item \textit{Case d.1: and nothing else:} This is a.2.
\item \textit{Case  d.2: and $Q_i$ fixed by $\iota_1$} The divisor is the same than in a.2. 
\item \textit{Case d.3: and $Q_{i+1}$ is fixed by $\ita$} This case can not occur. Indeed, Lemma \ref{lem:useful} implies that $Q_{i+1}=P_l$ or $Q_{i+1}=(0,\infty)$. The first case contradicts the assumption $d_{1,1} \neq 0$ whereas the second implies $\tau(Q_i)=Q_i$ which is in contradiction with the fact that the curve has genus $1$.
\item \textit{Case d.4: and $P_0$ is fixed by $\ita$}: Assume $Q_1 = (0,\infty)$. In this case $d_{-1,1} = 0$ so $b$ does not have a pole at $Q_1$.  $(b)_\infty = P_0 + Q_0+ \ita(Q_0)$. 
$$
\begin{array}{|c|c|c|c|c|}
\hline
\mbox{Points} &  P_0 & Q_0 &\ita(Q_0)   \\ \hline
\mbox{Residues of order 1} &  \alpha &\beta&\beta   \\ \hline
\end{array}
$$ 
with $\alpha\beta \neq 0$. Lemma~\ref{lem:horesidues} implies that $\alpha+2\beta = 0$. If the orbit residues are zero, then we must have all the poles in the same orbit, so $P_0 \sim Q_0$. If $P_0 \sim Q_0$, then $P_0 = P_1 \sim Q_1 \sim \tau^{-1}(Q_1) = \ita(Q_0)$, so all the poles are in the same orbit. If $P_0 \sim Q_1$, then $P_0 \sim \tau^{-1}(Q_1) = \ita(Q_0)$ so all the poles are in the same orbit and the orbit sum is zero.
\end{itemize}
{ Note that many cases disappear because we avoid having a $Q_i$ and a $P_j$ fixed simultaneously by two involution and also avoid one of the $Q_i$ equaling one of the $P_j$.}
\end{proof}

\section{Determining weights for which the generating series are  $D$-algebraic.}\label{sec:Dalg}

 In Section~\ref{sec:rescrit}, we show that either $b$ has no certificate or that the existence of a certificate is equivalent to two special points being in the same $\tau$-orbit. In this section we will describe an algorithm and its refinements to decide the question of two such points being in the same orbit. 
 
 The algorithm and its refinement are based on well known tools developed in arithmetic algebraic geometry to study  { elliptic surfaces, that is, families of elliptic curves}. In particular the Neron-Tate height $\hat{h}$ on elliptic curves $E$ over function fields $k$\footnote{ for instance $k =\C(t)$}, is the crucial ingredient.  This is a function $\hhat: E(k) \rightarrow {\R}$ one of whose properties is that if $P,Q \in E(k)$ and $Q = nP, n\in \ZX$ (which means that $Q$ is the $n$-multiple of $P$ with respect to the group law defined on $E(k)$) then $\hhat (Q) = n^2 \hhat(P)$.  In Section~\ref{sec:algo}, we describe how  the question of determining if points $P$ and $Q$ lie in the same $\tau$-orbit can be reduced to deciding if some point is a multiple of another point.

  {For fixed values of the weights, {\mfss  the Sage Package {\tt comb\_walks} (see \cite{JPBCL}) allows one to calculate $nP$ for fixed integers $n$ and points $P\in E$ as well as the necessary ancillary objects. In addition an implemented algorithm in MAGMA computes exactly the height of a point $P$ and so, for fixed weights, one can calculate if $P$ and $Q$ lie in the same orbit.}  However, our goal is to characterize the $D$-algebraicity of a weighted model in terms of a set of polynomial equations on the weights. Therefore, we need to  unravel the height computation. {The height of a point $P$ is given by a formula \eqref{NT} involving  certain numerical data associated with $E$, $P$ and $Q$. In Section~\ref{sec:algo} we show how one can determine $\hhat (Q), \hhat(P)$ up to a finite number of possibilities by  estimating these numerical data using  the celebrated Tate algorithm (calculating the  Weierstrass equation equation for $\Etproj$ and deducing certain properties from  tables produced by Tate)   as well as estimating the other numerical data by further consulting   tables produced by Kodaira, N\'eron, Oguiso and Shioda.  From the possible values of $\hhat (Q), \hhat(P)$, we can determine a finite set of possible $n$ with $\hhat (Q) = n^2 \hhat(P)$. A computation then allows one to determine which values of $n$ (if any) imply $Q = \tau^n(P)$ for some integer $n$. We emphasize that thanks to the deep work of those authors, once the Weierstrass equation is determined, only simple arithmetic is required to carry out this algorithm.}}
  
 {The key object lying behind these calculations is an elliptic surface associated with $E$.
 In Section~\ref{sec:refine},  we}  {construct this elliptic surface by blowing up the base points of the pencil of elliptic curves attached to $\Etproj$ and use it to refine the algorithm of Section~\ref{sec:algo}. This point of view emphasizes the importance of the relative position of the base points in the study of the $D$-algebraicity of the weighted model and also allows one 
 one to reduce drastically the number of possible values of $n$ as well as other information related to the mapping $\tau$.}
 
 
\subsection{An algorithm}\label{sec:algo}  So far, we have considered the kernel of the walk as defining, for a fixed  $t\in \CX$, transcendental over $\QX$ a curve  $\Etproj \subset \PX^1(\CX) \times \PX^1(\CX)$. The algorithm described in this and the next section depends on another object associated with the kernel. We now consider $t$ as a variable and consider $\Etproj$ as an elliptic curve defined over the field $\CX(t)$. The group law of this elliptic curve is defined over $\CX(t)$ and  and we consider the maps $\ita,\itb,\tau$ as automorphisms of $\Etproj$. We will make use of the {\it Kodaira-N\'eron {model}} $\Scal$ associated to $\Etproj$ {(see \cite[Def. 5.18 and 5.2 and Proposition 5.4]{SchuttShiodaBook} for the most recent reference on the subject but also \cite{DuistQRT,OgSh91, Shioda90,Silverman94} as general references)}. {In Section~\ref{sec:refine} we will give a description of the construction of  $\Scal$ as well as a more precise explanation of its properties  but for this algorithm we will only need the following properties:}

 \begin{enumerate}
 \item $\Scal$ is a smooth projective rational surface defined over $\CX$ with a surjective morphism $\pi:\Scal \rightarrow \PX^1(\CX)$ ;
 \item {Almost all fibers are isomorphic to $\Etproj$}, that is, they  are nonsingular elliptic curves.
 
 \item The remaining  finite number of fibers are called singular fibers and  are singular (reduced) curves. The fiber over $0$ is singular. 
 \item {There exists a section $\sigma_0: \PX^1(\CX) \rightarrow \Scal \, (\pi\circ \sigma_0= {\rm id}_\Scal)$ } and
there is a bijection between $\CX(t)$-points {$P$} of $\Etproj$ and sections { $\sigma_P:\PX^1 \rightarrow \Scal \ (\pi\circ \sigma_P= {\rm id}_\Scal)$ so that $\sigma_0$ corresponds to the origin of the elliptic curve $\Etproj$.}
\end{enumerate}

Let us denote by {$\Pcal$ the image in $\Scal$,  of the section $\sigma_P$ corresponding to a  $\CX(t)$-point $P$ of $\Etproj$.  $\Pcal$ is a curve in the surface $\Scal$. Abusing terminology, we shall call $\mathcal{P}$ the section associated to $P$}. The N\'eron-Tate height of a point $P$ is defined in terms of a numerical invariant of $\Scal$, how the section $\Pcal$ intersects $\Ocal$, the section corresponding to the origin $O$ of $\Etproj$, and how $\Pcal$ intersects some of the singular fibers. The (at first intimidating) formula defining the N\'eron-Tate height is
\begin{align}\label{NT}
\hhat(P) = 2\chi(\Scal) + 2(\Pcal.\Ocal) - \sum_{v\in R}{\rm contr}_v(P)
\end{align}
The term $\chi(\Scal)$ is the arithmetic genus of $\Scal$.. { By \ref{lem:rationalellipticsurface}, the surface $\Scal$ is rational  so that its arithmetic genus  is $1$ (\cite[Proposition 7.1]{SchuttShiodaBook})}. The term $(\Pcal.\Ocal)$ is the intersection number of  $\Pcal$ and  $\Ocal$, where $\Ocal$ is the section corresponding to the origin of $\Etproj$.  In our applications, these sections are disjoint so, for us, $(\Pcal.\Ocal)=0$. For the remaining sum, $R$ is the finite set of  {singular} fibers $v$ and ${\rm contr}_v(P)$ is a rational number determined by how $\Pcal$ intersects the components of $v$. Much is known about $R$ and the numbers ${\rm contr}_v(P)$.

Kodaira~\cite{Kodaira1,Kodaira2} and N\'eron~\cite{Neron} classified the types of fibers which can occur { in such a fibration} (see also\cite[Ch.IV,\S9, Table 4.1]{Silverman94}). Based on the configuration of the intersections of the components of such a fiber $v$, one associates a root lattice $T_v$ of type $A,D,$ or $E$. Up to a finite number of possibilities, ${\rm contr}_v(\Pcal)$ is determined by the root lattice of the fiber $T_v$. This information is summarized in Table~\ref{table:contr} (see~{\cite[Table 6.1]{SchuttShiodaBook}}, \cite[(8.16)]{Shioda90},\cite[Lemma 7.5.3]{DuistQRT}).

\begin{table}[h!]
\centering
\begin{tabular} { c c c c c c c}
\hline\\
Kodaira Fiber Type &$III$ & $III^*$ &$IV$ &$IV^*$ &  $I_n(n>1) $ & $I_n^*$\\[0.1in]
\hline\\
Root Lattice $T_v$ of Fiber & $A_1$& $E_7$ & $A_2$& $E_6$ & $A_{n-1}$ & $D_{n+4}$ \\[0.1in]
\hline\\
Possible ${\rm contr}_v(P)$& $1/2$& $3/2$ & $2/3$& $ 4/3$  & \begin{minipage}{1in} $i(n-i)/n$ \\ $0\leq i \leq n-1$\end{minipage} & 
 $\begin{cases} 1, & i=1\\ 1+n/4 ,& {i>1} \end{cases}$\\[0.2in]
 \hline\\
\end{tabular}

\caption[Caption for LOF]{This table gives the range of possibilities for ${\rm contr}_v(P)$. In Section~\ref{sec:refine} we show how $i$ can be determined exactly based on the explicit construction of $\Scal$ and the specific $\Pcal$ but for now we are only concerned with knowing the finite set of possibilities\footnotemark. }\label{table:contr}
\end{table}

\footnotetext{Kodaira's classification of fiber types included an additional fiber referred to as type $II^*$. It is not included in this table since in this situation any point $P \in \Etproj(\C(t))$ has finite order and the group of the walk is finite (see \cite[Table 8.2 ]{SchuttShiodaBook}).}

 The direct sum $T = \oplus_{v\in R}T_v$ is defined to be the {\it root lattice} associated with the the {singular} fibers. In \cite{OgSh91}, Oguiso and Shiota give a finite list of the possible root lattices which can occur (there are 74).  This implies that if one can determine $T_v$ for at least one fiber, then seeing which root lattices contain $T_v$ allows one to determine the term $\sum_{v\in R}{\rm contr}_v(P)$ in \eqref{NT} up to a finite set of possibilities. 
 
 {\begin{rmk}\label{rmk:finiteheight}
 By \cite[Theorem 6.20]{SchuttShiodaBook}, a point $P \in \Etproj(\C(t))$ has height zero if and only $P$ is a torsion point. 
 Choosing some point $O$ to be  the origin of $\Etproj$, one remarks that $\tau^n(O)=O$ if and only if $n \tau(O)=O$ if and only if $\tau(O)$ is a torsion point. Therefore the group of the walk is finite if and only   $\hhat(\tau(O))=0$. In that situation, the order of the group    is $2n$ where $n$ is the order of torsion of $\tau(O)$.  If one knows the root lattice of the singular fibers, \cite[Table 8.2]{SchuttShiodaBook} gives the torsion subgroup and thereby an upper-bound for  the order of the group of the walk.  By \cite[Cor. 8.21]{SchuttShiodaBook}, the order of the torsion is bounded by $6$ and therefore the order of the group is bounded by $12$. Note that we are considering  the group of the walk acting on a generic fiber. If one considers its action on an arbitrary fiber, its order might be bigger than $12$ but less than $24$ by Mazur Theorem ({\mfs assuming that the fiber is defined over $\QX$}; see \cite{JiangTavakoliZhao}). 
\end{rmk}}

  An algorithm due to Tate~\cite{Tate}  allows us to determine {the type of any fiber.  We shall use it only to determine the type of the fiber  above $0$. The Tate algorithm relies on the Weierstrass model of $\Etproj$ (see \cite[Sections 5.7 and 5.8]{SchuttShiodaBook} and  also \cite[Ch. IV, \S9]{Silverman94}, \cite[Lemma 6.3.1]{DuistQRT}).} This leads to the following algorithm. \\[0.1in]
 
  \noindent\underline{\it Algorithm.} {\mfs As noted in Remark~\ref{rmk:finiteheight}, if the $\tau$ has finite order, then its order is bounded by $6$. Calculating $\tau^n, 1\leq n \leq 6$ will give polynomial conditions on the $d_{i,j}$ equivalent to $\tau$ being of finite order, c.f.~\cite{KauersYatchak} (in  Section~\ref{sec:refine} we will see that a more careful examination of $\mathcal{S}$ and its Mordell-Weil Lattice  will yield such equations directly).  We can therefore assume that $\tau$ is of infinite order and that we are given a kernel $K$ whose associated curve satisfies the conditions of Proposition~\ref{prop:sameorbitQiequalsPj} or Proposition~\ref{prop:sameorbitQjneqP_i}. } These propositions say that $b$ has a certificate if and only if two distinct $\CX$-points (which we will denote by $N$ and $M$) of the curve are in the same $\tau$-orbit. {By Lemma \ref{lem:intersectionsection}, the curves 
  $\mathcal M$ and $\tau(\mathcal{N})$ do not intersect $\mathcal N$ in $\Scal$}.  
We will show how to decide if $\tau^n(N) = M$ for some $n \in \ZX$.  We have freedom to select the point of $\Etproj$ that will be the origin $O$ of the associated group and so we will let $O = N$. Recall that { $\tau(P) = P\oplus \tau(N)$ for any point $P$}, so we have that if $\tau^n(N) = M$, then   $M = n\tau(N)$. In particular, $\hhat(M) = n^2 \hhat(N)$. We will first find a finite set $H$ of rational numbers, depending on $K$, such that if $Q$ is any point of $\Etproj$ such that the corresponding {curve} $\Qcal$ does not intersect $\Ocal$, then $\hh(Q) \in  H$.   Since this hypothesis holds for $M$ and {$\tau(N)$}, we can compare all pairs of values $r_1,r_2$ in $H$ and determine all integers $n$ such that $n^2 = r_1/r_2$. For these integers,  a computation will check if $\tau^n(N) = M$. \\

\noindent \ul{\it Step 1: Find the Kodaira type of the fiber above  $0$ and its associated root lattice $T_0$.} This can be done using the algorithm of Tate mentioned above.  Tate's algorithm determines, in all characteristics,  the Kodaira type of a singular fiber (assumed to be above $0$) of {an elliptic surface whose generic fiber is given by a Weierstrass equation $y^2=4x^3-g_2x-g_3$ with $g_2,g_3 \in \C(t)$}. In characteristic $0$,  the algorithm shows that the  type is determined  by the order of vanishing of the discriminant $\Delta$ and  the invariants $g_2$ and $g_3$ at $0$. Formulas to express $\Delta,g_2,g_3$  in terms of the coefficients of $K$ are given in~\cite[Section 2.3.5, Proposition 2.4.3, Corollary 2.5.10]{DuistQRT}. Restricting to the fiber types in Table~\ref{table:contr},   Table \ref{table:Tate}  gives the type of the fiber in terms of the order of vanishing of $\Delta, g_2,$ and $g_3$ (See also \cite[Table 1]{SchuttShioda}, \cite[Lemma 6.3.1]{DuistQRT}). Note that the table does not deal with the cases when the order of $g_2 \geq 4$ and the order of $g_2\geq 6$.  When this is the case a successive  changes of variables of the for $x\mapsto t^2x, y\mapsto t^3 y$ will ensure that this condition is met since with this transformation, the order of $\Delta$ drops by $12$ and this can happen only a finite number of times. 
\begin{table}[h!]\
\centering
\begin{tabular} { c | c c c}
Type & $g_2$ & $g_3$ & $\Delta$\\
\hline
$I_n, n\geq 1$ & $0$ & $0$ & $n$\\
$I_0^*$ & $\geq 2$ & $ \geq 3$& $6$ \\
$I_n^*, n\geq 1$& $2$& $3$ & $n+6$\\
$III$ & $1$ & $\geq 2$ & $3$\\
$III^*$ & $3$ & $\geq 5$& $9$\\
$IV$ & $\geq 2$ & $2$ & $4$\\
$IV^*$ & $\geq 3$ & $4$ & $8$\\
\end{tabular}
\caption{Local contributions of the singular fibers }\label{table:Tate}
\end{table}
One now uses Table~\ref{table:contr}  to find the associated root lattice~$T_0$.
 \begin{exa}\label{exa:algo} Consider the weighted model:
 $$
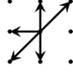
{\begin{tikzpicture}[scale=.4, baseline=(current bounding box.center)]
\foreach \x in {-1,0,1} \foreach \y in {-1,0,1} \fill(\x,\y) circle[radius=2pt];
\draw[thick,->](0,0)--(0,1);
\draw[thick,->](0,0)--(1,1);
\draw[thick,->](0,0)--(-1,0);
\draw[thick,->](0,0)--(-1,-1);
\draw[thick,->](0,0)--(0,-1);
\end{tikzpicture}}
$$
with nonzero weights $d_{1,1}, d_{0.-1}, d_{-1,-1}, d_{-1,0}, d_{0,1},d_{0,0}$. {When unweighted, this model was called  $w_{IIC.2}$ and we shall keep this notation for the weighted model.} The associated kernel is

\begin{align*}&K(x_0,x_1,y_0,y_1,t_0,t_1) = x_{{0}}x_{{1}}y_{{0}}y_{{1}}- \\
&t \left( d_{{-1,-1}}{x_{{1}}}^{2}{y_{{1}}
}^{2}+d_{{-1,0}}{x_{{1}}}^{2}y_{{0}}y_{{1}}+d_{{0,-1}}x_{{0}}x_{{1}}{y
_{{1}}}^{2}+d_{{0,0}}x_{{0}}x_{{1}}y_{{0}}y_{{1}}+d_{{0,1}}x_{{0}}x_{{
1}}{y_{{0}}}^{2}+d_{{1,1}}{x_{{0}}}^{2}{y_{{0}}}^{2} \right).\\
\end{align*}

The polar divisor of $b$  is $(b)_\infty = P_1 + Q_0 + \ita(Q_0)$, where
\begin{itemize}
\item $P_1 = P_0  = ([1:0],[0:1])$
\item $Q_0 = ([-d_{0,1}:d_{1,1}], [1:0])$
\item $\ita(Q_0) = ([-d_{0,1}:d_{1,1}], [t(d_{-1,-1}d_{1,1}- d_{0,-1} d_{0,1}) : -(d_{0,1} +t(d_{-1.0}d_{1,1} - d_{0,0}d_{0,1})]$
\end{itemize}
Furthermore, $Q_1 = ([0:1]:[1:0])$.  This means we are in Case d.4 of Proposition~\ref{prop:sameorbitQjneqP_i} and we must decide if $P_1$ and $Q_0$ are in the same $\tau$-orbit.\\[0.1in]
A {\sc Maple} calculation   shows that the orders of $g_2$ and $g_3$ are $0$ and the order of $\Delta$ is $7$ (see \cite{HSurl}).// 

Therefore Table 5.2 impies that the associated fiber is $I_7$ and Table 5.1 implies that the root lattice $T_0$ is $A_6$.
\end{exa}

\noindent\underline{\it Step 2: Determine $T$.} Once one has found the reducible fiber $v$ above $0$, use Table~\ref{table:contr}   to determine its associated root lattice $T_0$. Consult the table of all possible root lattices in \cite[Table 8.2]{SchuttShiodaBook} {or the table in \cite{OgSh91}} to find all possible $T$ of which $T_0$ is a summand.\\[0.1in]
\noindent {\it Example~\ref{exa:algo}(bis):}  Since  $T_0 = A_6$,  the possibilities for $T$ {listed in these tables} are $A_6$ and $A_6\oplus A_1$.  This implies that there are one or two singular fibers.\\[0.1in]
\noindent \underline{\it Step 3: Determine possible ${\rm contr}_v(Q)$ and possible $\hhat(Q)$.}   For each of the possible $T$ found in Step 2 and each of the summands $T_v$, determine the set of possible {values of ${\rm contr}_v(Q)$} from Table~\ref{table:contr} and then determine the possible values of $\hhat(Q)$. Our assumption on $P$ and $\Scal$ imply that \eqref{NT} simplifies to
\begin{align}\label{NTbis}
\hhat(Q) = 2 - \sum_{v\in R}{\rm contr}_v(Q)
\end{align}

\noindent {\it Example~\ref{exa:algo}(bis):}   If $T = A_6$, then there is only one reducible fiber  $v_0$ and Table~\ref{table:contr} implies that ${\rm contr}_{v_0}(Q) \in \{ 0, 6/7, 10/7, 12/7\}$ and $\hhat(Q) = 2 - {\rm contr}_{v_0}(Q) \in \{2,8/7,4/7,2/7\}$. If $T = A_6 \oplus A_1$ then there are two fibers: $v_0$ as before and $v_1$.  We have ${\rm contr}_{v_1}(Q) \in \{0,1/2\}$. Therefore $\hhat(Q) = 2 - {\rm contr}_{v_0}(Q)  -{\rm contr}_{v_1}(Q) \in \{2,8/7,4/7,2/7,3/2, 9/14/,1/14\} = H$.\\[0.1in]
\noindent \ul{\it Step 4: Determine possible values of $n$ such that  $\hhat(M) = n^2\hhat(N)$ and test if $M = \tau^n(N)$ for these values.} This involves determining if $r_1/r_2$ is a square for $r_1, r_2 \in H$ and then using the definitions of $\ita$ and $\itb$ to calculate $\tau^n(N)$ and compare this with $M$.  For weighted models this will yield polynomial conditions that are necessary and sufficient for $\tau^n(N) = M$.\\[0.1in]
\noindent {\it Example~\ref{exa:algo}(bis):} One finds that the possible values of $n$ are $-4,-3,-2,-1,0,1,2,3,4$.  The entries in the coordinates of $\tau^n(P_1)$ and $Q_0$ are polynomials in $t$ and the weights.    In all cases, except $n = -1,$ we show  via a {\sc Maple} calculation (see \cite{HSurl}) that $Q_0 \neq \tau^{n}(P_1)$  For $n = -1$, we have 
\begin{align*}
\tau^{-1}(P_1) &= \ita(\itb(([1:0],[[0:1]) ) = \ita(([-d_{-1,-1}:d_{0,-1}],[0:1]) )\\
  &= ([-d_{-1,-1}:d_{0,-1}],[y_0:y_1])\\
 &{\rm where}\\
 y_0 &= d_{{0,-1}} \left( td_{{-1,-1}}d_{{0,0}}-td_{{-1,0}}d_{{0,-1}}-d_{{-1,-1}} \right) {\rm and}\\
 y_1 &= td_{{-1,-1}} \left( d_{{-1,-1}}d_{{1,1}}-d_{{0,-1}}d_{{0,1}} \right).
\end{align*}
Since $Q_0  = ([-d_{0,1}:d_{1,1}],[1:0])$ we have that  $Q_0= \tau^{-1} (P_1)$ if and only if \[d_{{-1,-1}}d_{{1,1}}-d_{{0,-1}} d_{0,1}= 0.\]
This implies that this weighted model has an $x$- and $y$-$D$-algebraic generating series if and only if this latter condition holds.
{Note that this condition is automatically satisfied if the model is unweighted so that the unweighted model $w_{IIC.2}$ was
$x$-and $y$-$D$-algebraic.}

\subsection{Refinements}\label{sec:refine} In this section we shall give a more detailed description of the {\it Kodaira-N\'eron model} associated to $\Etproj$  and the computation of the numbers ${\rm contr}_v(P)$.  This will allow us to refine the algorithm described in the previous section.  We will assume a familiarity with several concepts from the algebraic geometry of surfaces {with a particular emphasis on  intersection theory   and  resolution of singularities via blowups (see for instance \cite[Chap 4]{Shafa1}).}

\subsubsection{The geometric objects}\label{subsec:geometricobjects}
One attaches to the kernel polynomial some geometric objects. We denote by $S([x_0: x_1],[ y_0:y_1])$ the homogeneous biquadratic polynomial defined by $x_1^2y_1^2S(\frac{x_0}{x_1},\frac{y_0}{y_1})$ in the notation of Section\ref{sec:ker}.
First, one can consider the pencil $\mathfrak C$  of biquadratic  curves $C_{[\lambda:\mu]}$ in $\P1(\C) \times \P1(\C)$ defined by 
$C_{[\lambda: \mu ]} =\{ ([x_0:x_1],[y_0:y_1]) \in \P1(\C) \times \P1(\C) |  \mu x_0x_1y_0y_1 -\lambda S([x_0:x_1],[y_0:y_1]) =0 \}$ whose base points, that are the common zeros of $ x_0x_1y_0y_1$ and 
$S([x_0: x_1],[ y_0:y_1])=0$ are represented in the  figure \ref{fig:positionbasepoints}.

\begin{figure}[h!]

  \begin{tikzpicture}[scale=1.5]
\coordinate (O) at (0:0);
\foreach \i in {1,...,4}{
\coordinate (A\i) at (\i * 90 -4 -45:2);
\coordinate (B\i) at (\i * 90 +94 -45 :2);
 \tkzDefBarycentricPoint(A\i=1,B\i=2)
  \tkzGetPoint{I\i};
   \tkzDefBarycentricPoint(A\i=2,B\i=1)
  \tkzGetPoint{J\i};
}
\coordinate (C1) at (0, 2);
\coordinate (C2) at (-2, 0);
\coordinate (C3) at (0, -2);
\coordinate (C4) at (2, 0);
 
\node[label=above: $Q_0$]  at (I1) {$\circ$};
   \node[label=above: $R_0$]  at (I2) {$\circ$};
    \node[label=above: $S_0$]  at (I3) {$\circ$};
     \node[label=above: $P_0$]  at (I4) {$\circ$};
     
      \node[label=above: $Q_1$]  at (J1) {$\circ$};
   \node[label=above: $R_1$]  at (J2) {$\circ$};
    \node[label=above: $S_1$]  at (J3) {$\circ$};
     \node[label=above: $P_1$]  at (J4) {$\circ$};
     \draw[red] plot[smooth cycle] coordinates {(C1) (I1) (J2) (C2) (I2) (J3) (C3) (I3) (J4) (C4) (I4) (J1)}node[pos=0.5,  left ]{\small{$S=0$}}  ;
     
\draw[blue] (A1)  -- (B1)  node[pos=0.5, below  ] {\textcolor{blue}{$y_1=0$}};
\draw[violet ](A2) --  (B2) node[pos=0.5, left ] {\textcolor{violet}{$x_0=0$}};
\draw[blue] (A3) --(B3) node[pos=0.5, below ] {\textcolor{blue}{$y_0=0$}} ;
\draw[violet] (A4) --(B4)  node[pos=0.5,  right ] {\textcolor{violet}{$x_1=0$}};
\end{tikzpicture}

\caption{Position of the base points } \label{fig:positionbasepoints}
\end{figure}
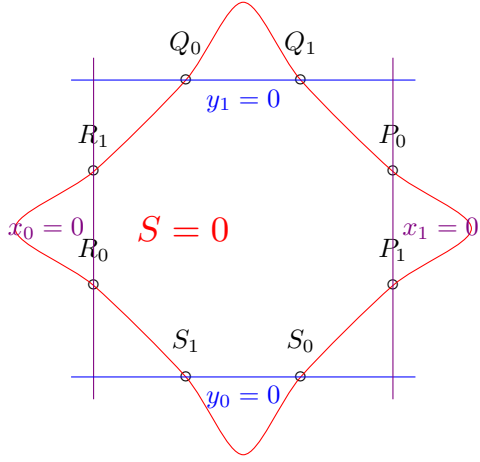 

  Any member of the pencil $\mathfrak{C}$  passes through  $\{P_0,P_1,Q_0,Q_1,R_0,R_1,S_0,S_1\}$.  There are $8$ of these base points counted with multiplicities. 

For any  pair of elements  $t_1,t_2 \in \C$, each transcendental over $\Q$, the curves $C_{[t_1:1]}$ and  $C_{[t_2:1]}$ are isomorphic over $\Q$. These curves are general members of the pencil. They are isomorphic to $\Etproj$ over $\C$.   The following Lemma shows how 
to construct a Kodaira-N\'{e}ron model  $\mathcal S$ attached to $\Etproj$, that is a relatively minimal fibration over $\P1(\C)$ with a rational  section and  whose general fiber is $\Etproj$. 

\begin{prop}[Cor. 3.3.10 and \S3.3.5 in \cite{DuistQRT}]\label{prop:ellipticsurfaceobtainedbybowingup}
Let $\mathcal{S}$ be the surface obtained by successively blowing up $\P1 \times \P1$ at the eight base points of the pencil $\mathfrak C$ counted with multiplicities.  Write $\pi=\pi_1\circ \hdots \circ \pi_8: \mathcal{S} \mapsto \P1\times \P1$. Then the space $W$ of holomorphic
$2$-vector fields  on $\Scal$ is two dimensional and $\Scal$ together with the  mapping $\kappa :\mathcal{S} \mapsto \P1(W), s \mapsto  \{w \in W |w(s)=0\}$ is a Kodaira-N\'{e}ron model for $\Etproj$. 
Moreover, the following holds
\begin{itemize}
\item  a member $C$ of the pencil $\mathfrak C$ is smooth if and only if its strict transform $\pi'(C)$ is a smooth fiber of $\kappa$ when $\pi|_{\pi'(C)} $ is an isomorphism from $\pi'(C)$ to $C$; In particular the general fiber of $\mathcal{S}$ is $\Etproj$. 
\item  $\kappa$ coincides with $\phi \circ \pi$ with $\phi: \P1 \times \P1 \dashrightarrow \P1, ([x_0:x_1],[y_0:y_1]) \mapsto (x_0x_1y_0y_1, S([x_0: x_1],[ y_0:y_1])$
on the open dense subset of $\Scal$ where $\phi  \circ \pi$ is defined.
\end{itemize}
\end{prop}
Note that the indeterminacy locus of the  rational map $\phi$ is precisely the set of base points.
A straightforward  corollary of Proposition \ref{prop:ellipticsurfaceobtainedbybowingup} is the following.
\begin{lem}\label{lem:rationalellipticsurface}
The Kodaira-N\'{e}ron model  $\mathcal S$ of 
$\Etproj$ is rational elliptic surface. 
\end{lem}
\begin{proof}
Indeed, it is birational to $\P1 \times \P1$ via $\pi$ and $\P1 \times \P1$ is birational to $\bold{P}^2$.
\end{proof}

 Proposition 5.4 of \cite{SchuttShiodaBook} describes the correspondence between  $\C(t)$-points of $E$ and rational sections of $\kappa:\mathcal{S} \rightarrow \P1$.  The following Lemma shows how  one can make explicit this dictionary   in the special cases of   base points.

 \begin{lemma}\label{lem:basepointssections}
Let $P=(a,b) \in \P1 \times \P1$ be a base point. Then the multiplicity $m$ of  $P$ as base point is less than or equal to $3$.   Moreover,  the last exceptional divisor $\mathcal{E}_{(a,b)}$
obtained by blowing up $m$ times at $P$ is  the image of the  section of $\mathcal{S}$ that corresponds to  the point $(a,b)$ in  $E$. 
\end{lemma} 
\begin{proof}
The multiplicity is  less than or equal to $3$ because otherwise the base point would be singular for any member of the pencil contradicting the fact that $\Etproj$ is a genus one curve. {The second assertion is \cite[Cor.3.3.9]{DuistQRT}}. \end{proof}

In Section\ref{sec:algo}, we use Propositions \ref{prop:sameorbitQiequalsPj} and \ref{prop:sameorbitQjneqP_i} to implement an algorithm, which allows us to decide if the weighted model was decoupled or not.  We use the formula (\ref{NT}) defining the N\'{e}ron-Tate height and claimed that when we apply this formula in our situation, the term representing the  intersection multiplicity $(\Pcal.\Ocal)$ is zero.
 
The main purpose of the following lemma is to verify this claim.

\begin{lem}\label{lem:intersectionsection}
Assume that $\Etproj$ is a genus $1$ curve and that there is no $P_j$'s and $Q_k$'s that are simultaneously fixed by an involution. The following holds:
\begin{itemize}
\item \emph{ Case $1$: $P_j=Q_k$ for some $j$ and $k$} Then, the section $\mathcal{P}_{j+1}$ has empty intersection with $\mathcal{P}_j$ and $\mathcal{Q}_{k+1}$, which is the section corresponding to $\tau(P_{j+1})=Q_{k+1}$.
\item \emph{ Case $2$: $P_j \neq Q_k$ for any $j$ and $k$} Then, the section $\tau^{-1}(\mathcal{Q}_k)$ corresponding to the point $\tau^{-1}(Q_k)$ does not intersect the sections $\mathcal{Q}_k$ and $\mathcal{P}_j$. 
\end{itemize}
\end{lem}
\begin{proof}
In the first case, we have $P_0 \neq P_1$ and $Q_0 \neq Q_1$ by assumption. For simplicity, let us assume that $P_0=Q_0$. By Lemma \ref{lem:basepointssections}, the section $\mathcal{P}_1$ (resp.~$\mathcal{Q}_1$, $\mathcal{P}_0$) is the last  exceptional divisor obtained by blowing up at $P_1$ (resp.~$Q_1$,$P_0$). Then, $\mathcal{P}_1 \subset \pi^{-1}(P_1)$, $\mathcal{P}_0 \subset \pi^{-1}(P_0)$ and $\mathcal{Q}_1 \subset \pi^{-1}(Q_1)$. Since $P_1 \neq Q_1$ and $P_1 \neq P_0$, we conclude that $\mathcal{P}_1$ has empty intersection with $\mathcal{Q}_1$
and $\mathcal{P}_0$.

In the second case, let $\alpha \in \C$ such that $Q_{k+1}=(\alpha, \infty)$. Then,$\tau^{-1}(Q_k)$ is the point $(\alpha , [-t(\sum d_{i,-1}\alpha^{i+1}): \alpha-t(\sum d_{i,0}\alpha^{i+1})])$. Let us now consider the curve $C$ in $\P1 \times \P1$ defined by 
$C=\{ (\alpha , [-t_0(\sum d_{i,-1}\alpha^{i+1}): t_1\alpha-t_0(\sum d_{i,0}\alpha^{i+1})]\mbox { with } [t_0:t_1] \in \P1 \}$.  The strict transform of $C$ by $\pi$ corresponds to the section $\tau^{-1}(\mathcal{Q}_k)$. Then, it is easily seen that $\mathcal{P}_j$ does not intersect $\tau^{-1}(\mathcal{Q}_k)$ because $P_j$ does not belong to $C$. If $Q_k \neq Q_{k+1}$ then $Q_k$ does not belong to $C$ so that $\mathcal{Q}_k$ does not intersect $\tau^{-1}(\mathcal{Q}_k)$. If $Q_k=Q_{k+1}$ then the multiplicity of $Q_k$ is  $2$ if $\alpha \neq 0$ and $3$ if $\alpha = 0$.  Since the curve is non singular, the point $Q_{k+1}$ is not fixed by $\ita$. Thus we need to blow up at least two times at $Q_k$. At the  first blowup $\pi_1$ at $Q_k$, the strict transforms of the curve $y_1=0$ and $S([x_0:x_1],[y_0:y_1])=0)$ still intersect the  exceptional divisor at the same point $Q_k^{(1)}$ because they have the same tangent at $Q_k$. The second blowup will be performed at the $Q_k^{(1)}$. Since the curve $C$ has not the same tangent than $y_1=0$ at $Q_k$, it intersects the exceptional divisor at some point $P \neq Q_k^{(1)}$. Then, one can reason as above to conclude that the sections $\mathcal{Q}_k$ and $\tau^{-1}(\mathcal{Q}_k)$ do not intersect because the first one is contracted on $Q_k^{(1)}$ by $\pi_2 \circ \dots \circ \pi_8$ whereas the second is sent on a curve that does not pass through $Q_k^{(1)}$.
\end{proof}

\begin{rmk}\label{rmk:symmetryintersection}
Using some symmetry arguments as in Lemma \ref{lemma:symmetryxandy}, one can easily deduce from  Lemma \ref{lem:intersectionsection} that
\begin{itemize}
\item \emph{ Case $1$: $P_j=Q_k$ for some $j$ and $k$} Then, the section $\mathcal{Q}_{k+1}$ has empty intersection with $\mathcal{P}_j$ and $\mathcal{Q}_{k}$, which is the section corresponding to $\tau(P_{j+1})=Q_{k+1}$.
\item \emph{ Case $2$: $P_j \neq Q_k$ for any $j$ and $k$} Then, the section $\tau(\mathcal{P}_j)$  does not intersect the sections $\mathcal{Q}_k$ and $\mathcal{P}_j$. 
\end{itemize}

\end{rmk}

\subsubsection{The fiber above zero}\label{subsubsec:fiberzero}
The construction of $\pi$ {aims at separating} the members of the pencil $\mathfrak C$ so that they define an elliptic fibration. In order to understand the type of the fiber $F_0$ above zero of $\Scal$, one has to understand how the curve $C_{[0:1]}:=\{([x_0:x_1],[y_0:y_1]) \in \P1(\C) \times \P1(\C) | x_0x_1y_0y_1 =0 \}$ behaves after each blowup.

{In Example~\ref{exa:algo} of Section~\ref{sec:algo}, computing the Weierstrass form and applying the table related to the Tate algorithm, allows us to conclude that the Kodaira type of the fiber $F_0$ above $0$ is $I_n$ with $n=7$.  This is an instance of the following result which we prove in this section. In Section~\ref{sec:exs}, we show in two examples that by calculating $F_0$ one can furthermore determine the contribution ${\rm contr}_0(P)$ in a more exact manner, sharpening the computation described in  Section~\ref{sec:algo}.}   \\

 \begin{lem}
The type of $F_0$ is $I_{n}$ where the 
number $n$ of components of $F_0$ varies between $4$ and $9$ depending on the multiplicity and the position of the base points. \end{lem} 

Then, according to \cite[Table 8.2]{SchuttShiodaBook}, there are precisely 
\begin{itemize}
\item one possible root lattice when $n=9$, 
\item two possible root lattices when $n=8$,
\item two possible root lattices when $n=7$, 
\item five  possible root lattices when $n=6$,
\item seven  possible root lattices when $n=5$,
\item $19$ possible root lattices when $n=4$.
\end{itemize}
All together, there are at worst  $28$ distinct root lattices, which can be associated to   $\Scal$. 
Thus, the number of possibilities for the local contributions of the singular fibers is quite low once one has determined the local contribution of the fiber above $0$.  In the rest of this section, we show how to determine the number of components $n$ of the fiber $F_0$  with respect to the multiplicity of the base points and their relative positions. {Knowing the relative position of these components allows us to decrease the number of cases considered in the algorithm.}\\

\paragraph{\emph{A.~No multiple base points}}

Then the multiplicity of $C_{[0:1]}$ at each base point is $1$. The \emph{strict transform } of $C_{[0:1]}$ is the fiber above $0$. It is  a cycle of $n=4$ projective lines. The sections corresponding to the base points are exactly the $8$ exceptional divisors  and their intersection with $F_0$  is similar to Figure \ref{fig:positionbasepoints}. \\


\paragraph{\emph{B.~Multiple base points}}

In this paragraph, we show how the multiplicity of a base point contributes to the number of components of $F_0$. There are three cases. \\

\subparagraph{\emph{B.1~Two base points in a corner}}

Assume that for instance $Q_0=R_0$ and $Q_0 \notin \{R_1,Q_1\}$. We  perform a first blowup at $Q_0=R_0$ and we choose the affine chart of  $\bold{A}^2 \subset \P1 \times \P1$ given by $x_1=1, y_0=0$. The coordinate of this chart are $x:=x_0$ and $y:=y_1$. By assumption, 
$S(x,y)=d_{-1,0}y +d_{0,1}x + R(x,y)$ with $R(x,y)$ with no monomials of degree less than or equal to $1$ and $d_{-1,0}d_{0,1}\neq 0$. 
In this chart, the blowup of $Q_0$ consists in considering the map $\pi_1: X \rightarrow \P1  (x,y)\times[u:v] \mapsto  (x,y)$ where $X= \{(x,y)\times [u:v] |ux=vy \} \subset \bold{A}^2 \times \P1$. In the chart $u=1$, the exceptional divisor $\mathcal{E}_1$ is given by $y=0$.  The total transform of a member $C_{[\lambda:\mu]}$ is given by the zero set of
\begin{align*}
\mu xy -\lambda S(x,y) =&\mu vy^2 - \lambda (d_{-1,0}y +d_{0,1}vy + R(vy,y))\\
=&y\left(\mu vy -\lambda(d_{-1,0} +d_{0,1}v + R'(v,y)) \right),
\end{align*}
where $R'(v,y)= R(vy,y)/y$. Thus,  the strict transform of a general member of the pencil is given by 
$\mu vy -\lambda(d_{-1,0} +d_{0,1}v + R'(v,y))=0$.  This defines a new pencil $\mathfrak D$. The member of $\mathfrak{D}$ over zero corresponds to $vy=0$ and is therefore equal to the union  of the  proper transform of $C_{[0:1]}$ and of the first exceptional divisor $\mathcal{E}_1$. Moreover,   one can easily see that all members of  $\mathfrak{D}$ intersect $\mathcal{E}_1$ at the point $Q_0^{(1)}$ with coordinates $v=-\frac{d_{0,1}}{d_{-1,0}}, y=0$. A second blowup at this point yields a separation of the members of the pencil and resolves the singularity of the rational map  $\phi$ defined in Proposition \ref{prop:ellipticsurfaceobtainedbybowingup}  at $Q_0$. One concludes that each time this case happens one has to add a new component at the proper transform of $C_{[0:1]}$. The last exceptional divisor $\mathcal{E}_2$ corresponds to the section $\mathcal{Q}_0$.  It intersects $F_0$ at some point $Q_0^{(2)}$ of $\mathcal{E}_1$.

\begin{figure}[h!]

\centering
\subfloat[$Q_0=R_1$]{\label{fig:in}

\begin{tikzpicture}[scale=1.5]
\coordinate (O) at (0:0);
\foreach \i in {1,...,4}{
\coordinate (A\i) at (\i * 90 -4 -45:2);
\coordinate (B\i) at (\i * 90 +94 -45 :2);
 \tkzDefBarycentricPoint(A\i=1,B\i=2)
  \tkzGetPoint{I\i};
   \tkzDefBarycentricPoint(A\i=2,B\i=1)
  \tkzGetPoint{J\i};
}

   \node[label=above: $R_0$]  at (I2) {$\circ$};
    \node[label=above: $S_0$]  at (I3) {$\circ$};
     \node[label=above: $P_0$]  at (I4) {$\circ$};
     
      \node[label=above: $Q_1$]  at (J1) {$\circ$};
    \node[label=above: $S_1$]  at (J3) {$\circ$};
     \node[label=above: $P_1$]  at (J4) {$\circ$};
\draw [name path=A1--B1] (A1) to[ swap] (B1);
\draw [name path=A2--B2](A2) to[ swap] (B2);
\draw (A3) to[ swap] (B3);
\draw (A4) to[ swap] (B4);
\path [name intersections={of=A1--B1 and A2--B2,by=E}];
\node  [label=above left: $Q_0$] at (E) {$\circ$};
\end{tikzpicture}
}
\quad
\subfloat[After two Blowups]{\label{fig:in}
\begin{tikzpicture}[scale=1.5]
\coordinate (O) at (0:0);
\foreach \i in {1,...,4}{
\coordinate (A\i) at (\i * 90 -4 -45:2);
\coordinate (B\i) at (\i * 90 +94 -45 :2);
 \tkzDefBarycentricPoint(A\i=1,B\i=2)
  \tkzGetPoint{I\i};
   \tkzDefBarycentricPoint(A\i=2,B\i=1)
  \tkzGetPoint{J\i};
}

   \node[label=above: $R_0$]  at (I2) {$\circ$};
    \node[label=above: $S_0$]  at (I3) {$\circ$};
     \node[label=above: $P_0$]  at (I4) {$\circ$};
     
      \node[label=above: $Q_1$]  at (J1) {$\circ$};
    \node[label=above: $S_1$]  at (J3) {$\circ$};
    \node[label=above: $P_1$]  at (J4) {$\circ$};
\draw (A3) to[ swap] (B3);
\draw (A4) to[ swap] (B4);
\tkzDefBarycentricPoint(A2=3,B2=1)
  \tkzGetPoint{C2};
\tkzDefBarycentricPoint(A1=1,B1=3)
  \tkzGetPoint{C1};
  \tkzDefBarycentricPoint(A2=3,B2=2)
  \tkzGetPoint{C22};
\tkzDefBarycentricPoint(A1=2,B1=3)
  \tkzGetPoint{C11};
\tkzDefBarycentricPoint(C22=-3,C11=1)
  \tkzGetPoint{C3};  
\tkzDefBarycentricPoint(C11=-3,C22=1)
  \tkzGetPoint{C4};  
  
  \tkzDefBarycentricPoint(C3=1,C4=1)
  \tkzGetPoint{Q0};
  \node[label=above: $Q_0^{(2)}$]  at (Q0) {$\circ$};

\draw (A1) to [swap] (C1);
\draw (B2) to [swap] (C2);
\draw [blue, label=$E_1$] (C3) --  (C4) node[pos=0.5, below left ] {$\mathcal{E}_1$};
\end{tikzpicture}
 } \caption{The fiber above $0$ when $Q_0=R_1$ }\label{fig:twopointscorner}
\end{figure}
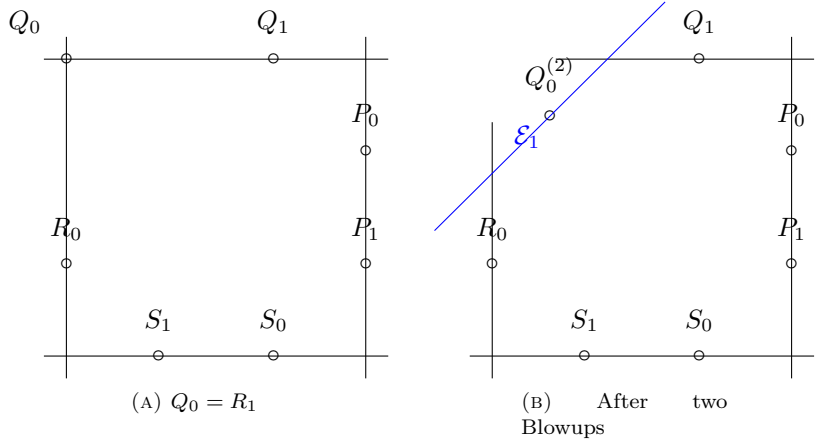

\vspace{1cm}
  
 \subparagraph{\emph{B.2~Two base points equal on a line}}
 Assume that for  instance $Q_0=Q_1=(a ,\infty)$ with $a \notin \{0, \infty\}$. We  perform a first blowup at $Q_0=R_0$ and we choose the affine chart of  $\bold{A}^2 \subset \P1 \times \P1$ given by $x_1=1, y_0=0$. The coordinate of this chart are $x:=x_0$ and $y:=y_1$.
 By assumption $S(x,y)=(x-a)^2  +A(x)y +B(x)y^2$. The  member $C_{[\lambda_a:\mu_a]}$  with $\mu_a a +\lambda_a A(a)=0$ of the pencil is singular at the point $(a,0)$.  The blowup at $Q_0$ is  the map $\pi_1: X \subset \bold{A}^2 \times \P1 (x,y)\times[u:v] \rightarrow  (x,y)$ where $X= \{(x,y)\times [u:v] |u(x-a)=vy \}$. In the chart $v=1$, the exceptional divisor $\mathcal{E}_1$ is given by $(x-a)=0$ and  a strict transform of a general member of the pencil $\mathfrak C$ is given $\mu ux- \lambda( (x-a) +A(x)u +B(x)u(x-a))$. This defines a new pencil $\mathfrak{D}$  whose member above zero is given by $ux=0$ that is by the proper transform of $C_{[0:1]}$. {\mfs All} members of the pencil $\mathfrak{D}$ meet on the point $Q_0^{(1)}$ given by  $u=0, x=a$ of the exceptional divisor $\mathcal{E}_1$. Thus one needs to blowup one more time at $Q_0^{(1)}$  to separate the members of the pencil $\mathfrak{D}$ and resolve the singularity of $\phi$ at $Q_0$. An easy computation shows that the exceptional divisor $\mathcal{E}_1$ is after the second blowup one of the components of  the fiber $F_{[\lambda_a:\mu_a]}$  with $\mu_a a +\lambda_a A(a)=0$. The last exceptional divisor $\mathcal{E}_2$ corresponds to the section $\mathcal{Q}_0$. It intersects $F_0$ at  some point $Q_0^{(2)}$ on the strict transform of $y_1=0$.

\begin{figure}[h!]
\centering
\subfloat[$Q_0=Q_1$]{\label{fig:in}
 \begin{tikzpicture}[scale=1.5]
\coordinate (O) at (0:0);
\foreach \i in {1,...,4}{
\coordinate (A\i) at (\i * 90 -4 -45:2);
\coordinate (B\i) at (\i * 90 +94 -45 :2);
 \tkzDefBarycentricPoint(A\i=1,B\i=2)
  \tkzGetPoint{I\i};
   \tkzDefBarycentricPoint(A\i=2,B\i=1)
  \tkzGetPoint{J\i};
}
\node[label=above: $Q_0$]  at (I1) {$\circ$};
   \node[label=above: $R_0$]  at (I2) {$\circ$};
    \node[label=above: $S_0$]  at (I3) {$\circ$};
     \node[label=above: $P_0$]  at (I4) {$\circ$};
     
   \node[label=above: $R_1$]  at (J2) {$\circ$};
    \node[label=above: $S_1$]  at (J3) {$\circ$};
     \node[label=above: $P_1$]  at (J4) {$\circ$};
\draw (A1) to[ swap] (B1);
\draw (A2) to[ swap] (B2);
\draw (A3) to[ swap] (B3);
\draw (A4) to[ swap] (B4);
\end{tikzpicture}}
\quad
\subfloat[After two Blowups]{\label{fig:in}
\begin{tikzpicture}[scale=1.5]
\coordinate (O) at (0:0);
\foreach \i in {1,...,4}{
\coordinate (A\i) at (\i * 90 -4 -45:2);
\coordinate (B\i) at (\i * 90 +94 -45 :2);
 \tkzDefBarycentricPoint(A\i=1,B\i=2)
  \tkzGetPoint{I\i};
   \tkzDefBarycentricPoint(A\i=2,B\i=1)
  \tkzGetPoint{J\i};
}
\node[label=above: $Q_0^{(2)}$]  at (I1) {$\circ$};
   \node[label=above: $R_0$]  at (I2) {$\circ$};
    \node[label=above: $S_0$]  at (I3) {$\circ$};
     \node[label=above: $P_0$]  at (I4) {$\circ$};
     
   \node[label=above: $R_1$]  at (J2) {$\circ$};
    \node[label=above: $S_1$]  at (J3) {$\circ$};
     \node[label=above: $P_1$]  at (J4) {$\circ$};
\draw (A1) to[ swap] (B1);
\draw (A2) to[ swap] (B2);
\draw (A3) to[ swap] (B3);
\draw (A4) to[ swap] (B4);
\end{tikzpicture}}

\caption{The fiber above zero when $Q_0=Q_1$ } \label{fig:q0equalQ1}
\end{figure}
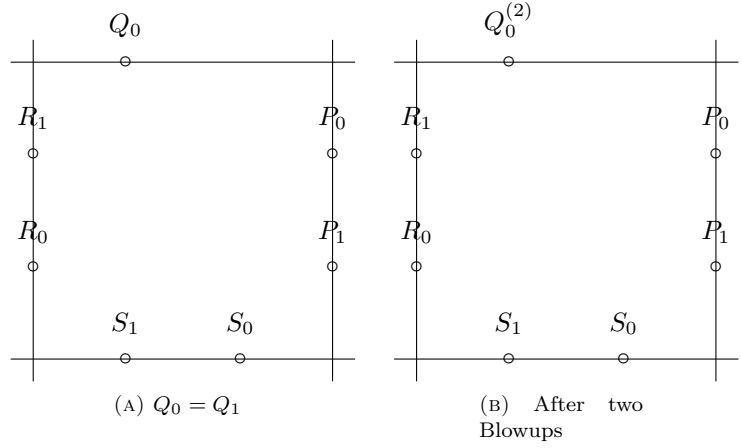

\vspace{1cm}
 \subparagraph{\emph{B.3~Three points  in a corner}}
 Assume that for instance $Q_0=R_0=Q_1$. In the coordinates  $x:=x_0$ and $y:=y_1$ of the affine chart  of  $\bold{A}^2 \subset \P1 \times \P1$ given by $x_1=1, y_0=0$, the polynomial $S(x,y)$ is of the form  $\alpha x^2 +A(x)y +B(x)y^2$ where $\alpha \neq 0$ because the general member of $\mathfrak{C}$ is non singular. In this chart, the blowup of $Q_0$ consists in considering the application $\pi_1: X \rightarrow \P1, (x,y)\times[u:v] \mapsto  (x,y)$ where $X= \{(x,y)\times [u:v] |ux=vy \}\subset \bold{A}^2 \times \P1$. In the chart $v=1$, the exceptional divisor $\mathcal{E}_1$ is given by $x=0$ and the strict transform of a general member of the pencil is given by 
 $$\mu ux +\lambda (\alpha x +A(x)u +B(x) u^2x^2).$$
 This allows to conclude that the member of $\mathfrak{D}$ above zero corresponds to $ux=0$ and is therefore the union of the  proper transform of $C_{[0:1]}$ and of the first exceptional divisor $\mathcal{E}_1$. Moreover all members of $\mathfrak{D}$ intersect at the point $Q_0^{(1)}$ given by $u=x=0$. Thus, one needs to perform a second blowup at the point $Q_0^{(1)}$. In the coordinates $u$ and $x$, this blowup is $\pi_2:  X \rightarrow \P1, (x,u)\times[c:d] \mapsto  (x,u)$ where $X= \{(x,u)\times [u:v] |uc=dx \} \subset \bold{A}^2 \times \P1$. In the chart $d=1$, the exceptional divisor $\mathcal{E}_2$ is given by $u=0$.  An easy computation shows that the total transform of a general member of $\mathfrak{D}$ is the zero set of 
 $$\mu cu+ \lambda(\alpha c +A(cu)+B(cu)uc).$$
 This defines a new pencil $\mathfrak{E}$ of curves.  The member above zero is given by $cu=0$ and is therefore the union of the proper transform of 
 $D_{[0:1]}$ and of the exceptional divisor $\mathcal{E}_2$. All the members of the pencil $\mathfrak{E}$ intersect on the point $Q_0^{(2)}$ given by $u=0, c=\frac{-A(0)}{\alpha}$. One needs to blowup once more at $Q_0^{(2)}$ to resolve the singularity of $\phi$ at $Q_0$. The fiber $F_0$ is thus the union of the strict transform of $C_{[0:1]}$, $\mathcal{E}_1$ and $\mathcal{E}_2$. The last exceptional divisor $\mathcal{E}_3$ corresponds to the section $\mathcal{Q}_0$ and intersects the fiber above zero   on $\mathcal{E}_2$. It intersects $F_0$ at 
$Q_0^{(3)}$ on $\mathcal{E}_2$. \\

\begin{figure}[h!]

\centering
\subfloat[$Q_0=R_1=Q_1$]{\label{fig:in}

\begin{tikzpicture}[scale=1.5]
\coordinate (O) at (0:0);
\foreach \i in {1,...,4}{
\coordinate (A\i) at (\i * 90 -4 -45:2);
\coordinate (B\i) at (\i * 90 +94 -45 :2);
 \tkzDefBarycentricPoint(A\i=1,B\i=2)
  \tkzGetPoint{I\i};
   \tkzDefBarycentricPoint(A\i=2,B\i=1)
  \tkzGetPoint{J\i};
}

   \node[label=above: $R_0$]  at (I2) {$\circ$};
    \node[label=above: $S_0$]  at (I3) {$\circ$};
     \node[label=above: $P_0$]  at (I4) {$\circ$};
     
    \node[label=above: $S_1$]  at (J3) {$\circ$};
     \node[label=above: $P_1$]  at (J4) {$\circ$};
\draw [name path=A1--B1] (A1) to[ swap] (B1);
\draw [name path=A2--B2](A2) to[ swap] (B2);
\draw (A3) to[ swap] (B3);
\draw (A4) to[ swap] (B4);
\path [name intersections={of=A1--B1 and A2--B2,by=E}];
\node  [label=above left: $Q_0$] at (E) {$\circ$};
\end{tikzpicture}
}
\quad
\subfloat[After one Blowup]{\label{fig:in}
\begin{tikzpicture}[scale=1.5]
\coordinate (O) at (0:0);
\foreach \i in {1,...,4}{
\coordinate (A\i) at (\i * 90 -4 -45:2);
\coordinate (B\i) at (\i * 90 +94 -45 :2);
 \tkzDefBarycentricPoint(A\i=1,B\i=2)
  \tkzGetPoint{I\i};
   \tkzDefBarycentricPoint(A\i=2,B\i=1)
  \tkzGetPoint{J\i};
}

   \node[label=above: $R_0$]  at (I2) {$\circ$};
    \node[label=above: $S_0$]  at (I3) {$\circ$};
     \node[label=above: $P_0$]  at (I4) {$\circ$};
     
    \node[label=above: $S_1$]  at (J3) {$\circ$};
    \node[label=above: $P_1$]  at (J4) {$\circ$};

\draw (A3) to[ swap] (B3);
\draw (A4) to[ swap] (B4);
\tkzDefBarycentricPoint(A2=3,B2=1)
  \tkzGetPoint{C2};
\tkzDefBarycentricPoint(A1=1,B1=3)
  \tkzGetPoint{C1};
  \tkzDefBarycentricPoint(A2=3,B2=2)
  \tkzGetPoint{C22};
\tkzDefBarycentricPoint(A1=2,B1=3)
  \tkzGetPoint{C11};
\tkzDefBarycentricPoint(C22=-3,C11=1)
  \tkzGetPoint{C3};  
\tkzDefBarycentricPoint(C11=-3,C22=1)
  \tkzGetPoint{C4};  
 \draw [name path=A1--C1] (A1) to[ swap] (C1);
\draw (B2) to [swap] (C2);
\draw [blue, label=$E_1$, name path=C3--C4] (C3) --  (C4) node[pos=0.5, below left ] {$\mathcal{E}_1$};
\path [name intersections={of=A1--C1 and C3--C4,by=F}];
\node  [label=above left: $Q_0^{(1)}$] at (F) {$\circ$};
\end{tikzpicture}
 }
 \quad
\subfloat[After three Blowups]{\label{fig:in}
\begin{tikzpicture}[scale=1.5]
\coordinate (O) at (0:0);
\foreach \i in {1,...,4}{
\coordinate (A\i) at (\i * 90 -4 -45:2);
\coordinate (B\i) at (\i * 90 +94 -45 :2);
 \tkzDefBarycentricPoint(A\i=1,B\i=2)
  \tkzGetPoint{I\i};
   \tkzDefBarycentricPoint(A\i=2,B\i=1)
  \tkzGetPoint{J\i};
}

   \node[label=above: $R_0$]  at (I2) {$\circ$};
    \node[label=above: $S_0$]  at (I3) {$\circ$};
     \node[label=above: $P_0$]  at (I4) {$\circ$};
     
    \node[label=above: $S_1$]  at (J3) {$\circ$};
    \node[label=above: $P_1$]  at (J4) {$\circ$};

\draw (A3) to[ swap] (B3);
\draw (A4) to[ swap] (B4);
\tkzDefBarycentricPoint(A2=3,B2=1)
  \tkzGetPoint{C2};
\tkzDefBarycentricPoint(A1=1,B1=3)
  \tkzGetPoint{C1};
  \tkzDefBarycentricPoint(A2=3,B2=2)
  \tkzGetPoint{C22};
\tkzDefBarycentricPoint(A1=2,B1=3)
  \tkzGetPoint{C11};
\tkzDefBarycentricPoint(C22=-3,C11=1)
  \tkzGetPoint{C3};  
\tkzDefBarycentricPoint(C11=-3,C22=1)
  \tkzGetPoint{C4};  
  \tkzDefBarycentricPoint(C3=1,C1=4)
  \tkzGetPoint{C5};  
    \tkzDefBarycentricPoint(C5=-6,C4=1)
  \tkzGetPoint{C6};  
 \tkzDefBarycentricPoint(C6=1,C4=3)
  \tkzGetPoint{C7};
  \tkzDefBarycentricPoint(C7=-3,A1=1)
  \tkzGetPoint{C8};
   \tkzDefBarycentricPoint(C6=1,C4=2)
  \tkzGetPoint{H};
\draw (B2) to [swap] (C2);
\draw [blue, label=$E_1$, name path=C3--C4] (C3) --  (C1) node[pos=0.5, below left ] {$\mathcal{E}_1$};
\draw [red, label=$E_2$] (C6)--(C4) node[pos=0.5, below  ] {$\mathcal{E}_2$};;
\draw (C8)--(A1);
\node[label=above: $Q_0^{(3)}$]  at (H) {$\circ$};
\end{tikzpicture}} \caption{The fiber above $0$ when $Q_0=Q_1=R_1$}\label{fig:threepointscorner}
\end{figure}
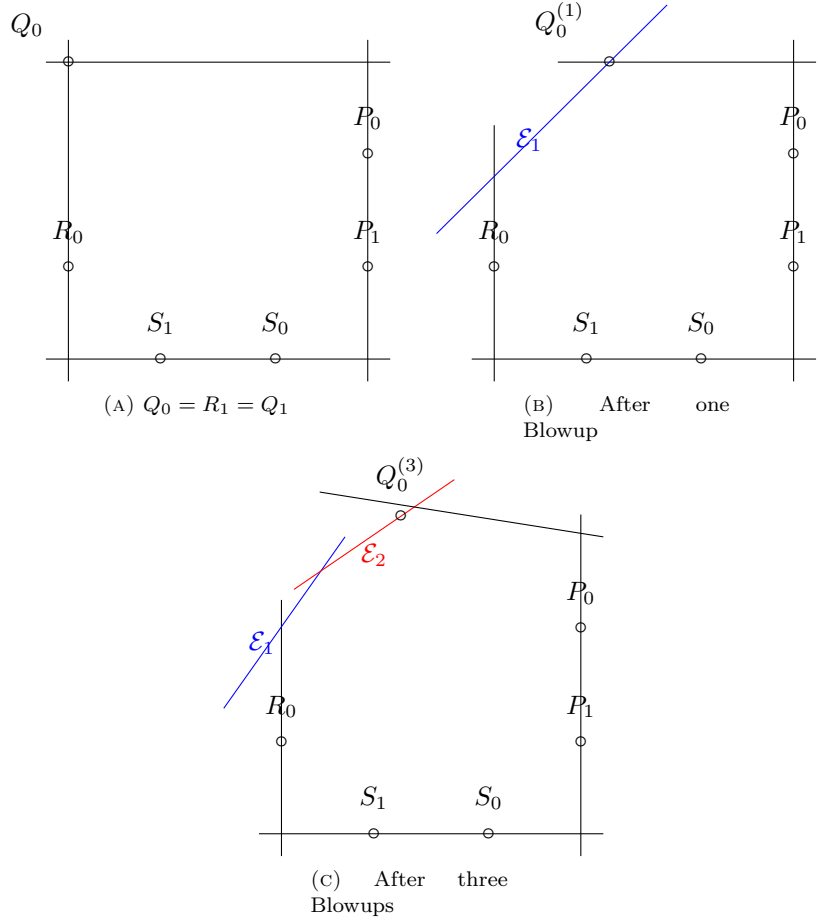

Since the curve $\Etproj$ is non singular, one can not have four points in a corner. The discussion above shows that the singular fiber above $0$ is an $I_n$ with 
\begin{itemize}
\item $n=4$ when all the base points are distinct or they are equal on a line,
\item $n=5$ when  for instance $Q_0=Q_1$,
\item $n=6$ when   for instance $Q_0=Q_1=R_1$,
\item $n=7$ when  for instance $Q_0=Q_1=R_1$ and $P_1=S_0$,
\item $n=8$ when for instance  $Q_0=Q_1=R_1$ and $P_1=S_0=P_0$,
\item $n=9$ when for instance $Q_0=Q_1=R_1$,  $P_1=S_0=P_0$ and $R_0=S_1$.
\end{itemize}
In this last case, one has $\tau^3(S_0)=S_0$ so that the group of the walk is finite. Indeed, the group of the walk will be always finite when $n=9$ because the root lattice is $A_8$ (see \cite[Table 8.2]{SchuttShiodaBook}).

 \subsubsection{Some examples}\label{sec:exs}
The fiber $F_0$ above zero is an $I_{n}$ and the contribution of this fiber to the height of a section $\mathcal{Q}$ is defined as follows.  Let $\mathcal{O}$ be the zero section. The fiber $F_0$ is a cycle of $n$ components $\Theta_i$ for $i=0,\dots,n-1$. The component of $F_0$ that meets the section $\mathcal{O}$ is denoted $\Theta_0$ and we number the components cyclically, that is, $\Theta_i$ meets $\Theta_j$ if and only if $|i-j| \equiv 1 \mbox{ mod } n$. The contribution of $F_0$ to the height of  a section $\mathcal{P}$  is equal to $\frac{i(n-i)}{n}$ when $\mathcal{P}$ meets $F_0$ on the component $\Theta_i$. With the process detailed in \ref{subsubsec:fiberzero}, one can easily determine the contribution of $F_0$ to the height of the section. This allows one to refine the algorithm presented in Section \ref{sec:algo} by lowering the number of possibilities for the height. In this section, we present this refinement via the study of \charl{three} weighted models.\\

\noindent{\emph{{\mfs Example~\ref{exa:algo} revisited}: The weighted model  $w_{IIC.2}$.}}\label{sec:exampleIIB6}\\
 In this paragraph, we show how the computation of the contribution of the fiber above zero allows to drastically simplify the algorithm  presented in \ref{sec:algo}. We will illustrate this on an example and we will   study the $D$-algebraicity of  the weighted model $w_{IIC.2}$, which corresponds to $d_{1,-1}=d_{1,0}=d_{-1,1}=0$. For this model, we have 

\begin{itemize}
\item $Q_1=R_1 \neq Q_0$,
\item $P_0=P_1=S_1$. 
\end{itemize}
Following the method detailed in Section \ref{subsubsec:fiberzero}, the fiber above zero  given by Figure \ref{fig:exampleweightedIIC2}.  

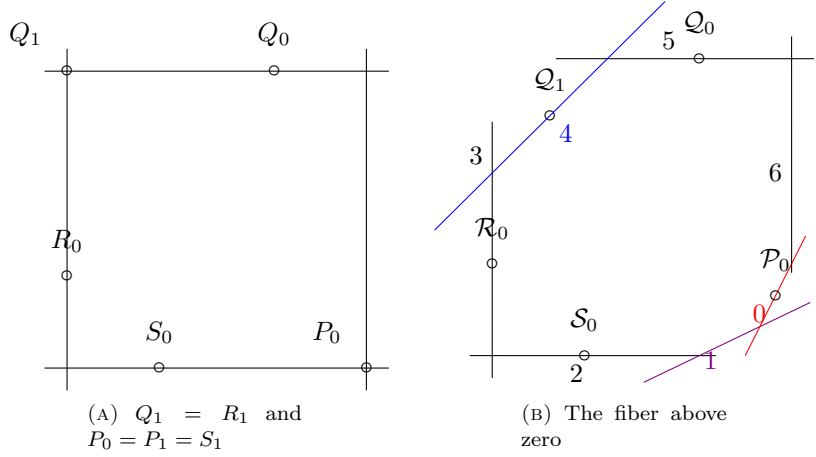
\begin{figure}[h!]

\centering
\subfloat[$Q_1=R_1$ and $P_0=P_1=S_1$]{\label{fig:in}

\begin{tikzpicture}[scale=1.5]
\coordinate (O) at (0:0);
\foreach \i in {1,...,4}{
\coordinate (A\i) at (\i * 90 -4 -45:2);
\coordinate (B\i) at (\i * 90 +94 -45 :2);
 \tkzDefBarycentricPoint(A\i=1,B\i=2)
  \tkzGetPoint{I\i};
   \tkzDefBarycentricPoint(A\i=2,B\i=1)
  \tkzGetPoint{J\i};
}

   \node[label=above: $R_0$]  at (I2) {$\circ$};
     
      \node[label=above: $Q_0$]  at (J1) {$\circ$};
    \node[label=above: $S_0$]  at (J3) {$\circ$};
\draw [name path=A1--B1] (A1) to[ swap] (B1);
\draw [name path=A2--B2](A2) to[ swap] (B2);
\draw[name path=A3--B3] (A3) to[ swap] (B3);
\draw[name path=A4--B4] (A4) to[ swap] (B4);
\path [name intersections={of=A1--B1 and A2--B2,by=E}];
\node  [label=above left: $Q_1$] at (E) {$\circ$};

\path [name intersections={of=A3--B3 and A4--B4,by=F}];
\node  [label=above left: $P_0$] at (F) {$\circ$};
\end{tikzpicture}
}
\quad
\subfloat[The fiber above zero]{\label{fig:in}
\begin{tikzpicture}[scale=1.5]
\coordinate (O) at (0:0);
\foreach \i in {1,...,4}{
\coordinate (A\i) at (\i * 90 -4 -45:2);
\coordinate (B\i) at (\i * 90 +94 -45 :2);
 \tkzDefBarycentricPoint(A\i=1,B\i=2)
  \tkzGetPoint{I\i};
   \tkzDefBarycentricPoint(A\i=2,B\i=1)
  \tkzGetPoint{J\i};
}

   \node[label=above: $\mathcal{R}_0$]  at (I2) {$\circ$};
     
      \node[label=above: $\mathcal{Q}_0$]  at (J1) {$\circ$};
    \node[label=above: $\mathcal{S}_0$]  at (J3) {$\circ$};

\tkzDefBarycentricPoint(A2=3,B2=1)
  \tkzGetPoint{C2};
\tkzDefBarycentricPoint(A1=1,B1=3)
  \tkzGetPoint{C1};
  \tkzDefBarycentricPoint(A2=3,B2=2)
  \tkzGetPoint{C22};
\tkzDefBarycentricPoint(A1=2,B1=3)
  \tkzGetPoint{C11};
\tkzDefBarycentricPoint(C22=-3,C11=1)
  \tkzGetPoint{C3};  
\tkzDefBarycentricPoint(C11=-3,C22=1)
  \tkzGetPoint{C4};  
  
  \tkzDefBarycentricPoint(C3=1,C4=1)
  \tkzGetPoint{Q0};
  \node[label=above: $\mathcal{Q}_1$]  at (Q0) {$\circ$};

 \tkzDefBarycentricPoint(A3=1,B3=2)
  \tkzGetPoint{R};
\tkzDefBarycentricPoint(B3=5,B4=1)
  \tkzGetPoint{R1};
  \tkzDefBarycentricPoint(R=-3,R1=1)
  \tkzGetPoint{R2};
   \tkzDefBarycentricPoint(A3=1,B3=4)
  \tkzGetPoint{S};
  \tkzDefBarycentricPoint(B3=5,B4=3)
  \tkzGetPoint{S1};
    \tkzDefBarycentricPoint(A3=2,B3=5)
  \tkzGetPoint{S3};
   \tkzDefBarycentricPoint(A4=9,B4=4)
  \tkzGetPoint{S2};
   \tkzDefBarycentricPoint(S=1,S1=1)
  \tkzGetPoint{S4};

\draw (A1)--(C1) node[pos=0.5, above  left ]{$5$};
\draw (B2)--(C2)node[pos=0.8, above left ]{$3$};
\draw [blue] (C3) --(C4) node[pos=0.5, below  right ]{$4$};
  \draw (A3)-- (S3)node[pos=0.5, below  left ]{$2$};
\draw[violet](R2)-- (R1)node[pos=0.5, below  left ]{$1$};
\draw[red](S) --(S1) node[pos=0.5, below  left ]{$0$};
\draw (B4)--(S2)node[pos=0.5, below  left ]{$6$};
 \node[label=above: $\mathcal{P}_0$]  at (S4) {$\circ$};
\end{tikzpicture}
 }\caption{Fiber above zero for  $w_{IIC.2}$}\label{fig:exampleweightedIIC2}
\end{figure}

In Figure \ref{fig:exampleweightedIIC2}, we abuse notation and denote by $\mathcal{Q}_i, \mathcal{P}_i$ the intersections of the sections
with the fiber $F_0$.

As detailed in Section \ref{sec:algo}, the model is decoupled if and only if there exists $n$ such that 
$Q_0=\tau^n(P_0)$ (Note that since $P_0$ is fixed by $\ita$, one has $P_0 \sim Q_0$ if and only if $P_0 \sim Q_1$. Choosing $P_0$ as the zero of $\Etproj$, we must decide if there exists an integer $n$ such that 
$Q_0=n\tau(P_0)=nS_0$. The fiber above zero is an $I_7$, which corresponds to a root lattice $A_6$. By \cite[Table 8.2]{SchuttShiodaBook}, the root lattice $T$ is either $A_6$ or $A_6 \oplus A_1$.   Numbering the components of the fiber above zero as in Figure \ref{fig:exampleweightedIIC2},   we find that   the height of the points $Q_0$ and $S_0$ are given by
\begin{itemize}
\item $\hhat(Q_0)=2 -\frac{5(7-5)}{7} -\frac{\epsilon_1}{2}$,
\item $\hhat(S_0)=2 -\frac{2(7-2)}{7}  -\frac{\epsilon_2}{2}$,
\end{itemize}
where $\epsilon_1,\epsilon_2 \in \{0,1\}$ depending on the intersection of $\mathcal{Q}_0$ and $\mathcal{S}_0$ with a putative singular fiber of  root lattice $A_1$. Note that the height of $S_0$ is never zero so that the point $\tau(P_0)$ is not torsion and the group of the walk is infinite (see the remarks following Lemma~\ref{lem:inforder} and Remark \ref{rmk:finiteheight}). Then, $\hhat(Q_0)=n^2\hhat(S_0)$ is equivalent to $8 -7 \epsilon_1=n^2(8 -7 \epsilon_1)$ and the only solution is $n^2=1$ that is $n=\pm 1$. Since $\tau(P_0)= S_0 \neq Q_0$, the integer  $n$ must be equal to $-1$. For the weighted model $w_{IIC.2}$, the condition $Q_0=\tau^{-1}(P_0)$ is equivalent to  
\begin{equation}\label{eq:decoupledwIIC2}
d_{0,1}d_{0,-1} -d_{1,1}d_{-1,-1}=0.
\end{equation}  When the model $w_{IIC.2}$ is unweighted, the condition \eqref{eq:decoupledwIIC2} is satisfied so that the unweighted $w_{IIC.2}$ is $D$-algebraic.

Once one knows that the weighted model is decoupled, it is quite easy to find the certificate for $b$. Indeed, thanks to the orbit residue criteria, one knows  the distribution of the poles of $b$ on $\tau$-orbits. Finding the certificate of $b$ is just a question of finding an elliptic function with prescribed set of poles and residues. 

The weighted model $w_{IIC.2}$ is decoupled if and only if $d_{0,1}d_{0,-1} -d_{1,1}d_{-1,-1}=0$  if and only if $Q_0=\ita(S_0)$. In that situation, the residues and poles of $b$ are as follows:
$$
\begin{array}{|c|c|c|c|}
\hline
\mbox{Points} &  S_0=\tau( P_0 )&P_{0} & \tau^{-1}(P_0)=Q_0   \\ \hline
\mbox{Residues of order 1}& \a &  -2\a & \a  \\ \hline

\end{array}
$$
In $\C(\Etproj)$, the function $h=\frac{1}{y}$ has the following residues and poles
$$
\begin{array}{|c|c|c|}
\hline
\mbox{Points} &  S_0=\tau( P_0 )&P_{0}   \\ \hline
\mbox{Residues of order 1}& -\b &  \b   \\ \hline

\end{array}
$$
 so that for any $\lambda \in \C^*$, the function $\tau(\lambda h) -\lambda h$ has 
  the following residues and poles
 $$
\begin{array}{|c|c|c|c|}
\hline
\mbox{Points} &  S_0=\tau( P_0 )&P_{0} & \tau^{-1}(P_0)=Q_0   \\ \hline
\mbox{Residues of order 1}& \lambda \b &  -2\lambda\b& \lambda \b  \\ \hline

\end{array}
$$ 
Then  $\tau(\frac{\a}{\b y} )-\frac{\a}{\b y}  $ and $b$ have same poles and residues so that  there exists $c \in \C$ such that $b= \tau(\frac{\a}{\b y} ) -(\frac{\a}{\b y} ) +c$. It is easily seen that $c$ must be zero since $\ita(b)=-b$ and 
$\ita\left( \tau(\frac{\a}{\b y} ) -(\frac{\a}{\b y} )   \right) = -(\tau(\frac{\a}{\b y} ) -(\frac{\a}{\b y} ))$.  Therefore, the function $\frac{\a}{\b y}$ is a certificate for $b$. To compute the residues $\a$ and $\b$, we generalize  \cite{BBMR15} to the decoupled  weighted case and, using \eqref{eq:decoupledwIIC2}, we  note that  
\begin{equation}\label{eq:yitay}
y \ita(y) =\frac{(d_{-1,-1} +d_{0,-1}x)}{ d_{0,1} x + d_{1,1}x^2}=\frac{d_{-1,-1}}{d_{0,1}} \frac{1}{x}. 
\end{equation}
Then, one finds that 
\begin{align*}
\a=\Res_{Q_0}(b)=& -\Res_{Q_0}(xy)=-\frac{d_{-1,-1}}{d_{0,1}}\Res_{Q_0}(\frac{1}{\ita(y)})\\
=&\frac{d_{-1,-1}}{d_{0,1}}\Res_{\ita(Q_0)}(\frac{1}{y})=- \frac{d_{-1,-1}}{d_{0,1}}\b,
\end{align*}
where we use $\Res_{\ita(P)}(f)=-\Res_P(\ita(f))$ for any $P \in \Etproj, f \in \C(\Etproj)$ and $\ita(Q_0)=S_0$. This proves that the function $\frac{-d_{0,1}}{d_{-1,-1}y}$ is a certificate for $b$. \\
\begin{exa}{\emph{The weighted model $IB.6$.}} 
 This weighted model  corresponds to $d_{1,-1}=d_{1,0}=0$. When unweighted, it was called $IB.6$ and we keep this terminology for the weighted model.  In that situation,  $P_0=P_1=S_1$. The fiber above zero is 

\begin{figure}[h!]

\centering
\subfloat[ $P_0=P_1=S_1$]{\label{fig:in}

\begin{tikzpicture}[scale=1.5]
\coordinate (O) at (0:0);
\foreach \i in {1,...,4}{
\coordinate (A\i) at (\i * 90 -4 -45:2);
\coordinate (B\i) at (\i * 90 +94 -45 :2);
 \tkzDefBarycentricPoint(A\i=1,B\i=2)
  \tkzGetPoint{I\i};
   \tkzDefBarycentricPoint(A\i=2,B\i=1)
  \tkzGetPoint{J\i};
}

  \node[label=above: $Q_1$]  at (I1) {$\circ$};
   \node[label=above: $R_0$]  at (I2) {$\circ$};
     \node[label=above: $R_1$]  at (J2) {$\circ$};
      \node[label=above: $Q_0$]  at (J1) {$\circ$};
    \node[label=above: $S_0$]  at (J3) {$\circ$};
\draw [name path=A1--B1] (A1) to[ swap] (B1);
\draw [name path=A2--B2](A2) to[ swap] (B2);
\draw[name path=A3--B3] (A3) to[ swap] (B3);
\draw[name path=A4--B4] (A4) to[ swap] (B4);

\path [name intersections={of=A3--B3 and A4--B4,by=F}];
\node  [label=above left: $P_0$] at (F) {$\circ$};
\end{tikzpicture}
}
\quad
\subfloat[The fiber above zero]{\label{fig:in}
\begin{tikzpicture}[scale=1.5]
\coordinate (O) at (0:0);
\foreach \i in {1,...,4}{
\coordinate (A\i) at (\i * 90 -4 -45:2);
\coordinate (B\i) at (\i * 90 +94 -45 :2);
 \tkzDefBarycentricPoint(A\i=1,B\i=2)
  \tkzGetPoint{I\i};
   \tkzDefBarycentricPoint(A\i=2,B\i=1)
  \tkzGetPoint{J\i};
}
 \node[label=above: $\mathcal{R}_1$]  at (J2) {$\circ$};
 
   \node[label=above: $\mathcal{R}_0$]  at (I2) {$\circ$};
     
      \node[label=above: $\mathcal{Q}_0$]  at (J1) {$\circ$};
  \node[label=above: $\mathcal{Q}_1$]  at (I1) {$\circ$};
    \node[label=above: $\mathcal{S}_0$]  at (J3) {$\circ$};

\tkzDefBarycentricPoint(A2=3,B2=1)
  \tkzGetPoint{C2};
\tkzDefBarycentricPoint(A1=1,B1=3)
  \tkzGetPoint{C1};
  \tkzDefBarycentricPoint(A2=3,B2=2)
  \tkzGetPoint{C22};
\tkzDefBarycentricPoint(A1=2,B1=3)
  \tkzGetPoint{C11};
\tkzDefBarycentricPoint(C22=-3,C11=1)
  \tkzGetPoint{C3};  
\tkzDefBarycentricPoint(C11=-3,C22=1)
  \tkzGetPoint{C4};  
  
  \tkzDefBarycentricPoint(C3=1,C4=1)
  \tkzGetPoint{Q0};

 \tkzDefBarycentricPoint(A3=1,B3=2)
  \tkzGetPoint{R};
\tkzDefBarycentricPoint(B3=5,B4=1)
  \tkzGetPoint{R1};
  \tkzDefBarycentricPoint(R=-3,R1=1)
  \tkzGetPoint{R2};
   \tkzDefBarycentricPoint(A3=1,B3=4)
  \tkzGetPoint{S};
  \tkzDefBarycentricPoint(B3=5,B4=3)
  \tkzGetPoint{S1};
    \tkzDefBarycentricPoint(A3=2,B3=5)
  \tkzGetPoint{S3};
   \tkzDefBarycentricPoint(A4=9,B4=4)
  \tkzGetPoint{S2};
   \tkzDefBarycentricPoint(S=1,S1=1)
  \tkzGetPoint{S4};

\draw (A1)--(B1) node[pos=0.5, above  left ]{$4$};
\draw (B2)--(A2)node[pos=0.8, above left ]{$3$};
  \draw (A3)-- (S3)node[pos=0.5, below  left ]{$2$};
\draw[violet](R2)-- (R1)node[pos=0.5, below  left ]{$1$};
\draw[red](S) --(S1) node[pos=0.5, below  left ]{$0$};
\draw (B4)--(S2)node[pos=0.5, below  left ]{$5$};
 \node[label=above: $\mathcal{P}_0$]  at (S4) {$\circ$};
\end{tikzpicture}
 }\caption{Fiber above zero for  $w_{IB.6}$}\label{fig:exampleweightedIB6}
\end{figure}
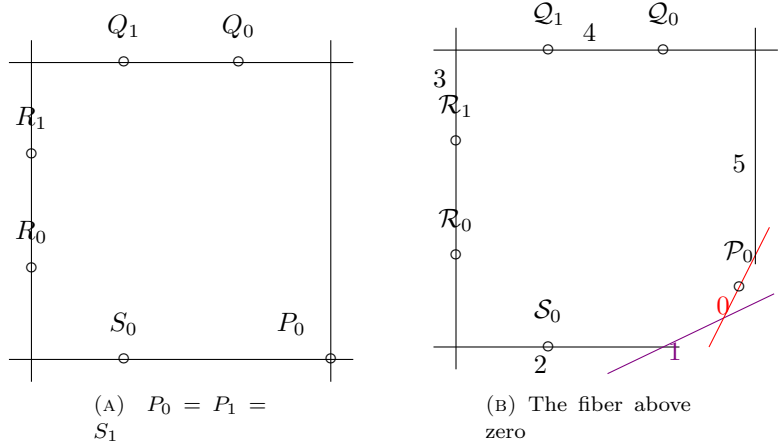


As described in Section \ref{sec:algo}, the model is decoupled if and only if there exists $n$ such that 
$Q_0=\tau^n(P_0)$ (Note that since $P_0$ is fixed by $\ita$, one has $P_0 \sim Q_0$ if and only if $P_0 \sim Q_1$. Choosing $P_0$ as the zero of $\Etproj$, we must decide if there exists an integer $n$ such that 
$Q_0=n\tau(P_0)=nS_0$. The fiber above zero is an $I_6$, which corresponds to a root lattice $A_5$. By \cite[Table 8.2]{SchuttShiodaBook}, the root lattice $T$ is either $A_5,A_5 \oplus A_1, A_5 \oplus A_1^2, A_5 \oplus A_2,A_5 \oplus A_2 \oplus A_1 $.   Numbering the components of the fiber above zero as in Figure \ref{fig:exampleweightedIB6}, we find that   the heights of the points $Q_0$ and $S_0$ are given by
\begin{itemize}
\item $\hhat(Q_0)=2 -\frac{4(6-4)}{6} -\frac{\epsilon_1}{2}-\frac{2\epsilon_2}{3}$,
\item $\hhat(S_0)=2 -\frac{2(6-2)}{6}  -\frac{\eta_1}{2}-\frac{2\eta_2}{3} $,
\end{itemize}
where $\epsilon_i, \eta_i \in \{0,1\}$ except for the root lattice $A_5 \oplus A_1^2$, where  the height of the points $Q_0$ and $S_0$ are given by
\begin{itemize}
\item $\hhat(Q_0)=2 -\frac{4(6-4)}{6} -\frac{\epsilon_1}{2}-\frac{\epsilon_2}{2}$,
\item $\hhat(S_0)=2 -\frac{2(6-2)}{6}  -\frac{\eta_1}{2}-\frac{\eta_2}{2} $,
\end{itemize}
where $\epsilon_i, \eta_i \in \{0,1\}$. {Note that $\hhat(S_0)$ might be equal to zero if $\eta_1=0,\eta_2=1$. In that case, the group of the walk is finite and the generating series are holonomic. If $\hhat(S_0) \neq 0$ then,}
it is easily seen that if  $\hhat(Q_0)=n^2\hhat(S_0)$  then $n^2$ equals $1$ or  $4$. Since $Q_0 \neq S_0$ and $\itb(Q_0)=Q_1)\neq \ita(S_0)=S_1$, it is easily seen that $n$ must be equal to $-1$ or $-2$.  A simple computation (see \cite{HSurl}) shows that 
$Q_0=\tau^{-1}(P_0)$ if and only if 
\begin{equation}\label{eq:decoupledIB6}
d_{-1,1}d_{0,-1}^2 -d_{0,1}d_{-1,-1}d_{0,-1} +d_{1,1}d_{-1,-1}^2=0.
\end{equation}

The condition $Q_0=\tau^{-2}(P_0)$ is impossible  (see \cite{HSurl}). 
Nonetheless, it is easily seen that if the walk is unweighted then the condition \eqref{eq:decoupledIB6} {does not hold}. Therefore, the unweighted model $IB.6$ has a $D$-transcendental generating series. \end{exa} 

\begin{rem} In \cite[Proposition 5.1]{DHRS19}, the authors show that if  $\delta^x=d_{1,0}^2-4d_{1,-1}d_{1,1}$  or $\delta_y=d_{0,1}^2-4d_{-1,1}d_{1,1}$ is not a square in $\Q(d_{i,j})$ then the  generating series are differentially  hypertranscendental. For the unweighted model $IB.6$, one has
 $\delta^x=0$ and $\delta_y=-3$ so that the  generating series is differentially transcendental in that case. \cite[Theorem 35]{DHRS19} shows that \cite[Proposition 5.1]{DHRS19} remains valid in the weighted case. If condition \ref{eq:decoupledIB6} is satisfied then $\delta_x=0$ and 
 $\delta_y=  \left(\frac{(d_{0,1}d_{0,-1} -2d_{1,1}d_{-1,-1})}{d_{0,-1}}\right)^2$  is a square in $\Q(d_{i,j})$. Thus, our computation  gives a necessary and sufficient condition  for the $D$-algebraicity weighted model $IB.6$ and generalize  \cite[Theorem 35]{DHRS19} for this model.
 \end{rem}

\vspace{.1in}

\begin{exa}{\emph{The weighted Gouyou-Beauchamps model}}\label{sec:gouyoubeauchamp}

{In \cite{CourtielMelczerMishnaRaschel}, the authors adapt some probabilistic notions such as the drift to define subfamilies of weighted models, which they call \emph{universality classes} since they met common algebraic behaviour. They consider the \emph{generic central weighting of the Gouyou-Beauchamps model} given by Figure~\ref{fig:gouyoubeauchamp}.}

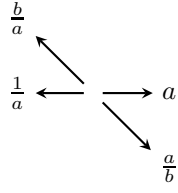
\begin{figure}

{\begin{tikzpicture}
\node (A) at (-1,1) {$\frac{b}{a}$} ;
\node (B) at (0,0) {};
\node (C) at (-1,0) {$\frac{1}{a}$}; 
\node (D) at (1,0) {$a$};
\node (E) at (1,-1) {$\frac{a}{b}$};
\draw[thick,->] (B) -- (A) ;
\draw[thick,->] (B) -- (C);
\draw[thick,->] (B) -- (D);
\draw[thick,->] (B) -- (E) ;

\end{tikzpicture}}
\caption{Generic central weighting of the Gouyou-Beauchamps model}\label{fig:gouyoubeauchamp}
\end{figure}

In \cite{CourtielMelczerMishnaRaschel}, the authors showed that the group of the models of Figure \ref{fig:gouyoubeauchamp}  was the dihedral group $D_8$ and they study the asymptotics of the combinatorial sequence. In this section, we weight the Gouyou-Beauchamps model with arbitrary weights $d_{-1,1},d_{1,-1},d_{0,1},d_{1,0}$ and prove the following proposition. 

\begin{prop}
The generating function of the weighted Gouyou-Beauchamps model is differentially algebraic if and only if 
\begin{equation}\label{eq:gouyoubeauchamp}
d_{1,0}d_{-1,0}- d_{-1,1}d_{1,-1}=0.
\end{equation} 
If \eqref{eq:gouyoubeauchamp} is satisfied, the group of the walk is either   $D_4$ or  $D_8$ and the generating function is $D$-finite.
\end{prop}
\begin{proof}
In that situation, $S_1=S_0=R_0$, $P_0=Q_1=Q_0$ and the fiber above zero is as follows

\begin{figure}[h!]

\centering
\subfloat[ $P_0=Q_1=Q_0$ and $S_1=S_0=R_0$ ]{\label{fig:in}

\begin{tikzpicture}[scale=1.5]
\coordinate (O) at (0:0);
\foreach \i in {1,...,4}{
\coordinate (A\i) at (\i * 90 -4 -45:2);
\coordinate (B\i) at (\i * 90 +94 -45 :2);
 \tkzDefBarycentricPoint(A\i=1,B\i=2)
  \tkzGetPoint{I\i};
   \tkzDefBarycentricPoint(A\i=2,B\i=1)
  \tkzGetPoint{J\i};
}

     \node[label=above: $R_1$]  at (J2) {$\circ$};
     \node[label=above: $P_1$]  at (J4) {$\circ$};
\draw [name path=A1--B1] (A1) to[ swap] (B1);
\draw [name path=A2--B2](A2) to[ swap] (B2);
\draw[name path=A3--B3] (A3) to[ swap] (B3);
\draw[name path=A4--B4] (A4) to[ swap] (B4);

\path [name intersections={of=A1--B1 and A4--B4,by=F}];
\node  [label=above left: $P_0$] at (F) {$\circ$};

\path [name intersections={of=A3--B3 and A2--B2,by=F}];
\node  [label=above left: $S_0$] at (F) {$\circ$};
\end{tikzpicture}
}
\quad
\subfloat[The fiber above zero]{\label{fig:in}
\begin{tikzpicture}[scale=1.5]
\coordinate (O) at (0:0);
\foreach \i in {1,...,4}{
\coordinate (A\i) at (\i * 90 -4 -45:2);
\coordinate (B\i) at (\i * 90 +94 -45 :2);
 \tkzDefBarycentricPoint(A\i=1,B\i=2)
  \tkzGetPoint{I\i};
   \tkzDefBarycentricPoint(A\i=2,B\i=1)
  \tkzGetPoint{J\i};
}
 \node[label=above: $\mathcal{R}_1$]  at (J2) {$\circ$};
 
  
     \node[label=below: $\mathcal{P}_1$]  at (I4) {$\circ$};


\tkzDefBarycentricPoint(A4=1,B4=3)
  \tkzGetPoint{C4};

\tkzDefBarycentricPoint(A1=8,B1=1)
  \tkzGetPoint{C1};
  \tkzDefBarycentricPoint(C1=-3,C4=1)
  \tkzGetPoint{C14};
   \tkzDefBarycentricPoint(C4=-3,C1=1)
  \tkzGetPoint{C41};
\draw[blue ](C41)--(C14) node[pos=0.3, below  left ]{$7$};
\tkzDefBarycentricPoint(A4=1,B4=8)
  \tkzGetPoint{D4};
\tkzDefBarycentricPoint(A1=3,B1=1)
  \tkzGetPoint{D1};
  \tkzDefBarycentricPoint(D1=-3,D4=1)
  \tkzGetPoint{D14};
   \tkzDefBarycentricPoint(D4=-3,D1=1)
  \tkzGetPoint{D41};
    \tkzDefBarycentricPoint(D41=1,D14=1)
  \tkzGetPoint{Q41};
\draw[red] (D41)--(D14)  node[pos=0.3, below  left ]{$0$};
\node[label=above: $\mathcal{P}_0$]  at (Q41) {$\circ$};
\tkzDefBarycentricPoint(D4=1,C4=1)
  \tkzGetPoint{E4};

\tkzDefBarycentricPoint(D1=1,C1=1)
  \tkzGetPoint{E1};

\tkzDefBarycentricPoint(A2=1,B2=2)
  \tkzGetPoint{H2};

\tkzDefBarycentricPoint(A3=5,B3=1)
  \tkzGetPoint{H3};
  \tkzDefBarycentricPoint(H2=-3,H3=1)
  \tkzGetPoint{H23};
   \tkzDefBarycentricPoint(H3=-3,H2=1)
  \tkzGetPoint{H32};
\draw[violet ](H23)--(H32) node[pos=0.3, below  left ]{$3$};

  \tkzDefBarycentricPoint(B3=1,A3=2)
  \tkzGetPoint{K3};
\tkzDefBarycentricPoint(B2=3,A2=1)
  \tkzGetPoint{K2};
  \tkzDefBarycentricPoint(K3=-3,K2=1)
  \tkzGetPoint{K32};
   \tkzDefBarycentricPoint(K2=-3,K3=1)
  \tkzGetPoint{K23};
    \tkzDefBarycentricPoint(K23=1,K32=1)
  \tkzGetPoint{R32};
\draw[green] (K32)--(K23)  node[pos=0.3, below  left ]{$4$};
\node[label=above: $\mathcal{R}_0$]  at (R32) {$\circ$};

  \tkzDefBarycentricPoint(H3=1,K3=1)
  \tkzGetPoint{F3};

\tkzDefBarycentricPoint(K2=1,H2=1)
  \tkzGetPoint{F2};

\draw (F3)--(B3)node[pos=0.5, above left ]{$5$};
  
\draw (A4)--(E4) node[pos=0.5, below  left ]{$6$};
\draw (E1)--(B1) node[pos=0.3, below  left ]{$1$};
\draw (F2)--(A2)node[pos=0.5, above left ]{$2$};
\end{tikzpicture}
 }\caption{Fiber above zero for the weighted Gouyou-Beauchamps}\label{fig:exampleweightedGB}
\end{figure}
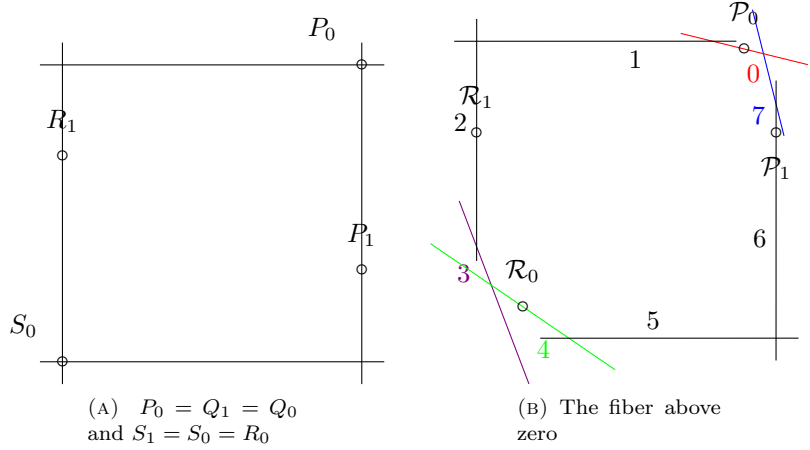

The fiber above zero is an $I_8$, which corresponds to a root lattice $A_7$. By \cite[Table 8.2]{SchuttShiodaBook}, the root lattice $T$ is either $A_7$ or $A_7 \oplus A_1$. In the latter case, the Mordell Weil group is $\Z/4\Z$ which shows that any point of the kernel curve
is  of order less than or equal to  $4$. This proves that the group of the walk is either $D_4$ or $D_8$. Following \cite[Lemma 3.3]{Tsuda}, one can compute the discriminant of the Kernel curve and one finds (see the {\sc Maple} calculation at \cite{HSurl} for this calculation and the ones that follow):

\begin{align}\label{eq:discrimGouyouBeauchamps}
\Delta:=&d_{1, 0}^2 d_{1, -1}^2 d_{-1, 1}^2 d_{-1, 0}^2 t^8(16t^4 d_{1, 0}^2 d_{-1, 0}^2-32 t^4 d_{1, 0}d_{-1, 0}d_{1, -1}d_{-1, 1}\\
&+16t^4d_{1, -1}^2 d_{-1, 1}^2-8t^2 d_{1, 0}d_{-1, 0}-8t^2d_{1, -1}d_{-1, 1}+1) \notag
\end{align}

By Tate's algorithm, the existence of a singular fiber of type $I_2$ which  would give a contribution $A_1$ to the lattice  is equivalent to the vanishing of the discriminant $\delta$ of $16t^4 d_{1, 0}^2 d_{-1, 0}^2-32 t^4 d_{1, 0}d_{-1, 0}d_{1, -1}d_{-1, 1}+16t^4d_{1, -1}^2 d_{-1, 1}^2-8t^2 d_{1, 0}d_{-1, 0}-8t^2d_{1, -1}d_{-1, 1}+1$. A {\sc Maple} computation yields  
\begin{align}
\delta=&16777216(d_{1, 0}^2 d_{-1, 0}^2-2 d_{1, 0} d_{-1, 0}d_{1, -1}d_{-1, 1}+d_{1, -1}^2 d_{-1, 1}^2) d_{1, 0}^2 d_{-1, 0}^2 d_{1, -1}^2 d_{-1, 1}^2. 
\end{align}

Since the weights are nonzero, the vanishing of $\delta$ is equivalent to 
$ (\frac{ d_{1,0}d_{-1,0}}{d_{1,-1}d_{-1,1}} -1)^2=0$ that is to $d_{1,0}d_{-1,0}=d_{1,-1}d_{-1,1}$. If $d_{1,0}d_{-1,0} \neq d_{1,-1}d_{-1,1}$, the group of the walk is infinite and the model is not decoupled because $P_0$ and $Q_0$ are fixed by an involution (see Proposition \ref{prop:fixedbytwoinvolutions}). This ends the proof.
\end{proof}

Note that  Condition \eqref{eq:gouyoubeauchamp} is automatically fulfilled by the generic central weightings of the Gouyou-Beauchamps model.
These examples illustrate how the $D$-algebraicity of a model does depend on the configuration of the base points. It is conditioned by certain algebraic relations on the weights of the model and the classification of unweighted models in terms of $D$-algebraic and $D$-transcendental ones is in a certain sense accidental since the $D$-algebraic models corresponds to the cases where the algebraic relations are satisfied when all the weights are equal. \end{exa}

\appendix
\section{Poles and Residues}\label{appendix:PR} In this section we collect various technical facts concerning the poles and residues of rational functions on $\Etproj$, that is, elements of $\CX(\Etproj)$. We will assume throughout this section that $\Etproj$ is an elliptic curve {endowed with two involutions $\ita,\itb$. We denote by  $P$ the point of $\Etproj$ such that $\tau=\itb\circ \ita $ is the translation by $P$}. In our discussions below, we need to expand elements of $\CX(\Etproj)$ in power series at points of $\Etproj$ and compare the expansions at various points.  In order to do this in a consistent way  the following was introduced in \cite{DHRS}

\begin{define}\label{def:coherentlocparam} Let $\mathcal{S} = \{ u_Q \ | \ Q\in \Etproj\}$ be a set of local parameters at the points of $\Etproj$.  We say $S$  is a {\em coherent set of local parameters} if for any $Q \in \Etproj$, 
\[ u_{\tau^{-1}(Q)} = \tau(u_Q).\]
Note that $\tau^{-1}(Q) = Q\ominus P$, where $\ominus$ is subtraction in the group structure of the elliptic curve.
\end{define}
A coherent set of local parameters always exits. To see this, Let $O$ be the origin of the group law on the  elliptic curve $\Etproj$ and, for any $Q \in \Etproj(\CX)$ let $\tau_Q$ be the translation by $Q$. The map $\tau_Q$ induces and isomorphism $\tau_Q: \CX(\Etproj) \rightarrow  \CX(\Etproj)$ (here we abuse notation and use the same symbol). Let $t$ be a local parameter at $O$. The set of local paramters $\{ u_Q = \tau_{-Q}(t) \ | \ Q\in \Etproj\}$ is a coherent set of local parameters.

%

\begin{define}\label{def:residue}Let $u_Q$ be a local parameter at a point $Q \in \Etproj$ and let $v_Q$ be the valuation corresponding to the valuation ring at $Q$.  If $f \in k(\Etproj)$ has a pole at $Q$ or order $n$, we may write

\[f = \frac{c_{Q,n}}{u_Q^n} + \ldots + \frac{c_{Q,2}}{u_Q^2} + \frac{c_{Q,1}}{u_Q} + \tilde{f}\]
where $v_Q(\tilde{f}) \geq 0$. We shall refer to $c_{Q,i}$ as the {\rm residue of order $i$ at $Q$.}\end{define}

 In the usual presentation of Riemann surfaces, one speaks of residues of meromorphic differential forms.  These do not depend on the local parameters whereas any discussion of a powerseries expansion of a function at a point does depend on the local parameter. Fixing a set of local parameters allows the notion of residue of order $i$ to be well defined. 

%

\vspace{.1in}

The following definition is similar to Definition 2.3 of \cite{Chen_Singer12}.

\begin{define} Let $f \in k(\Etproj)$ and $S =\{ u_Q \ | \ Q\in \Etproj\}$ be a coherent set of local parameters and $Q\in \Etproj$. For each $j \in \NX_{>0}$ we define the {\rm orbit residue of order $j$ at $Q$} to be
\[ \ores_{Q,j}(f) = \sum_{i \in \ZX} c_{Q\oplus iP, j.}\]
\end{define}

Note that if  $Q' = Q \sim P$, then $ \ores_{Q',j}(f) =  \ores_{Q,j}(f)$ for any $j \in \NX_{>0}$. Furthermore $\ores_{Q,j}(f) = \ores_{Q,j}(\tau(f))$. The following refines   Proposition B.8 in \cite{DHRS} and is the reason for defining the orbit residue.

\begin{prop}\label{Prop2} Let $b \in k(\Etproj)$ and $S =\{ u_Q \ | \ Q\in \Etproj\}$ be a coherent set of local parameters.  The following are equivalent.
\begin{enumerate}
\item There exists  $g\in k(\Etproj)$  such that
\[b = \tau(g) - g.\]
\item For any $Q \in \Etproj$ and $j \in \NX_{>0}$
\[ \ores_{Q,j}(b) = 0.\]
\end{enumerate}
\end{prop}
\begin{proof} Proposition B.8 in \cite{DHRS} implies that $(2)$ is equivalent to {\it There exists $Q\in \Etproj$,  $h \in \calL(Q + \tau(Q))$ and $g \in \Etproj$. such that
$b = \tau(g) - g + h.$}. Lemma~\ref{cor:genus1b} implies that this latter condition is equivalent to $(1)$. \end{proof}

When applying Proposition~\ref{Prop2} we would like to verify the second condition using the fact that on a compact Riemann surface one has that the sum of the residues of a differential form is zero.  Denoting by ${\rm Res}_Q\omega$ the usual residue at a point $Q$  of a differential form $\omega$, we want to compare ${\rm Res}_Q(f\omega)$ with $c_{Q,1}$ where $f$ is as in Definition~\ref{def:residue}. To do this we need to make a more careful selection of a coherent family of local parameters. For this we will use the following lemma whose proof is similar to \cite[Theorem 14, p.~127]{chevalley}.

\begin{lem}\label{lem:chevlem} Let $C$ be a nonsingular curve and $K = \CX(C)$ its function field.
 Given a point $Q \in C$, a differential form $\omega$ regular  and nonzero at $Q$,  and  integer $n \in \NX$, there exists a local parameter $t_n \in K$ at $Q$ such $\omega = (1+f)dt_n$ where $v_Q(f)>n$. 
 \end{lem}
 \begin{proof} Let $t \in K$ be any local parameter at $Q$ and let 
 \[ \omega = (a_0 + a_1t + \ldots +a_nt^n + f_n) dt.\]
 where $f_n \in K$ and  $v_Q(f_n) > n$. Let 
 \[t_n = a_0t + \frac{a_1}{2} t^2 + \ldots + \frac{a_n}{n+1} t^{n+1}.\]
 We then have that
 \[ \omega - dt_n = (a_0 + a_1t + \ldots +a_nt^n + f_n - \frac{dt_n}{dt} dt =f_n dt.\]\end{proof}
   Let $\Omega$ be a fixed regular differential form on $\Etproj$. The maps $\ita, \itb, \tau = \itb\ita$ induce maps $\ita^*, \itb^*, \tau^* $ on the space of differential forms. From \cite[Lemma 2.5.1 and Proposition 2.5.2]{DuistQRT}, we have that $\iota^*_i(\Omega) = -\Omega$ for $i=1,2$ and $\tau^*(\Omega) = \Omega$. 
 
 \begin{defn} Let $n \in \NX$. We say that a coherent set $\{u_Q \ | \ Q\in \Etproj\}$ of local parameters is {\rm n-coherent} if for each $Q \in \Etproj$, $\Omega = (1 + f_Q)du_Q$ where $v_Q(f_Q) > n$.\end{defn}
There always exists an $n$-coherent set of local parameters. To see this one modifies    the construction following Definition~\ref{def:coherentlocparam} by starting  with a local parameter $t_n$ at $O$ satisfying the conclusion of Lemma~\ref{lem:chevlem} with respect to $\Omega$, that is, the order of $\Omega - dt_n$ at $O$ is  greater than $n$.\\

\noindent \emph{{\bf Fixed Assumption:} Through the  paper,  we assume that when the kernel curve $\Etproj$ is of genus one, we  fix  a $3$-coherent set of local parameters ${\{u_{Q} \ | \ Q \in \Etproj \}}$. The various elements  that we consider will have poles of order at most $3$ so we can always apply  Lemmas~\ref{resthm} and \ref{lem:horesidues}.}\\

Having an  $n$-coherent set of local parameters allows us to use the usual Residue Theorem. 
\begin{lemma}\label{resthm} Let $h\in\CX(\Etproj)$ and assume that $h$ has poles of order at most $n$ at any point of $\Etproj$. If $\{u_Q\}$ is an $n$-coherent set of local parameters, then  for each $Q \in \Etproj$, ${\rm Res}_{Q}(b\Omega)=c_{Q,1}$. Therefore, $\sum_{Q \in \Etproj} c_{Q,1}=0$.
\end{lemma}
 \begin{proof} Since $\Omega = (1 + f_Q) du_Q$ with $v_Q(f_Q) > n$ we have 
\[b\Omega = (\frac{c_{Q,n}}{u_Q^n} + \ldots + \frac{c_{Q,2}}{u_Q^2} + \frac{c_{Q,1}}{u_Q} + \tilde{f_Q}) du_Q\]
where $v_Q(\tilde{f_Q}) > 0$.  One now applies the usual Residue Theorem.\end{proof}

 \begin{rmk} 1. In \cite{DHRS19}, the authors introduced the notion of a coherent set of analytic local parameters and showed that such a set exists  on the universal cover of $\Etproj$ and using these to induce such a set on $\Etproj$. Alternatively,  one can always find a coherent set of local parameters $\{u_Q \ | \ Q\in \Etproj\}$ such that for each $Q$, $\Omega = du_Q$. One does this in the following way. If $t$ is an analytic local parameter at $O$, we write $\Omega = \sum_{i = 0}^\infty a_i t^i dt$, $a_0 \neq 0$. The analytic function $u_0 = \sum_{i = 0}^\infty \frac{a_i}{i+1} t^{i+1}$ is an analytic local parameter at $0$ and one can propagate this to become a coherent local family as above. Nonetheless, the $u_Q$ gotten in this way need not be in the function field of the curvve since they are only defined locally. We introduce the notion of $n$-coherence to be able to stay in the algebraic setting.   
 
 2.~{In \cite{dreyfus2019length}, the authors uniformize the kernel curve $E$ as a Tate curve, that is, as $C^*/q^\Z$ where $C$ is an algebraically closed field extension of $\Q(t)$. In that setting, the field $C(E)$ corresponds 
to the field $\mathcal{M} er (C^*)$ of meromorphic function over $C^*$ fixed by the automorphism $f(z) \mapsto f(qz)$ of $\mathcal{M} er (C^*)$. The first involution corresponds to $f(z) \mapsto f(1/z)$ and the automorphism $\tau$ to 
$f(z) \mapsto f(\tilde{q}z)$. The regular differential form on $C^*/q^\Z$ is $\frac{dz}{z}$ and the coherent set of local parameters given by the $ u_{\overline{\alpha}}: \overline{z } \mapsto ln(\frac{z}{\alpha})$ for $z$ close to $\alpha$ satisfies all the required properties.}\end{rmk}

 The following summarizes useful properties of the $c_{Q,i}$ and the $ \ores_{Q,j}(f)$ .

 \begin{lemma}\label{lem:horesidues} Let $n>1$ and $\{u_Q\}$ be an $n$-coherent set of local parameters. Assume $b \in \CX(\Etproj)$ satisfy $\ita(b) = -b$. \\
1. For each $Q \in \Etproj$,  $\ita(u_Q) = -u_{\ita(Q)} +g_{\ita(Q)}$ where $v_{\ita(Q)}(g_{\ita(Q)}) > n+1$.\\
2.  If  
\begin{align} \label{eq:laurentb}
b =& \frac{c_{Q,n}}{u_Q^n} + \ldots + \frac{c_{Q,2}}{u_Q^2} + \frac{c_{Q,1}}{u_Q} + \tilde{f}
\end{align}
where $v_Q(\tilde{f}) \geq 0$,  then
\[
b  =\frac{c_{\iota_{1}(Q),n}}{u_{\iota_{1}(Q)}^n} + \ldots + \frac{c_{\iota_{1}(Q),2}}{u_{\iota_{1}(Q)}^2} + \frac{c_{\iota_{1}(Q),1}}{u_{\iota_{1}(Q)}} + \tilde{g}
\]
where $v_{\iota_{1}(Q)}(\tilde{g}) \geq 0$ and  $c_{\iota_{1}(Q),j}=(-1)^{j+1}c_{Q,j}$ for any $j$. If follows that, if all the poles of $b$ belong to the same $\tau$-orbit, then, for any even number $j$, we have $\ores_{Q,j}(b)=0$.\\

\end{lemma}
\begin{proof} 1.  We have $\Omega = (1+f_Q) du_Q = (1+f_{\ita(Q)})d(u_{\ita(Q)})$ where $v_Q(f_Q) >n, v_{\ita(Q)}(f_{\ita(Q)}) >n$. Applying $\ita^*$ to the first equality we have 
\[-\Omega = \ita^*(\Omega) = (1 + \ita(f_Q))\ita^*(du_Q) = (1 + \ita(f_Q))d(\ita(u_Q)).\]

Since $\ita(u_Q)$ is again a local parameter at $\ita(Q)$ we have $\ita(u_Q) = cu_{\ita(Q)} + g_{\ita(Q)}$ where  $c \neq 0$ and $v_{\ita(Q)}(g_{\ita(Q)}) > 1$. Therefore 
\[ d(\ita(u_Q)) = (c+ \frac{d g_{\ita(Q)}}{du_{\ita(Q)}})du_{\ita(Q)}\]
and
\[-\Omega = (-1 - f_{\ita(Q)}) du_{\ita(Q)}  = (1 +  \ita(f_Q))(c +\frac{d g_{\ita(Q)}}{du_{\ita(Q)}})du_{\ita(Q)}.\]
Expanding the final product, one sees that $c = -1 $ and $v_{\ita(Q)}(g_{\ita(Q)}) > n+1$.\\[0.1in]
2. This statement and proof are similar to \cite[Lemma C.1]{DHRS}. Applying $\ita$ to \eqref{eq:laurentb}, we have 
)
\begin{align*}
-b = \ita(b) &=  \frac{c_{Q,n}}{\ita(u_Q)^n} + \ldots + \frac{c_{Q,2}}{\ita(u_Q)^2} + \frac{c_{Q,1}}{\ita(u_Q)} + \ita(\tilde{f})\\
    & = \frac{(-1)^nc_{Q,n}}{u_{\ita(Q)}^n}(1+g_n) + \ldots + \frac{(-1)^2c_{Q,2}}{u_{\ita(Q)}^2}(1+g_2) + \frac{(-1)^1c_{Q,1}}{u_{\ita(Q)}}(1+g_1) + \ita(\tilde{f})
\end{align*}
where $v_{\ita(Q)}(g_\ell) > n, n\geq \ell \geq 1$.  This follows from the fact that $\ita(u_Q) = u_{\ita(Q)} +g_{\ita(Q)}, v_{\ita(Q)}(g_{\ita(Q)}) >n+1$ and so $\ita(u_Q)^{-\ell} = (-1)^\ell u_{\ita(Q)}^{-\ell} (1 + g_\ell)$ for some $g_\ell$ with $v_{\ita(Q)}(g_\ell) > n$. Equating negative powers of $u_{\ita(Q)}$ yield the result.\end{proof}

\end{document}